\documentclass[11pt,reqno]{amsart}
\usepackage[a4paper,left=0.9in,right=0.9in,top=0.9in,bottom=0.9in]{geometry}
\usepackage{amsmath,amssymb,mathtools,lmodern}
\usepackage[colorlinks=true,linkcolor=blue,citecolor=blue,urlcolor=blue]{hyperref}
\usepackage{enumitem}
\usepackage{needspace}
\usepackage{tikz}
\usepackage{bm}
\usetikzlibrary{arrows.meta,calc,decorations.pathmorphing,decorations.pathreplacing}
\definecolor{abBlue}{HTML}{255A96}
\definecolor{abTeal}{HTML}{28796C}
\definecolor{abPurple}{HTML}{775394}
\definecolor{abOrange}{HTML}{BA6C2E}
\definecolor{abRed}{HTML}{B24C4C}
\definecolor{abGrey}{HTML}{878B91}
\tikzset{
 figedge/.style={line width=.85pt,-{Stealth[length=1.6mm,width=1.1mm]},shorten <=1.6pt,shorten >=1.6pt},
 figpath/.style={line width=1.1pt,-{Stealth[length=1.6mm,width=1.1mm]},shorten <=1.6pt,shorten >=1.6pt},
 figvertex/.style={circle,fill=black,inner sep=0pt,minimum size=3.4pt}
}

\usepackage{setspace}
\setlist{nosep}
\newtheorem{theorem}{Theorem}[section]
\newtheorem{proposition}[theorem]{Proposition}
\newtheorem{lemma}[theorem]{Lemma}
\newtheorem{corollary}[theorem]{Corollary}
\newtheorem{conjecture}[theorem]{Conjecture}

\theoremstyle{definition}
\newtheorem{definition}[theorem]{Definition}
\newtheorem{remark}[theorem]{Remark}
\newtheorem{observation}[theorem]{Observation}
\newtheorem*{setting}{Setup}

\theoremstyle{plain}
\newtheorem{claim}[theorem]{Claim}
\newenvironment{claimproof}
  {%
   \begin{proof}[Proof of the claim]}
  {\end{proof}}

\newcommand{\Prb}{\mathbb P}
\newcommand{\Ex}{\mathbb E}
\newcommand{\F}{\mathbb F}

\DeclareMathOperator{\ord}{ord}

\DeclareMathOperator{\Col}{Col}

\newcommand{\CE}{\mathcal E}

\title{Anti-concentration and Graham's rearrangement problem in finite abelian groups}

\author{Hoi H. Nguyen}
\address{Ohio State University}
\email{nguyen.1261@math.osu.edu}

\author{Tuan Tran}
\address{University of Science and Technology of China}
\email{trantuan@ustc.edu.cn}

\author{Van H. Vu}
\address{University of Hong Kong}
\email{vanvu@hku.hk}

\begin{document}

\begin{abstract}
Let $S$ be a subset of $n$ elements  of a  group $G$. 
An ordering $x_1,\ldots,x_n$ of  its elements is \emph{valid} if the partial sums
$x_1,x_1+x_2,\ldots,x_1+\cdots+x_n$
are pairwise distinct. This notion has a long history in combinatorial group theory, dated back to 
Gordon's works in early 1960. 

Alspach-Liversidge conjecture asserts that any subset of the nonzero elements of a finite abelian group  $G$ has a valid ordering. This 
generalizes Graham's conjecture, which considers the case when the group $G$ is $\F_p$, the cyclic group of 
prime order. 

Graham's conjecture has been solved recently, as a combination of results by 
Kravitz \cite{Kravitz24}, Bedert and Kravitz \cite{BK26}, M\"uyesser and Pokrovskiy~\cite{MP25},  Bedert, Buci\'c, Kravitz, Montgomery, and
M\"uyesser~\cite{BBKMM}, and  Pham and Sauermann \cite{PS26}.  Beside this, the only other case where
Alspach-Liversidge conjecture has been confirmed is $G= \F_2^k$, thanks to a recent result by Bedert, Buci\'c, Kravitz, Montgomery, and
M\"uyesser~\cite{BBKMM}.

We prove the Alspach-Liversidge conjecture in its full generality.
Our proof combines ideas from many previous papers. In particular, we will use the probabilistic approach from \cite{PS26} and the 
absorption method following \cite{MP25,BBKMM}. Our new main technical ingredient is a Fourier-moment estimate, which is of independent interest. 
In particular,  we can use it to show that the distribution of certain kinds of random sums in an abelian group has an almost uniform distribution. 

\end{abstract}

\maketitle
\pagestyle{plain}
\enlargethispage{3pt}

\section{Introduction}

Let $S$ be a subset of $n$ elements  of a finite abelian group $G$. 
An ordering $x_1,\ldots,x_n$ of  its elements is \emph{valid} if the partial sums
$x_1,x_1+x_2,\ldots,x_1+\cdots+x_n$
are pairwise distinct. 

The notion of  valid orderings has a long history.
In 1961, motivated by the study of Latin squares,
Gordon~\cite{Gordon61} considered the case $S=G$
and asked which group $G$ admits a valid ordering.
He showed that a nontrivial finite abelian group $G$
has a valid ordering if and only if
$\sum_{a\in G}a\ne0$.
For the nonabelian case, Keedwell~\cite{Keedwell81}
conjectured in 1981 that every sufficiently large
nonabelian group has a valid ordering.
His conjecture was confirmed recently (2025)  by
M\"uyesser and Pokrovskiy~\cite{MP25}.

The problem becomes much harder when $S$ is an
arbitrary subset of $G$. In 1971,
Graham~\cite{Graham71}, motivated by the study of
subset sums (and also jugling), posed the following conjecture
for cyclic groups of prime order.

\begin{conjecture}\label{Graham}
For every prime $p$, every subset
$S\subseteq\F_p\setminus\{0\}$ has a valid ordering.
\end{conjecture}

Archdeacon, Dinitz, Mattern, and Stinson~\cite{ADMS16}
generalized the conjecture to arbitrary cyclic groups.
Alspach conjectured that, under the additional
assumption $\Sigma(S)\ne0$, the partial sums can
be made both distinct and nonzero;
see~\cite{BH05,CFOR}.  Various versions of Alspach's conjecture have been reiterated and stated, for example, for general Abelian groups by Costa, Morini, Pasotti and Pellegrini \cite{CMP18}, 
and by Costa, Della Fiore, Ollis and Rovner-Frydman \cite{CFOR}. 
Below we state a version of Graham's conjecture for general Abelian groups, which  appeared in a paper by  Alspach and Liversidge  \cite{AL20}.

\begin{conjecture}\label{Alspach}
For any finite abelian group $G$, every subset
$S\subseteq G\setminus\{0\}$ has a valid ordering.
\end{conjecture}

In recent years, there have been several important progresses concerning Conjectures \ref{Graham} and \ref{Alspach}. 
For Graham's conjecture, Kravitz~\cite{Kravitz24} proved
the result for $n\le \log p/\log\log p$.
Bedert and Kravitz~\cite{BK26} extended this range to
$n\le \exp((\log p)^{1/4})$ for any sufficiently large $p$.
At the other end, M\"uyesser and Pokrovskiy~\cite{MP25}
established the result for sets of size $(1-o(1))p$. Bedert, Buci\'c, Kravitz, Montgomery, and
M\"uyesser~\cite{BBKMM} subsequently proved it whenever
$n\ge p^{1-c}$, for an absolute constant $c>0$.
Early this year,  Pham and Sauermann~\cite{PS26} showed that,
for every $0<\alpha<1$, there is a constant $C_\alpha$
such that the result holds for
$C_\alpha\le n\le p^{1-\alpha}$.
Together, these results settle Graham's conjecture
for all sufficiently large primes.

\begin{theorem}  \label{Grahamtheorem} 
There is a constant $C$ such that the following holds. For all $p \ge C$, any set  $S \subseteq \F_p \setminus \{0\} $
has a valid ordering.
\end{theorem}

For general abelian groups, Alspach and
Liversidge~\cite{AL20} proved the conjecture for sets of size at most nine.
Buci\'c, Frederickson, M\"uyesser, Pokrovskiy, and Yepremyan~\cite{BFMPY} proved an approximate version: every $n$-element subset $S$ of an arbitrary group
has an ordering with at least $n-o(n)$ distinct partial products.
In 2025, Bedert, Buci\'c, Kravitz, Montgomery, and
M\"uyesser~\cite{BBKMM} proved Alspach's conjecture
for sufficiently large subsets of $\F_2^k$.
In the same paper, they also established the conjecture
for sets of size at least $|G|^{1-c}$, for an absolute
constant $0<c<1$; this result also holds  for
nonabelian groups.

\begin{theorem}  \label{F2theorem}   
There is an absolute constant $C$ such that  the following holds. For all $k$, any set  $S \subseteq \F_2^k \setminus \{ 0\} $
of size at least $C$ has a valid ordering.
\end{theorem}

The authors also mentioned that it is possible to extend their method to handle the case when $G= \F_q^k$, with $q $ fixed, 
$k \rightarrow \infty$, and $n \ge C(q)$, some constant depending on $q$.  The proof was, however, not included for the sake of a cleaner presentation.

Thus, so far, the only two cases where  Alspach-Leversidge conjecture  has been  solved is when $G= \F_p$ and $G= \F_2^k$
(and perhaps also $G =\F_q^k$ for $q$ fixed and $k \rightarrow \infty$, as mentioned above). 

\vskip2mm 


We prove the Alspach-Liversidge conjecture for all sufficiently large subsets of arbitrary finite abelian groups. As in Theorem~\ref{F2theorem}, the required lower bound on $n$ is an absolute constant, independent of the ambient group $G$.

\begin{theorem}\label{thm:Graham}
There is an absolute constant $C$ such that the following holds.  For every
finite abelian group $G$, any set 
$S \subseteq G\setminus\{0\}$ of size at least $C$
has a valid ordering.
\end{theorem}

Notice in our theorem, the constant $C$ is independent of the ambient group $G$. Thus, when applied to $G= \F_q^k$, our theorem is stronger than the claim  following Theorem 
\ref{F2theorem}, which requires $C$ to depend on $q$.

Our proof adapts the probabilistic local-repair method of Pham and Sauermann~\cite{PS26} to finite abelian groups through a structural reduction. The key new analytic ingredient is the Fourier-moment estimate of Lemma~\ref{lem:Fourier-moment}. Combined with Roos' permanent inequality~\cite{Roos}, it yields joint anti-concentration estimates arising in this reduction. 

The remaining case occurs when almost all of $S$ lies in a subgroup in which it has sufficiently large density. We handle this case by adapting the absorption method of Bedert et al.~\cite{BBKMM}, incorporating the exceptional elements outside the subgroup. For a detailed overview of the proof, see Section~\ref{sub:overview}.

We used ChatGPT to assist with the adaptation of the absorption argument, to help generate diagrams, and to improve the clarity and flow of the exposition. We take full responsibility for the mathematical content of the paper.

\medskip
\noindent\textbf{Notation and organization.}
Unless stated otherwise, all groups are finite abelian
and written additively.
We write $H\le G$ for a subgroup, $a+H$ for a coset,
$X\sqcup Y$ for a disjoint union, and
$[n]=\{1,\ldots,n\}$.
Logarithms are natural, and constants are absolute
unless their dependence is specified.

Section~\ref{sec:Fourier-moment} proves the Fourier-moment estimate.
Section~\ref{sec:subset-sums} establishes the structural
dichotomy and the prefix and nested-chain estimates.
Section~\ref{sec:repair-full} develops local repair,
while Section~\ref{sec:external-absorption} adapts
the path and absorption constructions to dense
subsets of subgroups with exceptional elements.
Section~\ref{sec:main-proof} combines these ingredients
to prove Theorem~\ref{thm:Graham}.

\section{A Fourier moment estimate and the distribution of random subset sums}
\label{sec:Fourier-moment}

Let $G$ be a finite abelian group. 
For a probability measure $\mu$ on $G$, define its
\emph{collision probability} by
$\Col(\mu)=\sum_{x\in G}\mu(x)^2$.
For $0<\alpha<1$, we say that $\mu$ is
\emph{$\alpha$-balanced} if every coset of every proper
subgroup of $G$ has mass at most $1-\alpha$.
A subset $S\subseteq G$ is $\alpha$-balanced if its
uniform probability measure is $\alpha$-balanced.

For a function $f:G\to\mathbb C$, define its Fourier coefficient by
$\widehat f(\chi)=\sum_{x\in G}f(x)\overline{\chi(x)}$.
Character orthogonality,
$\sum_{\chi\in\widehat G}\chi(x)=|G|\mathbf1_{\{x=0\}}$,
gives the Fourier inversion formula 
$f(z)=|G|^{-1}\sum_{\chi\in\widehat G}\widehat f(\chi)\chi(z)$
and Parseval's identity
$|G|^{-1}\sum_{\chi\in\widehat G}|\widehat f(\chi)|^2
=\sum_{x\in G}|f(x)|^2$.
For a probability measure $\mu$, the last sum equals
$\operatorname{Col}(\mu)$.

An \emph{additive character} is a map $\xi:G\to\mathbb R/\mathbb Z$
satisfying $\xi(x+y)=\xi(x)+\xi(y)$.
Every complex character has the form $\chi(x)=e^{2\pi i\xi(x)}$.
For $t\in\mathbb R/\mathbb Z$, write $\|t\|$ for the distance
of any real representative of $t$ to the nearest integer.

Our key new technical ingredient is the following lemma, which bounds all nontrivial Fourier moments in terms of
balancedness and collision probability.

\begin{lemma}\label{lem:Fourier-moment}
Let $\mu$ be an $\alpha$-balanced probability measure on a finite
abelian group $G$, where $0<\alpha<1$. Then, for every real $q\ge2$,
\[
 \frac1{|G|}\sum_{\chi\ne1}|\widehat\mu(\chi)|^q
 \le \frac{20\operatorname{Col}(\mu)}{\sqrt{\alpha q}}.
\]
\end{lemma}

We came across this lemma in a study on the anti-concentration phenomenon, following 
our earlier works \cite{BBDV26,DNgPTV}.  A subject in these works is the random subset sums.
Consider a finite set $S$, and let $1 \le k \le |S|$ be an integer. Choose a random subset $R$ of $S$ 
of size $k$, uniformly among all subsets of $S$ of this size. 
We are interested in the distribution of  the random variable $\Sigma(R)= \sum_{ x \in R }  x $.

Berger, Berkowitz, Devlin, and Vu~\cite{BBDV26} proved that, if $S$ consists of real numbers, then for any real number $z$, 
\[
\Prb (\Sigma (R)= z)  = \frac{O(1)}{|S|\sqrt{\min\{k,|S|-k\}}}.
\]

In  \cite{PS26}, Pham and Sauerman treated the case when $S \subseteq \F_p$. 
Assuming 
$C\log|S|\le k\le10^{-3}|S|/\log|S|$,
for an absolute constant $C$, 
they proved that for any $z \in \F_p$, 

$$\Prb(\Sigma(R)=z)
 \le\frac{1}{p} +\frac{C}{|S|\sqrt{k}}. $$ This bound is a key technical component of their proof
 concerning Graham's conjecture; see the introduction of \cite{PS26}.

Using Lemma \ref{lem:Fourier-moment}, we can deduce the following result, which bounds the pointwise discrepancy between the distribution of $\Sigma(R)$ and the uniform distribution on $G$. Its proof will appear elsewhere; the result itself is not needed for the proof of Theorem~\ref{thm:Graham}.

\begin{theorem}\label{thm:k-sum}
Let $S$ be an $\alpha$-balanced subset of a finite abelian
group $G$, and let $R$ be a uniformly random $k$-element
subset of $S$, where $1\le k<|S|$.
Then, for every $z\in G$,
\[
 \Bigl|\Prb(\Sigma(R)=z)-\frac1{|G|}\Bigr|
 \le\frac{2^{16}}
 {|S|\sqrt{\alpha\min\{k,|S|-k\}}}.
\]
\end{theorem}

Since every subset of $\F_p$ with at least two elements
is $1/2$-balanced, Theorem~\ref{thm:k-sum} sharpens and extends
Pham-Sauerman result considerably. 
We have a two-sided pointwise discrepancy estimate, which holds for every $1\le k<|S|$, and for any finite abelian group. 

\subsection{Proof of Lemma \ref{lem:Fourier-moment}}
We use the following classical theorem~\cite{Kneser53} in the character group:
when a sumset has trivial stabilizer, it must grow by at least
the size of the added set minus one.

\begin{theorem}[Kneser's addition theorem]
\label{thm:analytic-kneser}
Let $A,B$ be nonempty subsets of a finite abelian group $G$, and let
$K=\{x\in G:A+B+x=A+B\}$ be the stabilizer of their sumset. Then
$|A+B|\ge |A+K|+|B+K|-|K|$.
\end{theorem}

\begin{proof}[Proof of Lemma \ref{lem:Fourier-moment}]
Since $|\widehat\mu(\chi)|\le1$, Parseval's identity gives
\[
|G|^{-1}\sum_{\chi\ne1}|\widehat\mu(\chi)|^q
\le |G|^{-1}\sum_{\chi\ne1}|\widehat\mu(\chi)|^2\le\operatorname{Col}(\mu)
\]
for $2\le q<4$.
This is stronger than the desired bound, so assume $q\ge4$.

We estimate the moment by counting Fourier coefficients close to $1$.
For $0\le\theta\le\pi/2$, put
$T(\theta)=\{\chi\in\widehat G:|\widehat\mu(\chi)|\ge\cos\theta\}$
and $N(\theta)=|T(\theta)|-1$.
Since $\widehat\mu(1)=1$, $N(\theta)$ counts the nontrivial characters
in $T(\theta)$. Our key estimate is
\begin{equation}\label{eq:analytic-small-angle}
 N(\theta)\le
 \frac{8|G|\operatorname{Col}(\mu)}{\sqrt\alpha}\,\theta
 \quad \text{for } 0\le\theta\le\sqrt\alpha.
\end{equation}
We prove it using a closure property of the level sets and Kneser's theorem.

First, we claim that
$T(\theta_1)T(\theta_2)\subseteq T(\theta_1+\theta_2)$
whenever $\theta_1+\theta_2\le\pi/2$, where characters are multiplied
pointwise. Take $\chi_i\in T(\theta_i)$, and let $X$ have distribution
$\mu$. Put $U=\overline{\chi_1(X)}$ and $V=\overline{\chi_2(X)}$.
Since $|U|=|V|=1$, we have
$\mathbb E|U-\mathbb E U|^2=1-|\mathbb E U|^2$ and
$\mathbb E|V-\mathbb E V|^2=1-|\mathbb E V|^2$.
Cauchy--Schwarz therefore gives
\[
 \bigl|\widehat\mu(\chi_1\chi_2)-(\mathbb E U)(\mathbb E V)\bigr|
 =\bigl|\mathbb E[(U-\mathbb E U)(V-\mathbb E V)]\bigr|\le\sqrt{(1-|\mathbb E U|^2)(1-|\mathbb E V|^2)}.
\]
Using $|\mathbb E U|\ge\cos\theta_1$ and $|\mathbb E V|\ge\cos\theta_2$, we obtain
\[
 \begin{aligned}
|\widehat\mu(\chi_1\chi_2)|=|\mathbb E(UV)|
 &\ge |\mathbb E U||\mathbb E V|-\sqrt{(1-|\mathbb E U|^2)(1-|\mathbb E V|^2)}\\
 &\ge\cos\theta_1\cos\theta_2-\sin\theta_1\sin\theta_2
 =\cos(\theta_1+\theta_2),
 \end{aligned}
\]
which proves the claim.

Next, balancedness implies that $T(\sqrt\alpha)$ contains no
nontrivial subgroup of $\widehat G$.
Indeed, let $D\le\widehat G$ be nontrivial, and define its annihilator by
$D^\perp=\{x\in G:\chi(x)=1\text{ for every }\chi\in D\}$.
This is a proper subgroup of $G$.
Orthogonality on $D$ states that
$|D|^{-1}\sum_{\chi\in D}\chi(x)=\mathbf1_{\{x\in D^\perp\}}$.
Thus, for independent samples $X,Y$ from $\mu$,
\[
 \frac1{|D|}\sum_{\chi\in D}|\widehat\mu(\chi)|^2
 =\Pr(X-Y\in D^\perp)
 =\sum_{Q\in G/D^\perp}\mu(Q)^2.
\]
By balancedness, the last sum is at most
$(1-\alpha)\sum_Q\mu(Q)=1-\alpha$.
Consequently, some nontrivial $\chi\in D$ satisfies
$|\widehat\mu(\chi)|\le\sqrt{1-\alpha}$.
Since $\sin\sqrt\alpha<\sqrt\alpha$, we have
$\sqrt{1-\alpha}<\cos\sqrt\alpha$.
Hence $D$ cannot be contained in $T(\sqrt\alpha)$.

Now fix $0<\theta\le\sqrt\alpha$, and put
$m=\lfloor\sqrt\alpha/\theta\rfloor$.
For $1\le v\le m$, repeated use of angular closure gives
$T(\theta)^v\subseteq T(v\theta)\subseteq T(\sqrt\alpha)$.
Each $T(\theta)^v$ contains $1$.
Its stabilizer
$P_v=\{\psi\in\widehat G:\psi T(\theta)^v=T(\theta)^v\}$
is a subgroup, and $P_v\subseteq T(\theta)^v$ because $1\in T(\theta)^v$.
The preceding paragraph therefore forces $P_v=\{1\}$.

Since $T(\theta)^v=T(\theta)^{v-1}T(\theta)$ has trivial stabilizer, Kneser's theorem
gives 
\[
|T(\theta)^v|\ge|T(\theta)^{v-1}|+|T(\theta)|-1,
\]
for $2\le v\le m$.
Starting from $|T(\theta)|=N(\theta)+1$, induction yields
$|T(\theta)^m|\ge mN(\theta)+1$.
Since $T(\theta)^m\subseteq T(m\theta)$ and $m\theta<1$, every
$\chi\in T(\theta)^m$ satisfies $|\widehat\mu(\chi)|\ge\cos(m\theta)>1/2$.
Hence Parseval gives
\[
 |T(\theta)^m|\le4\sum_{\chi\in T(\theta)^m}|\widehat\mu(\chi)|^2
 \le4|G|\operatorname{Col}(\mu).
\]
Combining these bounds and using
$m\ge\sqrt\alpha/(2\theta)$, we obtain
\[
 N(\theta)\le\frac{4|G|\operatorname{Col}(\mu)}m
 \le\frac{8|G|\operatorname{Col}(\mu)}{\sqrt\alpha}\,\theta.
\]
This proves \eqref{eq:analytic-small-angle}.

We now turn the level-set estimate into a moment bound.
Note that, for $a\in[0,1]$, we have
$a^q=q\int_{\arccos a}^{\pi/2}
\cos^{q-1}\theta\sin\theta\,d\theta$.
Applying this identity to each coefficient and summing yield
\begin{equation}\label{eq:layercake}
 \frac1{|G|}\sum_{\chi\ne1}|\widehat\mu(\chi)|^q
 =\frac q{|G|}\int_0^{\pi/2}
 \cos^{q-1}\theta\sin\theta\,N(\theta)\,d\theta.
\end{equation}
We split this integral at $\sqrt\alpha$.

On $[0,\sqrt\alpha]$, use \eqref{eq:analytic-small-angle},
$\sin\theta\le\theta$, and $\cos\theta\le e^{-\theta^2/2}$ to obtain
\[
 \begin{aligned}
 \frac q{|G|}\int_0^{\sqrt\alpha}
 \cos^{q-1}\theta\sin\theta\,N(\theta)\,d\theta
 &\le\frac{8q\operatorname{Col}(\mu)}{\sqrt\alpha}
 \int_0^\infty\theta^2e^{-(q-1)\theta^2/2}\,d\theta\\
 &=\frac{4\sqrt{2\pi}\,q\operatorname{Col}(\mu)}
 {\sqrt\alpha(q-1)^{3/2}}\le\frac{16\operatorname{Col}(\mu)}{\sqrt{\alpha q}}.
 \end{aligned}
\]

For $\theta<\pi/2$, Parseval also gives the crude bound
$N(\theta)\cos^2\theta\le|G|\operatorname{Col}(\mu)$.
Hence the remaining contribution is at most
\[
 \begin{aligned}
 \frac q{|G|}\int_{\sqrt\alpha}^{\pi/2}
 \cos^{q-1}\theta\sin\theta\,N(\theta)\,d\theta
 &\le q\operatorname{Col}(\mu)
 \int_{\sqrt\alpha}^{\pi/2}
 \cos^{q-3}\theta\sin\theta\,d\theta\\
 &=\frac q{q-2}\operatorname{Col}(\mu)
 \cos^{q-2}\sqrt\alpha\le\frac{2\operatorname{Col}(\mu)}{\sqrt{\alpha q}}.
 \end{aligned}
\]
Adding the two contributions proves the lemma.
\end{proof}

\section{Nested subset sums}\label{sec:subset-sums}\label{sec:nested}

In this section, we develop the structural and probabilistic
estimates needed for the proof of Theorem~\ref{thm:Graham}. We begin
by introducing the relevant notation and giving an overview
of the proof, before turning to the technical details.

\begin{definition}\label{def:prefixes}
For an ordering $x_1,\ldots,x_{|A|}$ of a finite set $A\subseteq G$,
its \emph{prefix sums} are $P_0=0$ and
$$P_\ell=\sum_{j=1}^\ell x_j \mbox{ for }1\le\ell\le|A|.$$
For prescribed $1\le b_1<\cdots<b_k<|A|$, a \emph{uniform nested
chain} of these sizes in $A$ is a chain $R_1\subset\cdots\subset R_k\subsetneq A$ chosen uniformly among all such chains with $|R_i|=b_i$.
Equivalently, we take a uniform ordering and set $R_i=\{x_1,\ldots,x_{b_i}\}$ for each $i$.

For $k\ge1$ and $\boldsymbol z=(z_1,\ldots,z_k)\in G^k$, write
\[
 F_k(A;\boldsymbol z)=F_k(A;z_1,\ldots,z_k):=
 \sum_{1\le b_1<\cdots<b_k<|A|}
 \Prb(P_{b_i}=z_i\text{ for every }i).
\]
\end{definition}

Thus $F_k(A;z_1,\ldots,z_k)$ is the expected number
of index tuples at which the prefix sums take the
prescribed values $z_1,\dots, z_k$.

We also introduce a structural property that will
play an important role. 
\begin{definition}\label{def:fixed-size}
A finite set $P\subseteq G$
is called a \emph{reservoir} if any two distinct
subsets of $P$ having the same cardinality have
different sums. In other words, $P$ has distinct
subset sums at each fixed cardinality.
\end{definition}

\subsection{Overview of the proof of Theorem~\ref{thm:Graham}.}\label{sub:overview} Let $S\subseteq G\setminus\{0\}$ be sufficiently
large, and write $n=|S|$. Our proof combines
probabilistic anti-concentration estimates with
two methods for constructing valid orderings, local-repair method (Section \ref{sec:repair-full}) and absorption method (Section \ref{sec:external-absorption}).

The first method, which handles most cases, is due to Pham and Sauermann~\cite{PS26}.
We begin with a uniformly random ordering
$x_1,\ldots,x_n$ of $S$. Such an ordering may contain
zero-sum intervals, and our goal is to eliminate them
by locally modifying the ordering.

More precisely, whenever a zero-sum interval ends
at a position $b$, we seek a suitable nearby position
$y$ and interchange the elements at positions $b$
and $y$. The main difficulty is that eliminating one
zero-sum interval may create new ones. This is where anti-concentration enters. For a random
ordering, the event that a prescribed interval has
sum zero can be expressed as an equality between
two prefix sums. Moreover, the failure of a proposed
swap can be expressed through further equations
involving sums of selected elements. To control
these events simultaneously, we need estimates not
only for individual prefix sums, but also for several
prescribed prefix sums occurring together. The quantities $F_1$ and $F_k$ introduced above
provide precisely this information. The bound on
$F_1$ controls individual obstructions, while bounds
on $F_k$, for a fixed large $k$, prevent too many
obstructions from occurring simultaneously.
Proposition~\ref{prop:repair-full} shows that sufficiently strong
estimates for $F_1$ and $F_D$, where $D$ is a fixed
large integer, guarantee that the local-repair
procedure succeeds with high probability. 

A crucial point is that these estimates must
hold uniformly after deleting a bounded number of
elements. Indeed, when we consider a proposed swap,
we first reveal the elements involved in it and
then apply anti-concentration to the remaining
random ordering. The challenge is that an arbitrary set $S$ need not be balanced, so the Fourier estimates of
Section~\ref{sec:Fourier-moment} cannot always be applied directly.
We overcome this through a structural dichotomy: fixing small positive constants $c,\varepsilon$, Lemma~\ref{lem:filtration} shows that either $S$ contains a
reservoir $P$ with $|P|\ge n^\varepsilon$,
or there exist a subgroup $H\le G$ and
$a\in G$ such that
\[
T=S\cap(a+H),\quad |T|\ge(1-c)n,
\]
and $T-a$ is balanced in $H$, with a quantitative
balancedness parameter.

We now explain how the two alternatives lead
to the required estimates.
\vskip .1in
{\it \underline{The reservoir case}.} Suppose that $S$ contains a large reservoir $P$. Fix a positive integer $j$ and consider the first $j$ elements of a uniformly random ordering of $S$. These elements form a uniformly random $j$-subset $R$ of $S$. We wish to bound the probability that their sum equals a prescribed value $z\in G$.

To this end, we first condition on the selected elements outside $P$, namely $R\setminus P$. Their sum, say $s$, and their number are now fixed. Since $|R|=j$, the number of selected elements from $P$ is also fixed, and the condition $\Sigma(R)=z$ becomes
\[
\Sigma(R\cap P)=z-s.
\]
By the defining property of a reservoir, distinct subsets of $P$ of the same cardinality have different sums. Thus there is at most one possible choice of $R\cap P$ satisfying this equation. This uniqueness is the source of our anti-concentration estimate.

To make this observation quantitative, we use independent Bernoulli sampling: each element of $S$ is selected independently with probability $j/|S|$, and we condition on selecting exactly $j$ elements. The resulting set has the same distribution as $R$. Before conditioning on its cardinality, the probability of selecting any particular subset of $P$ is at most $\max\{j/|S|,1-j/|S|\}^{|P|}$. The uniqueness property therefore gives an exponentially small bound when $j$ is sufficiently far from $0$ and $|S|$. For prefixes near these endpoints, we use a simpler completion bound: after revealing all but one selected element, at most one of the remaining elements can complete the prescribed sum. Combining these two estimates and summing over all prefix lengths gives Lemma~\ref{lem:reservoir-prefixes}.

We next need to control the probability that several prefix sums simultaneously take prescribed values. Fix indices $1\le b_1<\cdots<b_k<|S|$ and divide the random ordering into $k+1$ consecutive blocks, with boundaries at these indices. Prescribing the prefix sums is equivalent to prescribing the sum of each block. However, these block sums are dependent. To overcome this difficulty, we choose a largest block and leave it unexposed. Since the sum of all elements of $S$ is fixed, the sum condition on this reserved block is automatically satisfied once the sums of the other blocks are prescribed. We then expose the remaining $k$ blocks successively, working inward from the two ends. At each step, the next block is a prefix or suffix of a uniformly random ordering of the elements not yet exposed. This allows us to apply the one-prefix estimate conditionally and repeatedly.

The advantage of reserving a largest block is that at least $|S|/(k+1)$ elements remain unexposed throughout the procedure. We also need sufficiently many reservoir elements among them, since otherwise the one-prefix estimate becomes ineffective. A sampling tail bound shows that the probability of having too few such elements is exponentially small in $|P|/(k+1)$. On the complementary event, each successive exposure contributes a factor controlled by the one-prefix estimate.

Finally, we sum over all possible choices of the prefix indices $b_1,\ldots,b_k$. This gives the joint estimate in Lemma~\ref{lem:reservoir-chains}, which is the key input for controlling several simultaneous obstructions in the local-repair argument.

In summary, Lemmas~\ref{lem:reservoir-prefixes} and~\ref{lem:reservoir-chains} provide the
bounds needed for local repair in the
reservoir case.
\vskip .1in
{\it \underline{The affine-core case}.} Suppose instead that most of $S$ lies in
$a+H$, as provided by the structural dichotomy.
We call $T=S\cap(a+H)$ the \emph{affine core}, and the elements of
$S\setminus T$ the \emph{outsiders}. Put $d=\operatorname{ord}(a+H)$ in the quotient group $G/H$.

We distinguish three regimes, depending on
$d$ and $|H|$.

\begin{enumerate}[label=(\arabic*),leftmargin=*]
\item (Large quotient order.) If $d\ge n^\varepsilon$, we exploit the
quotient group $G/H$. Conditional on the
relative order of the outsiders, their
contribution to the quotient sum is determined
once their number in the prefix is fixed.
Each core element contributes the same coset
$a+H$. Thus, for a prescribed quotient sum, the number
of core elements in the prefix must belong to
one residue class modulo $d$. A unimodality
argument bounds the expected number of such
prefixes, while the completion bound converts
this into an estimate for exact sums in $G$. The same idea, combined with a largest-block
argument, gives the corresponding joint estimate.
These are Propositions~\ref{aq:localtime} and~\ref{aq:chains}.
\vskip .1in

\item (Large subgroup.)
Suppose that $d<n^\varepsilon$ but
$|H|>n^{1+\varepsilon}$. Here the quotient
does not provide sufficient anti-concentration,
so we work inside $H$. Since $T-a$ is balanced in $H$, the Fourier
estimates of Section~\ref{sec:Fourier-moment} allow us to control
the discrepancy of random subset sums from
the uniform distribution on $H$. The natural
main term for hitting a prescribed element
is $1/|H|$, which becomes small after
summing over possible prefix lengths because
$|H|$ is much larger than $n$. For several prescribed prefix sums, we again
divide the ordering into blocks and reserve
a largest block. Lemma~\ref{lem:multivariate-partition} gives a joint
Fourier estimate for the sums of the other
large blocks, using Roos's permanent
inequality and the Fourier-moment estimate
of Lemma~\ref{lem:Fourier-moment}. Blocks containing no core elements, as well
as blocks containing only a few core elements,
are handled separately by counting and
completion bounds. Summing over all possible
block sizes then gives Proposition~\ref{prop:almost-coset-direct-chain},
which supplies the required estimates
for local repair.
\vskip .1in

\item (Dense subgroup). In the remaining case, the subgroup $K=H+\langle a\rangle$
contains $T$ and satisfies $|K|=d|H|\le n^{1+2\varepsilon}$, $\frac{|S\cap K|}{|K|}
\ge(1-c)n^{-2\varepsilon}$. Thus almost all elements of $S$ lie in a
subgroup in which they have relatively high
density. In this regime, the preceding
anti-concentration estimates are not
sufficiently strong for our local-repair
criterion. We instead adapt the path and absorption
method of Bedert et al.~\cite{BBKMM},
which was developed in connection with
Graham's problem over $\F_2^n$ and more
generally for dense subsets of finite
groups. We describe this approach in Section~5, culminating in Proposition~\ref{prop:localized-dense-core}.

\end{enumerate}

The rest of Section~\ref{sec:nested} establishes the
structural dichotomy and the required
prefix and nested-chain estimates.
Section~\ref{sec:repair-full} develops the local-repair
method and its probabilistic criterion. 
Section~\ref{sec:external-absorption} develops the dense-subgroup
construction, and Section~\ref{sec:main-proof} combines
these ingredients, with suitable
parameter choices, to prove
Theorem~\ref{thm:Graham}.

We now provide the details of the proof outlined above.

\subsection{Structure decomposition} We first establish the structural dichotomy.

\begin{lemma}\label{lem:filtration}
Let $S\subseteq G$, $0<c<1/2$, and let $1\le r\le|S|$ be an integer.
One of the following holds:
\begin{enumerate}[label=(\roman*)]
\item $S$ contains $r$ elements with distinct subset sums at each fixed size;
\item there are $H\le G$ and $a\in G$ such that
      $|S\cap(a+H)|\ge(1-c)|S|$ and $(S-a)\cap H$ is
      $\frac{c}{r+1}$-balanced in $H$.
\end{enumerate}
\end{lemma}

\begin{proof}
Put $\tau=c/(r+1)$. We will iteratively construct nested cosets
$C_j=a_j+H_j$ such that $S_j:=S\cap C_j$ has size at least
$(1-j\tau)|S|$, stopping when $S_j-a_j$ is $\tau$-balanced
in $H_j$ or after $r$ steps.

Start with $S_0=S$. For each nonempty $S_j$, choose $a_j\in S_j$
and set $H_j=\langle S_j-a_j\rangle$ and $C_j=a_j+H_j$,
the smallest coset containing $S_j$.
If $j<r$ and balancedness fails, choose $K<H_j$ and $b\in H_j$
such that $D:=a_j+b+K$ satisfies
$|S_j\cap D|>(1-\tau)|S_j|$.
Since $b\in H_j$ and $K<H_j$, we have $D\subsetneq C_j$.
Set $S_{j+1}=S_j\cap D$. Then
$|S_{j+1}|>(1-\tau)|S_j|\ge(1-(j+1)\tau)|S|>0$.
Minimality gives $S_{j+1}\subsetneq S_j$ and
$S_{j+1}\subseteq C_{j+1}\subseteq D\subsetneq C_j$.
These inclusions also give $S_{j+1}=S\cap C_{j+1}$,
completing the inductive step.

If we stop at a balanced translate, then
$|S_j|\ge(1-r\tau)|S|>(1-c)|S|$, giving (ii) with
$a=a_j$ and $H=H_j$.

Otherwise, choose $x_i\in S_{i-1}\setminus S_i$ for
$1\le i\le r$ and $a\in S_r$. Nestedness gives $C_i=a+H_i$, so
\[
 x_i-a\notin H_i,
 \quad x_j-a\in H_i\quad(j>i).
\]
We observe that an equality between distinct equal-size subset sums of
$\{x_1,\ldots,x_r\}$ remains valid after translating by $-a$.
Reducing modulo $H_i$, where $i$ is the least index occurring
in only one subset, forces $x_i-a\in H_i$, a contradiction.
This proves (i).
\end{proof}

\subsection{Reservoir improvement}  We begin with a simple observation, which is useful for short blocks.

\begin{observation}\label{obs:completion}
Let $R$ be a uniform $k$-subset of a finite set $A\subseteq G$,
where $1\le k <|A|$. For every $z\in G$,
\[
 \Prb(\Sigma(R)=z)
 \le\min\left\{\frac1{|A|-k+1},\frac1{k+1}\right\}
 \le\frac2{|A|}.
\]
\end{observation}

\begin{proof}
Order $A$ uniformly and reveal its first $k-1$ elements.
At most one of the $|A|-k+1$ equally likely remaining elements
completes the sum, giving the first bound.
Apply the same argument to $A\setminus R$ with target
$\Sigma(A)-z$ to obtain the second.
\end{proof}
As the reader can see, summing the bounds above over all prefix lengths gives only a
constant. A reservoir improves this to a bound decreasing with its size.

\begin{lemma}\label{lem:reservoir-prefixes}
Let $A\subseteq G$ contain a nonempty set $P$ with distinct subset
sums at each fixed size. Then, for every $z\in G$,
\[
 F_1(A;z)\le\frac{13\log(2|A|)}{|P|}.
\]
\end{lemma}

\begin{proof}
The assertion is immediate when $|A|=1$,
so assume $|A|\ge2$.
Let $P_\ell,1\le \ell \le n$ be the prefix sums of a uniform ordering of $A$,
as in Definition~\ref{def:prefixes}. 

Fix $1\le j<|A|$, and independently select each element of $A$ with probability $p:=j/|A|$,
and call the resulting set $R$. Conditional on every selection outside
$P$, the condition $|R|=j$ fixes the selected cardinality inside $P$.
At that cardinality, a prescribed total sum determines at most one
selected subset of $P$. Its probability is at most
$\max\{p,1-p\}^{|P|}$. Consequently,
\[
 \Prb\bigl(\Sigma(R)=z,\ |R|=j\bigr)\le\max\{p,1-p\}^{|P|}.
\]
The probability $\Prb(|R|=j)$ is the maximum of $\Prb(|R|=0),\ldots,\Prb(|R|=|A|)$.
These $|A|+1$ probabilities sum to $1$, and $j$ is the mode of this binomial distribution,
$\Prb(|R|=j)\ge \frac{1}{|A|+1}$. Conditional on $|R|=j$, the set $R$ is a uniform $j$-subset
of $A$, therefore,
\[
\begin{aligned}
\Prb(P_j=z)
=\frac{\Prb(\Sigma(R)=z,\ |R|=j)}
        {\Prb(|R|=j)}&\le (|A|+1)\max\{p,1-p\}^{|P|}\\
        &\le (|A|+1)e^{-\min\{j,|A|-j\}|P|/|A|}.
\end{aligned}
\]
On the other hand, Observation~\ref{obs:completion} gives
$\Prb(P_j=z)\le\frac2{|A|}$.
Put $f(x)=\min\{\frac{2}{|A|},(|A|+1)e^{-|P|x/|A|}\}$ for $x\ge0$.
This function is decreasing. Hence
\begin{align*}
 \sum_{j=1}^{|A|-1}\Prb(P_j=z)
 &\le2\sum_{j\ge1}f(j)\le2\int_0^\infty f(x)\,dx\le \frac4{|P|}\left(\log\frac{|A|(|A|+1)}2+1\right)
 \le\frac{13\log(2|A|)}{|P|}. \qedhere
\end{align*} 
\end{proof}

Several prescribed prefix sums are handled by leaving a largest
block unrevealed, called the \emph{reserved block}. To \emph{expose}
a block means to reveal its set of elements, but not their order.
Exposing the other blocks from the two ends keeps the remaining set
large, so the prefix estimate can be iterated whenever enough
reservoir elements remain.

\begin{lemma}\label{lem:reservoir-chains}
Let $A\subseteq G$ contain a nonempty set $P$ with distinct
subset sums at each fixed size. For every integer $k\ge1$
and all $z_1,\ldots,z_k\in G$,
\[
 F_k(A;z_1,\ldots,z_k)
 \le (k+1)\left(\frac{26(k+1)\log(2|A|)}{|P|}\right)^k
      +|A|^k e^{-|P|/8(k+1)}.
\]
\end{lemma}

\begin{proof}
If $|A|\le k$, then $F_k=0$, so assume
$|A|>k$. Let
$\mathcal I=\{(b_1,\ldots,b_k):1\le b_1<\cdots<b_k<|A|\}$,
and let $x_1,\ldots,x_{|A|}$ be a uniform ordering of $A$.
For each $\boldsymbol b\in\mathcal I$, put $b_0=0$,
$b_{k+1}=|A|$, $z_0=0$, and $z_{k+1}=\Sigma(A)$.
The blocks $B_i=\{x_{b_{i-1}+1},\ldots,x_{b_i}\}$,
$1\le i\le k+1$, are nonempty. Let
$\mathcal E_{\boldsymbol b}$ be the event
$P_{b_i}=z_i$ for all $1\le i\le k$, equivalently
$\Sigma(B_i)=z_i-z_{i-1}$ for all $1\le i\le k+1$.

Let $r$ be the first index for which $|B_r|$ is maximal,
and reserve $B_r$. Then $|B_r|\ge |A|/(k+1)$, and its
sum condition follows from those of the other blocks.
Expose $B_1,\ldots,B_{r-1},B_{k+1},B_k,\ldots,B_{r+1}$
in this order, revealing only their sets of elements.
Each exposed block is a prefix or suffix of the remaining
ordering.

For $1\le s\le k$, let $U_s$ be the set of unexposed
elements immediately before the $s$th exposure.
These sets depend on $\boldsymbol b$ and the random
ordering, and satisfy
$A=U_1\supsetneq\cdots\supsetneq U_k\supsetneq B_r$.
Thus $|U_s|\ge |A|/(k+1)$.
Call $U\subseteq A$ \emph{good} if
$|U\cap P|\ge \frac{|P|}{2(k+1)}$, and \emph{bad} otherwise.
By nestedness, all $U_s$ are good if and only if $U_k$
is good. Hence
\[
 F_k(A;z_1,\ldots,z_k)
 =\sum_{\boldsymbol b\in\mathcal I}
 \Prb(\mathcal E_{\boldsymbol b},\,U_k\text{ good})+\sum_{\boldsymbol b\in\mathcal I}
 \Prb(\mathcal E_{\boldsymbol b},\,U_k\text{ bad}).
\]
We bound these two sums separately.

\begin{claim}\label{clm:reservoir-good} We have
$$\sum_{\boldsymbol b\in\mathcal I}
 \Prb(\mathcal E_{\boldsymbol b},\,U_k\text{ good}) \le (k+1)\left(\frac{26(k+1)\log(2|A|)}{|P|}\right)^k.$$
\end{claim}

\begin{claimproof}
Fix $r$ and put $L=26(k+1)\log(2|A|)/|P|$.
We enlarge the sum by allowing arbitrary positive sizes
for the exposed blocks, with total less than $|A|$,
without requiring $B_r$ to be the first largest block.
For the uniform ordering $x_1,\ldots,x_{|A|}$, let $N_s$ denote the number of sequences of positive integers $(j_1,\ldots,j_s)$ with $j_1+\cdots+j_s<|A|$ such that, when the first $s$ blocks are exposed with these sizes, each block has its prescribed sum and the set of remaining elements immediately before each exposure is good. 
Put $N_0=1$.
We show that 
$$\Ex N_s\le L\Ex N_{s-1}.$$

Fix a choice of sizes for the first $s-1$ blocks and reveal
only their sets of elements. This determines whether the
partial choice is counted by $N_{s-1}$ and the remaining
set $U_s$, whose ordering is conditionally uniform.
If the partial choice is not counted, or $U_s$ is bad,
there are no extensions counted by $N_s$.
Otherwise, the next block $B_i$ is a prefix or suffix
of the remaining ordering, and the conditional expected
number of successful extensions is at most
\[
 F_1(U_s;z_i-z_{i-1})
 \le\frac{13\log(2|U_s|)}{|P\cap U_s|}
 \le L,
\]
by Lemma~\ref{lem:reservoir-prefixes} (reversing the
ordering for a suffix). Averaging over the revealed subsets
and summing over the partial size choices gives
$\Ex N_s\le L\Ex N_{s-1}$. 

Putting together, we hence obtain
$$\Ex N_k\le L^k.$$

Fix the reserved index $r$. Each original index tuple $\mathbf b=(b_1,\ldots,b_k)$ determines the sizes of the $k$ exposed blocks. Conversely, these sizes, together with $r$, uniquely determine the size of the reserved block and hence the original tuple $\mathbf b$. Thus this correspondence is injective. Whenever $\CE_{\mathbf b}$ holds and the remaining set before every exposure is good, the corresponding size sequence is counted by $N_k$. It follows that the sum of the probabilities of these events, over all tuples with reserved index $r$, is at most $\mathbb E N_k\le L^k$. Summing over the $k+1$ possible choices of $r$ gives the claimed bound.

Each original index tuple with reserved index $r$
determines a distinct complete size choice counted by
$N_k$ whenever $\mathcal E_{\boldsymbol b}$ holds and
$U_k$ is good. Thus the sum over such tuples is at most
$L^k$. Summing over the $k+1$ possible reserved indices
proves the claim.
\end{claimproof}

\begin{claim}\label{clm:reservoir-bad} We have
$$\sum_{\boldsymbol b\in\mathcal I}
 \Prb(\mathcal E_{\boldsymbol b},\,U_k\text{ bad}) \le |A|^k e^{-|P|/8(k+1)}.$$
\end{claim}

\begin{claimproof}
Fix $\boldsymbol b\in\mathcal I$.
The set $U_k$ occupies a fixed interval of positions
in the uniform ordering, so it is a uniform subset
of $A$ of deterministic size at least
$|B_r|\ge|A|/(k+1)$.
Consequently,
$\mu:=\Ex|U_k\cap P|=|P||U_k|/|A|
\ge|P|/(k+1)$.
Since $U_k$ is bad only if
$|U_k\cap P|<|P|/2(k+1)\le\mu/2$, the hypergeometric Chernoff bound
\cite[Theorem~2.10]{JLR} gives
\[
 \Prb(U_k\text{ bad})
 \le e^{-\mu/8}
 \le e^{-|P|/8(k+1)}.
\]
Therefore,
\[
 \sum_{\boldsymbol b\in\mathcal I}
 \Prb(\mathcal E_{\boldsymbol b},\,U_k\text{ bad})
 \le \sum_{\boldsymbol b\in\mathcal I}
       \Prb(U_k\text{ bad})\le |\mathcal I|e^{-|P|/8(k+1)}
 \le |A|^k e^{-|P|/8(k+1)}. \qedhere
\]
\end{claimproof}

Adding the bounds in Claims~\ref{clm:reservoir-good}
and~\ref{clm:reservoir-bad} proves Lemma~\ref{lem:reservoir-chains}.
\end{proof}

Next we turn to the other case of the dichotomy.

\subsection{Affine cores and prefix sums}

We use the affine core to bound visits to a prescribed target.
Each core element contributes the same value $a+H$ in the quotient,
so the first estimate counts visits to a target coset.  

\begin{lemma}\label{aq:quotient}
Let $H\le G$, $a\in G$, and assume that $S=E\sqcup U\subseteq G$, where
$\varnothing\ne E\subseteq a+H$ and $|U|\le|E|$.
Put $d=\operatorname{ord}(a+H)$ in $G/H$, and let $P_\ell$ be the
prefix sums of a uniform ordering of $S$ from Definition~\ref{def:prefixes}. For every $z\in G$, conditional
on any prescribed relative order of $U$, the expected number of
$0\le\ell\le|S|$ with $P_\ell\in z+H$ is at most
$$\frac{|S|+1}{d}+128\sqrt{|U|+1}.$$
\end{lemma}
\begin{proof} Put $m=|E|$ and $r=|U|$.
If $r=0$, then $P_j+H=ja+H$ for $0\le j\le m$, so at
most $(m+1)/d+1$ indices satisfy $P_j\in z+H$.
Assume henceforth that $1\le r\le m$.

Let $u_1,\ldots,u_r$ denote the prescribed ordering of $U$,
and let $T_i$ be the position of $u_i$ in the uniform
ordering of $S$. Throughout, condition on $T_1<\cdots<T_r$.
Under this conditioning, $\{T_1,\ldots,T_r\}$ is a uniform
$r$-subset of $\{1,\ldots,|S|\}$.
Thus, for $0\le i\le r$ and $0\le j\le m$, the probability $w_{ij}$
that the first $i+j$ elements of the ordering contain
$i$ elements of $U$ and $j$ elements of $E$ is
$$w_{ij}=\frac{\binom{i+j}{i}\binom{|S|-i-j}{r-i}}
{\binom{|S|}{r}}.$$
If the first $i+j$ elements contain $i$ elements of $U$,
those elements are $u_1,\ldots,u_i$ and the other $j$
elements lie in $E\subseteq a+H$. Hence
$P_{i+j}+H=\sum_{h=1}^{i}u_h+ja+H$.
Put 
$$R_i=\{j\in\mathbb Z:
\sum_{h=1}^{i}u_h+ja+H=z+H\}.$$
This set is either empty or one residue class modulo $d$.
Row $i$ of the array $(w_{ij})_{0\le i\le r, 0\le j\le m}$ therefore contributes
$\sum_{0\le j\le m,\ j\in R_i}w_{ij}$
to the expected number of indices $\ell$ with $P_\ell\in z+H$.

Set $T_0=0$ and $T_{r+1}=|S|+1$.
Exactly $i$ elements of $U$ occur among the first $\ell$
elements of the ordering if and only if
$T_i\le\ell<T_{i+1}$. Consequently,
\[
 \sum_{j=0}^{m}w_{ij}
 =\Ex(T_{i+1}-T_i)
 =\frac{|S|+1}{r+1}.
\]
The last equality holds because the gap vector
$(T_{i+1}-T_i)_{i=0}^{r}$ is uniform over all positive
integer $(r+1)$-tuples summing to $|S|+1$, so its
coordinates are exchangeable.

Note also that each row of $(w_{ij})_{0\le i\le r, 0\le j\le m}$ is unimodal, since for $j<m$,
$w_{i(j+1)}\ge w_{ij}$ if and only if
$j+1\le i(m+1)/r$.
For a fixed row $i$, put 
$$M_i=\max_{0\le j\le m}w_{ij}, \mbox{ and } I_t=\{j\in\{0,\ldots,m\}:w_{ij}>t\}.$$
By unimodality, $I_t$ is an interval.
If it contains $h\ge1$ members of $R_i$, the first and last
are at least $(h-1)d$ apart, so $(h-1)d+1\le|I_t|$.
Thus 
$$|I_t\cap R_i|\le |I_t|/d+1.$$
Integrating the indicators of membership in $I_t$ gives
\[
 \sum_{\substack{0\le j\le m\\j\in R_i}}w_{ij}
 =\int_0^{M_i}|I_t\cap R_i|\,dt
 \le\int_0^{M_i}\left(\frac{|I_t|}{d}+1\right)\,dt
 =\frac1d\sum_{j=0}^{m}w_{ij}+M_i.
\]
Summing over $i=0,\ldots,r$ and using the row-sum identity,
the expected number of indices $0\le\ell\le|S|$ with
$P_\ell\in z+H$ is at most
\[
 \sum_{i=0}^{r}
 \Bigl(\frac1d\sum_{j=0}^{m}w_{ij}+M_i\Bigr)
 =\frac{|S|+1}{d}+\sum_{i=0}^{r}M_i.
\]

It remains to bound the row maxima.
We have $M_0=M_r=1$.
For the other rows, the identity for beta distribution,
$\int_0^1 t^p(1-t)^q\,dt=p!q!/(p+q+1)!$
for nonnegative integers $p,q$, gives
\begin{align*}
 w_{ij}
 &=(|S|+1)\int_0^1
   \binom ri t^i(1-t)^{r-i}
   \binom mj t^j(1-t)^{m-j}\,dt\le \frac{|S|+1}{m+1}
   \max_{0\le t\le1}\binom ri t^i(1-t)^{r-i}.
\end{align*}
For $1\le i\le r-1$, one can check that the maximum occurs at $t=i/r$.
Using Robbins's bounds
$\sqrt{2\pi q}(q/e)^q\le q!\le e\sqrt q(q/e)^q$
for $q\ge1$, we obtain
$$M_i
 \le 2\binom ri
       \left(\frac ir\right)^i
       \left(\frac{r-i}{r}\right)^{r-i}
 \le \frac e\pi\sqrt{\frac r{i(r-i)}}
 \le \sqrt{\frac r{i(r-i)}}.
 $$
Consequently,
\[
 \sum_{i=0}^{r}M_i
 \le 2+\sqrt r\sum_{i=1}^{r-1}\frac1{\sqrt{i(r-i)}}
 \le 2+4\sqrt r
 \le128\sqrt{r+1},
\]
where we used the fact that $\sum_{i=1}^{q}i^{-1/2}\le2\sqrt q$ and hence $\sum_{i=1}^{r-1}[i(r-i)]^{-1/2}\le4$.
\end{proof}

\subsubsection{Large quotient order} Although the results below hold in general, they are only effective when $d$ is large, for instance $d\ge n^\varepsilon$. 


\begin{proposition}\label{aq:localtime}
Let $H\le G$, $a\in G$, and $S=E\sqcup U\subseteq G$, where
$\varnothing\ne E\subseteq a+H$ and $|U|\le|E|$.
Put $d=\operatorname{ord}(a+H)$ in $G/H$, and let $P_\ell$ be the
prefix sums of a uniform ordering of $S$. For every $z\in G$,
\[
 \sum_{\ell=1}^{|S|-1}\Prb(P_\ell=z)
 \le\frac6d+\frac{512}{\sqrt{|E|}}.
\]
\end{proposition}

\begin{proof}
Using the notation of Definition~\ref{def:prefixes}, put
$J_\ell=|\{x_1,\ldots,x_\ell\}\cap E|$ for $1\le\ell<|S|$.
For $1\le\ell<|U|$, we have
\[
 \Prb(P_\ell=z,J_\ell=0)
 =\Prb(P_\ell=z\mid J_\ell=0)\Prb(J_\ell=0)
 \le \frac2{|U|}\left(\frac{|U|}{|S|}\right)^\ell.
\]
Indeed, conditional on $J_\ell=0$, the prefix is a uniform
$\ell$-subset of $U$, so Observation~\ref{obs:completion}
bounds the first factor by $2/|U|$.
Sampling without replacement bounds the second by
$(|U|/|S|)^\ell$.
For $\ell=|U|\ge1$, the joint probability is at most
$\binom{|S|}{|U|}^{-1}\le1/|S|$, and for $\ell>|U|$ it is zero.
Therefore,
\[
 \sum_{\ell=1}^{|S|-1}\Prb(P_\ell=z,J_\ell=0)
 \le \frac2{|E|}+\frac1{|S|}
 \le \frac3{|E|}.
\]
Reversing the ordering with target $\Sigma(S)-z$ gives
the same total bound for $\Prb(P_\ell=z,J_\ell=|E|)$.
Both contributions vanish when $U=\varnothing$.

For each $\ell$, let $V:=\{x_1,\ldots,x_\ell\}\cap U$.
Since $J_\ell=\ell-|V|$ and
$P_\ell+H=J_\ell a+\Sigma(V)+H$, the conditions
$P_\ell\in z+H$ and $0<J_\ell<|E|$ depend only on $V$.
For every choice of $V$ satisfying these conditions,
$\{x_1,\ldots,x_\ell\}\cap E$ is a uniform nonempty proper
$J_\ell$-subset of $E$, so Observation~\ref{obs:completion}
bounds the conditional probability of $P_\ell=z$ by $2/|E|$.
Averaging over these choices of $V$ gives
$\Prb(P_\ell=z\mid P_\ell\in z+H,\ 0<J_\ell<|E|)
\le2/|E|$.
Consequently,
\[
 \begin{aligned}
 \Prb(P_\ell=z,0<J_\ell<|E|) &=\Prb(P_\ell=z\mid P_\ell\in z+H,\ 0<J_\ell<|E|)\Prb(P_\ell\in z+H,0<J_\ell<|E|)\\
 &\le\frac2{|E|}
          \Prb(P_\ell\in z+H,0<J_\ell<|E|)\le\frac2{|E|}\Prb(P_\ell\in z+H).
 \end{aligned}
\]
If the conditioning event has probability zero, the bound
is immediate.
Summing the three contributions and applying
Lemma~\ref{aq:quotient}, averaged over the relative order
of $U$, yields
\[
\sum_{\ell=1}^{|S|-1}\Prb(P_\ell=z)
 \le \frac6{|E|}
 +\frac2{|E|}
 \left(\frac{|S|+1}{d}+128\sqrt{|U|+1}\right)\le \frac6d+\frac{512}{\sqrt{|E|}},
\]
using $|U|\le|E|$ and $|E|\ge1$.
\end{proof}


The next proposition bounds joint prefix sums in a uniform
ordering of a retained set $A\subseteq S$, uniformly over
all choices of at most $20k+1$ deleted elements, as required
for local repair.

\begin{proposition}\label{aq:chains}
Let $k\ge1$ be an integer, and let $H\le G$, $a\in G$, and
$S=E\sqcup U\subseteq G$ satisfy
$|S|\ge50k$, $E\subseteq a+H$, and
$|U|\le \frac{|S|}{4(k+1)}$.
Put $d=\operatorname{ord}(a+H)$ in $G/H$.
For every $A\subseteq S$ with $|S\setminus A|\le20k+1$
and all $z_1,\ldots,z_k\in G$,
\begin{equation*}\label{aq:chain-universal}
 F_k(A;z_1,\ldots,z_k)
 \le(k+1)\left(\frac6d+
       \frac{1024\sqrt{k+1}}{\sqrt{|S|}}\right)^k.
\end{equation*}
\end{proposition}

\begin{proof}
Fix $A$ as in the statement, so
$|A|\ge|S|-20k-1\ge|S|/2$.
Let $x_1,\ldots,x_{|A|}$ be a uniformly random ordering of $A$,
and use the block decomposition from the proof of
Lemma~\ref{lem:reservoir-chains}, with $z_0=0$ and
$z_{k+1}=\Sigma(A)$.
For each index tuple, reserve its first largest block $B_r$;
its sum condition follows from those on the other blocks.

Fix $r$, and enlarge the sum by allowing all positive block
sizes with $|B_r|\ge|A|/(k+1)$.
Expose $B_1,\ldots,B_{r-1},B_{k+1},B_k,\ldots,B_{r+1}$
in this order.
For $1\le s\le k$, let $N_s=N_s(x_1,\ldots,x_{|A|})$
count the positive size choices for the first $s$ exposed
blocks that leave at least $|A|/(k+1)$ elements unexposed
and whose resulting blocks in the ordering have their
prescribed sums.
Put $N_0=1$.

Fix a deterministic choice of sizes for the first $s-1$
exposed blocks that leaves at least $|A|/(k+1)$ elements,
and reveal only their subsets.
The remaining set $A_s$ is then fixed and its ordering
is uniform. Moreover,
\[
 |A_s\cap E|
 \ge |A_s|-|U|
 \ge \frac{|A|}{k+1}-|U|
 \ge \frac{|S|}{4(k+1)}
 \ge |A_s\cap U|.
\]
If the partial choice is counted by $N_{s-1}$, let $B_i$
be the next block.
The conditional expected number of extensions counted
by $N_s$ is the sum of its prescribed-sum probabilities
over the allowed sizes, and is therefore at most
\[
 F_1(A_s;z_i-z_{i-1})
 \le \frac6d+\frac{512}{\sqrt{|A_s\cap E|}}
 \le \frac6d+\frac{1024\sqrt{k+1}}{\sqrt{|S|}}
 =:L,
\]
by Proposition~\ref{aq:localtime}, reversing the remaining
ordering when $B_i$ is exposed from the right.
If the partial choice is not counted by $N_{s-1}$,
it has no extensions counted by $N_s$.
Averaging over the revealed subsets and summing over the
partial size choices gives $\Ex N_s\le L\Ex N_{s-1}$,
and hence $\Ex N_k\le L^k$.

Argue similarly to the proof of Claim \ref{clm:reservoir-good}, each original index tuple reserving $B_r$ determines a
distinct complete size choice, counted by $N_k$ whenever
its prescribed sums hold.
Thus the total contribution for this $r$ is at most
$\Ex N_k\le L^k$.
Summing over the $k+1$ possible reserved indices proves the required estimate.
\end{proof}

\subsubsection{Large subgroup} When the quotient order is small, we instead exploit balancedness inside $H$. Although the estimates obtained in this way hold in general, they are useful for our purposes only when $|H|$ is sufficiently large, for instance, when $|H|\ge n^{1+\varepsilon}$.

Our first result provides a joint bound for the dependent block sums. 


\begin{lemma}\label{lem:multivariate-partition}
Let $A\subseteq H$ be $\alpha$-balanced. Partition a uniform
ordering of $A$ into a reserved block of size $B\ge1$ and
blocks $B_1,\ldots,B_t$ of sizes $u_1,\ldots,u_t$.
If every $u_i\ge80|A|/B$, then, for all $z_1,\ldots,z_t\in H$,
\[
 \Prb\bigl(\Sigma(B_i)=z_i\text{ for all }i\bigr)
 \le
 \prod_{i=1}^{t}\left[
 \frac1{|H|}
 +\frac{128}{\sqrt{\alpha B|A|u_i}}
 +e^{-3Bu_i/(80|A|)}\right].
\]
\end{lemma}

\begin{proof}
The case $t=0$ is immediate. Assume $t\ge1$ and put
$n=|A|$. Fix an enumeration $A=\{y_1,\ldots,y_n\}$,
and let $\{I_0,I_1,\ldots,I_t\}$ be a partition of $[n]$ with 
$|I_0|=B$ and $|I_\ell|=u_\ell$.

Define the deterministic $n\times n$ array $(b_{ir})_{1\le i,r \le n}$
by
$b_{ir}=
 \bigl(\mathbf1_{\{i\in I_1\}}y_r,\ldots,\mathbf1_{\{i\in I_t\}}y_r\bigr)\in H^t$. Note that $b_{ir}$ has at most one non-zero entry. 
Let $\pi$ be a uniform permutation of $[n]$.
Then $y_{\pi(1)},\ldots,y_{\pi(n)}$ is a uniform
ordering of $A$, with
$B_\ell=\{y_{\pi(i)}:i\in I_\ell\}$.
We see that
$\sum_{i=1}^n b_{i\pi(i)} =\bigl(\Sigma(B_1),\ldots,\Sigma(B_t)\bigr)$. Let $\nu$ denote the law of this vector.

Fix $\chi_1,\ldots,\chi_t\in\widehat H$, and put
$$a_\ell=n^{-1}|\sum_{r=1}^n\chi_\ell(y_r)|.$$
Let $\eta_\ell:H\to\mathbb R/\mathbb Z$ be the additive
character satisfying
$e^{2\pi i\eta_\ell(x)}=\overline{\chi_\ell(x)}$.
Expanding $a_\ell^2$ and using
$1-\cos(2\pi v)\le20\|v\|^2$ gives
\[
 1-a_\ell^2
 =\frac1{n^2}\sum_{r,s=1}^n
   \bigl(1-\cos(2\pi\eta_\ell(y_r-y_s))\bigr)
 \le\frac{20}{n^2}\sum_{r,s=1}^n
   \|\eta_\ell(y_r-y_s)\|^2.
\]
As $\nu$ is the joint distribution of the block sums $\bigl(\Sigma(B_1),\ldots,\Sigma(B_t)\bigr)$, we have
\[
\widehat{\nu}(\chi_1,\ldots,\chi_t)
=
\mathbb E_\pi
\exp\Big(
2\pi i\sum_{\ell=1}^t
\eta_\ell\big(\sum_{i\in I_\ell}y_{\pi(i)}\big)
\Big)=
\mathbb E_\pi
\exp\Big(
2\pi i \xi\big( \sum_{i=1}^nb_{i\pi(i)} \big)\Big),
\]
where $\xi(v_1,\ldots,v_t)=\sum_{\ell=1}^t\eta_\ell(v_\ell)$.

We apply Lemma~\ref{thm:Roos} to $(b_{ir})$ and $\xi$ with $G=H^t$. For $i\in I_\ell$ and $j\in I_0$, we have
\[
 b_{ir}-b_{jr}-b_{is}+b_{js}
 =(0,\ldots,0,y_r-y_s,0,\ldots,0),
\]
where $y_r-y_s$ occupies coordinate $\ell$.
Since there are $Bu_\ell$ such ordered row pairs,
Lemma~\ref{thm:Roos} gives
\[
 |\widehat\nu(\chi_1,\ldots,\chi_t)|
\le\exp\left(
 -\sum_{\ell=1}^t\frac{Bu_\ell}{n^3}
 \sum_{r,s=1}^n\|\eta_\ell(y_r-y_s)\|^2
 \right)\le\prod_{\ell=1}^t
 \exp\left(-\frac{Bu_\ell}{20n}(1-a_\ell^2)\right).
\]

Fourier inversion on $H^t$ and the preceding product bound give
\[
\Prb\bigl(\Sigma(B_\ell)=z_\ell
            \text{ for all }1\le\ell\le t\bigr)\\
\le
 \prod_{\ell=1}^t
 \Bigl[\frac1{|H|}
 \sum_{\chi_\ell\in\widehat H}
 e^{-Bu_\ell(1-a_\ell^2)/20n}\Bigr].
\]

We now estimate each of the factors. Since $Bu_\ell\ge80n$, we have $Bu_\ell/40n\ge2$.
For $a_\ell\ge1/2$, the inequality
$2(1-a^2)+\log a\ge0$ for $1/2\le a\le1$ gives
\[
 e^{-Bu_\ell(1-a_\ell^2)/20n}
 \le a_\ell^{Bu_\ell/40n}.
\]
For $a_\ell<1/2$, the same exponential is at most
$e^{-3Bu_\ell/80n}$.
Hence for each $\ell$, 
$$\frac1{|H|}\sum_{\chi_\ell\in\widehat H}
 e^{-Bu_\ell(1-a_\ell^2)/20n}\le\frac1{|H|}
 +\frac1{|H|}\sum_{\chi_\ell\in\widehat H, \chi_\ell\ne1}
       a_\ell^{Bu_\ell/40n}
 +e^{-3Bu_\ell/80n}.$$
To this end, let $\mu$ be the uniform measure on $A$. Lemma~\ref{lem:Fourier-moment}, applied with
$q=Bu_\ell/40n\ge2$, gives
 $$\frac1{|H|}\sum_{\chi_\ell\ne1}
       a_\ell^{Bu_\ell/40n} \le \frac{20}{n\sqrt{\alpha Bu_\ell/40n}}.$$
Therefore 
\[
 \begin{aligned}
 \frac1{|H|}\sum_{\chi_\ell\in\widehat H}
 e^{-Bu_\ell(1-a_\ell^2)/20n} <\frac1{|H|}
 +\frac{128}{\sqrt{\alpha Bnu_\ell}}
 +e^{-3Bu_\ell/80n}.
 \end{aligned}
\]
Substituting this into the preceding product bound
proves the lemma.
\end{proof}

We remark that the above proof relies heavily on the results developed in Section~\ref{sec:Fourier-moment}. We next use the joint estimates of Lemma \ref{lem:multivariate-partition} to deal with nested chains.


\begin{proposition}
\label{prop:almost-coset-direct-chain}
Let $H\le G$, $a\in G$, and $S=A\sqcup U\subseteq G$, where
$A\subseteq a+H$ is nonempty,
$A-a$ is $\alpha$-balanced in $H$ and $|U|\le |A|$.
For $k\ge1$, suppose $\alpha|A|\ge 512k^2\log(|A|+1)$.
For fixed $z_1,\ldots,z_k\in G$, a uniform nested chain in $S$
(Definition~\ref{def:prefixes}) with $|R_i|=b_i$ satisfies
\begin{equation*}\label{eq:almost-coset-direct-chain}
 \sum_{1\le b_1<\cdots<b_k<|S|}
 \Prb\bigl(\Sigma(R_i)=z_i\text{ for all }i\bigr)
 \le\left[2^{13}k^2\left(
       \frac{|A|}{|H|}+\frac1{\sqrt{\alpha|A|}}
                    +\frac{\log(|A|+1)}{|A|}\right)\right]^k.
\end{equation*}
\end{proposition}

\begin{proof}
Put $m=|A|$ and $r=|U|$, and set
$T=\lceil80(k+1)\log m\rceil$.
The hypothesis gives $m\ge3$, $2kT\le\alpha m$, and
$m/(k+1)>T$.
We first decompose the chain sum according to the numbers
of elements of $A$ and $U$ in each block.
For a fixed index tuple $1\le b_1<\cdots<b_k<m+r$, take a
uniform ordering $x_1,\ldots,x_{m+r}$ of $S$ as in
Definition~\ref{def:prefixes}, and put
$B_i=\{x_{b_{i-1}+1},\ldots,x_{b_i}\}$,
where $b_0=0$ and $b_{k+1}=m+r$.
The conditions $P_{b_i}=z_i$ for all $1\le i\le k$ are equivalent to
$\Sigma(B_i)=z_i-z_{i-1}$ for every $1\le i\le k+1$, where $z_0:=0$ and
$z_{k+1}:=\Sigma(S)$. 

Let $j_i,t_i\ge0$ be prescribed numbers of elements of $A$
and $U$ in $B_i$, respectively. For a fixed vector
$\boldsymbol j=(j_1,\ldots,j_{k+1})$ with $\sum_i j_i=m$, write
\[
 \mathcal T(\boldsymbol j)
 =\left\{\boldsymbol t\in\mathbb Z_{\ge0}^{k+1}:
       \sum_i t_i=r,\quad j_i+t_i\ge1\text{ for every }i\right\}.
\]
Every pair $(\boldsymbol j,\boldsymbol t)$ determines a unique
index tuple $(b_i)$ through $b_h=\sum_{i=1}^h(j_i+t_i)$.
Define
\begin{equation}\label{eqn:Wjt}
 W(\boldsymbol j,\boldsymbol t)
 :=\Prb\bigl(|B_i\cap A|=j_i,\ |B_i\cap U|=t_i
               \text{ for every }i\bigr)=\frac{\prod_{i=1}^{k+1}\binom{j_i+t_i}{j_i}}
         {\binom{m+r}{r}},
\end{equation}
because the positions occupied by the elements of $U$ in the
uniform ordering of $S$ form a uniform $r$-subset of
$\{1,\ldots,m+r\}$.

In what follows, conditioning on $(\boldsymbol j,\boldsymbol t)$ means
conditioning on this event. The partitions of $A$ and $U$
among the blocks are then independent and uniform.

For each $\boldsymbol j$, let $\rho$ be the first index with
maximal $j_\rho$, and reserve $B_\rho$. The reserved block $B_\rho$ contains $j_\rho$ elements of $A$,
with $j_\rho\ge m/(k+1)>T$.
Since the reserved sum condition follows from the others,
the contribution of $\boldsymbol j$ to the chain sum is
\[
 \Phi(\boldsymbol j)
 :=\sum_{\boldsymbol t\in\mathcal T(\boldsymbol j)}
 W(\boldsymbol j,\boldsymbol t)\,
 \Prb\bigl(\Sigma(B_i)=z_i-z_{i-1}\text{ for every }i\ne\rho
           \mid\boldsymbol j,\boldsymbol t\bigr).
\]
For each index tuple we sum over the possible values of
$|B_i\cap A|$ and $|B_i\cap U|$ using the law of total probability.
By grouping the resulting terms by $\boldsymbol j$, we obtain 
\begin{equation}\label{eq:direct-count-decomposition}
 F_k(S;z_1,\ldots,z_k)
 =\sum_{\boldsymbol j\in\mathbb Z_{\ge0}^{k+1}:\,
                   \sum_i j_i=m}\Phi(\boldsymbol j).
\end{equation}

Fix $\boldsymbol j$ for the next two claims.
Let $I_0=\{i:j_i=0\}$ and $q=|I_0|$.
These blocks do not contain elements of $A$, so their sum conditions
cannot use the balancedness of $A-a$. We first bound their
weighted contribution, before revealing any element of $A$.

\begin{claim}\label{cl:direct-zero} We have
$$\sum_{\boldsymbol t\in\mathcal T(\boldsymbol j)}
 W(\boldsymbol j,\boldsymbol t)\,
 \Prb\bigl(\Sigma(B_i)=z_i-z_{i-1}\text{ for every }i\in I_0
           \mid\boldsymbol j,\boldsymbol t\bigr)
 \le\frac{2^k}{m^q}.$$
\end{claim}

\begin{claimproof}
Let $E_0$ be the event that
$\Sigma(B_i)=z_i-z_{i-1}$ for every $i\in I_0$.
Write $I_0=\{i_1<\cdots<i_q\}$ and $v_a=t_{i_a}\ge1$.
Condition on $|B_i\cap A|=j_i$ and $|B_i\cap U|=t_i$
for every $i$, and expose $B_{i_1},\ldots,B_{i_q}$ in this order.
Given the actual subsets already revealed, $B_{i_a}$ is a uniform
$v_a$-subset of $U\setminus\bigcup_{b<a}B_{i_b}$,
which has $r-\sum_{b<a}v_b$ elements.
Observation~\ref{obs:completion} bounds its prescribed-sum
probability by $1/(r-v_1-\cdots-v_a+1)$, with the trivial bound
$1$ if the block contains all unexposed elements of $U$. Thus
\begin{equation}\label{eqn:cv}
 \Prb(E_0\mid\boldsymbol j,\boldsymbol t)
 \le\prod_{a=1}^q\frac1{r-v_1-\cdots-v_a+1}=: c(\boldsymbol v).
\end{equation}
In what follows, we will bound $\sum_{\boldsymbol t\in\mathcal T(\boldsymbol j)}
 W(\boldsymbol j,\boldsymbol t) c(\boldsymbol v)$. 
 
Fix $\boldsymbol v=(v_1,\ldots,v_q)$ with total $v_1+\dots+v_q=s\le r$.
The numbers $t_i=|B_i\cap U|$ for $i\notin I_0$
are nonnegative and sum to $r-s$. We sum over the possible values of these $t_i$, using \eqref{eqn:Wjt} 
\[\sum_{(t_i)_{i\notin I_0}}W(\boldsymbol j,\boldsymbol t)=\binom{m+r}{r}^{-1}\sum_{\substack{t_i\ge0,\ i\notin I_0\\ \sum_{i\notin I_0}t_i=r-s}}\prod_{i\notin I_0}\binom{j_i+t_i}{j_i}.\] 
To simplify, we evaluate the inner sum using the generating function identity \(\sum_{t\ge0}\binom{j+t}{j}x^t=(1-x)^{-j-1}\). Since \(\sum_{i\notin I_0}j_i=m\) and there are \(k+1-q\) indices outside \(I_0\), we have \(\prod_{i\notin I_0}\sum_{t\ge0}\binom{j_i+t}{j_i}x^t=(1-x)^{-(m+k+1-q)}\). The inner sum is precisely the coefficient of \(x^{r-s}\) in this product, which equals \(\binom{m+r-s+k-q}{r-s}\) by the negative-binomial expansion. Therefore

\begin{equation}\label{eqn:sumW}
 \sum_{(t_i)_{i\notin I_0}}W(\boldsymbol j,\boldsymbol t)
 =\frac{\binom{m+r-s+k-q}{r-s}}{\binom{m+r}{r}}=\frac{(r)_s}{(m+r)_s}
   \prod_{b=1}^{k-q}\frac{m+r-s+b}{m+b}\le2^{k-q}\frac{(r)_s}{(m+r)_s},
\end{equation}
where $(x)_s=x(x-1)\cdots(x-s+1)$ and $(x)_0=1$.

We now sum the resulting bound over all $\boldsymbol v$
with $v_a\ge1$ and $\sum_{a=1}^q v_a\le r$. Since $c(\boldsymbol v)$ depends only on $\boldsymbol v$,
it can be taken outside the sum over $t_i$, $i\notin I_0$.
Thus the quantity in the claim satisfies
\[
 \sum_{\boldsymbol t\in\mathcal T(\boldsymbol j)}
 W(\boldsymbol j,\boldsymbol t)\,
 \Prb(E_0\mid\boldsymbol j,\boldsymbol t)\le
 \sum_{\substack{v_1,\ldots,v_q\ge1\\
                 v_1+\cdots+v_q\le r}}
 c(\boldsymbol v)
 \sum_{(t_i)_{i\notin I_0}}W(\boldsymbol j,\boldsymbol t)\le2^{k-q}M_q(r),
\]
where we used~\eqref{eqn:cv},\eqref{eqn:sumW}. Here $M_0(R)=1$ and for $h\ge1$
\[
 M_h(R) =\sum_{\substack{v_1,\ldots,v_h\ge1\\
                   v_1+\cdots+v_h\le R}}
 \frac{(R)_{v_1+\cdots+v_h}}
      {(m+R)_{v_1+\cdots+v_h}}
 \prod_{a=1}^h\frac1{R-v_1-\cdots-v_a+1}.
\]

It remains to prove $M_h(R)\le m^{-h}$ uniformly for
$0\le R\le m$. Grouping terms according to the first size
$v_1=v$ gives
\[
 M_h(R)
 =\sum_{v=1}^R
   \frac{(R)_v}{(m+R)_v(R-v+1)}\,M_{h-1}(R-v).
\]
Observe that the coefficients in this recursion satisfy
\[
 \sum_{v=1}^R\frac{(R)_v}{(m+R)_v(R-v+1)}
 =\frac1m\sum_{v=1}^R
   \left(
    \frac{(R)_{v-1}}{(m+R)_{v-1}}
    -\frac{(R)_v}{(m+R)_v}
   \right)=\frac1m\left(1-\frac1{\binom{m+R}{R}}\right)
 \le\frac1m.
\]
Starting from $M_0(R)=1$, induction on $h$ therefore gives
\[
 M_h(R)
 \le m^{-(h-1)}
      \sum_{v=1}^R\frac{(R)_v}{(m+R)_v(R-v+1)}
 \le m^{-h}.
\]
By taking $h=q$ and $R=r$, we thus obtain
\[
 \sum_{\boldsymbol t\in\mathcal T(\boldsymbol j)}
 W(\boldsymbol j,\boldsymbol t)\,
 \Prb(E_0\mid\boldsymbol j,\boldsymbol t)\le2^{k-q}m^{-q}\le2^km^{-q}.
 \]
\end{claimproof}

We next bound the probability that
$\Sigma(B_i)=z_i-z_{i-1}$ for every $i\ne\rho$ with $j_i>0$,
conditional on the subsets $B_i\cap U$ in all blocks.
We obtain the same bound for every possible choice of these
subsets and combine it with Claim~\ref{cl:direct-zero}.
For $u\ge T$, put
\begin{equation}\label{eqn:fu}
f(u)=\frac1{|H|}
      +\frac{256\sqrt{k+1}}{m\sqrt{\alpha u}}
      +e^{-3u/[80(k+1)]}.
\end{equation}

\begin{claim}\label{cl:direct-count}
Let $p$ be the number of indices $i\ne\rho$ with $0<j_i<T$.
Then
\[
 \Phi(\boldsymbol j)
 \le2^k m^{-q}\left(\frac4m\right)^p
       \prod_{i\ne\rho: \,j_i\ge T}f(j_i).
\]
\end{claim}

\begin{claimproof}
Fix $\boldsymbol t\in\mathcal T(\boldsymbol j)$ and, conditional
on $(\boldsymbol j,\boldsymbol t)$, reveal the subsets
$B_i\cap U$ in every block. The subsets $B_i\cap A$ remain
a uniform partition of $A$ with sizes $j_1,\ldots,j_{k+1}$.
The condition $\Sigma(B_i)=z_i-z_{i-1}$ is then equivalent to
$\Sigma(B_i\cap A)=(z_i-z_{i-1})-\Sigma(B_i\cap U)$.

Expose the subsets $B_i\cap A$ for the $p$ indices with
$0<j_i<T$, in increasing order of $i$. Conditional on the
actual subsets already revealed, the next subset is uniform
among the $j_i$-subsets of the unexposed elements of $A$.
At least $m-kT\ge m/2$ elements remain, and
$B_\rho\cap A$ remains unexposed, so the chosen subset is proper.
Observation~\ref{obs:completion} bounds its prescribed-sum
probability by $2/(m-kT)\le4/m$.

After these exposures, let $A'$ be the set of unexposed
elements of $A$.
At most $kT\le\alpha m/2$ elements have been removed, so
$|A'|\ge m/2$. For every coset $C$ of a proper subgroup of $H$, by definition
$$|(A'-a)\backslash C| \ge |(A-a)\backslash C| - kT \ge \alpha m-kT\ge\alpha m/2\ge\alpha |A'|/2.$$
 Thus $A'-a$ is $\alpha/2$-balanced in $H$.

Note that the reserved subset $B_\rho\cap A$ still has
$j_\rho\ge m/(k+1)$ elements. Every remaining block other than $B_\rho$ that contains
elements of $A$ has $j_i\ge T$, with
\begin{equation}\label{eqn:ji}
 j_i\ge T\ge80(k+1)\ge80|A'|/j_\rho.
\end{equation}
The prescribed sum of its translated subset
$(B_i\cap A)-a$ is
$(z_i-z_{i-1})-\Sigma(B_i\cap U)-j_i a$; if this target lies outside
$H$, the condition is impossible.

By \eqref{eqn:ji}, Lemma~\ref{lem:multivariate-partition} is applicable to $A'-a$
with reserved size $j_\rho$. Since $|A'|\ge m/2$ and
$j_\rho\ge m/(k+1)$, the factors in that lemma satisfy
$\frac{128}{\sqrt{(\alpha/2)j_\rho|A'|j_i}}
 \le\frac{256\sqrt{k+1}}{m\sqrt{\alpha j_i}}$,
 and $e^{-3j_\rho j_i/80|A'|}
 \le e^{-3j_i/80(k+1)}$.
Thus the joint probability is at most
$\prod_{i\ne\rho,\ j_i\ge T}f(j_i)$, with $f(.)$ from \eqref{eqn:fu}.

These bounds hold conditional on every choice of the
previously revealed subsets. Multiplying the bounds for
the $p$ exposures and then applying the joint bound above
shows that, for every choice of the subsets $B_i\cap U$,
the conditional probability of
$\Sigma(B_i)=z_i-z_{i-1}$ for every $i\ne\rho$ with $j_i>0$ is at most
$$\left(4/m\right)^p\prod_{i\ne\rho,\ j_i\ge T}f(j_i).$$
For $i\in I_0$, we have $B_i=B_i\cap U$, so the corresponding
sum conditions are determined by the revealed subsets.

Let \(E_0\) denote the event that \(\Sigma(B_i)=z_i-z_{i-1}\) for every \(i\in I_0\). Conditional on \((\boldsymbol j,\boldsymbol t)\), we first reveal all the subsets \(B_i\cap U\). Since \(B_i=B_i\cap U\) for \(i\in I_0\), the event \(E_0\) is completely determined by these revealed subsets. On the other hand, the preceding conditional bound for the sum conditions with \(i\ne\rho\) and \(j_i>0\) holds uniformly over every possible choice of the subsets \(B_i\cap U\). We may therefore average over these subsets, retaining the indicator of \(E_0\), to obtain
\[
\begin{aligned}
&\Prb\bigl(\Sigma(B_i)=z_i-z_{i-1}\text{ for every }i\ne\rho
\mid\boldsymbol j,\boldsymbol t\bigr)\le
\left(\frac4m\right)^p
\prod_{\substack{i\ne\rho\\j_i\ge T}}f(j_i)\,
\Prb\bigl(E_0\mid\boldsymbol j,\boldsymbol t\bigr).
\end{aligned}
\]
The factor \((4/m)^p\prod_{i\ne\rho,\ j_i\ge T}f(j_i)\) depends only on \(\boldsymbol j\), and is therefore independent of \(\boldsymbol t\). Multiplying the preceding inequality by \(W(\boldsymbol j,\boldsymbol t)\) and summing over all \(\boldsymbol t\in\mathcal T(\boldsymbol j)\), we obtain
\[
\begin{aligned}
\Phi(\boldsymbol j)
&\le
\left(\frac4m\right)^p
\prod_{\substack{i\ne\rho\\j_i\ge T}}f(j_i)
\sum_{\boldsymbol t\in\mathcal T(\boldsymbol j)}
W(\boldsymbol j,\boldsymbol t)
\Prb(E_0\mid\boldsymbol j,\boldsymbol t)\le
\frac{2^k}{m^q}
\left(\frac4m\right)^p
\prod_{\substack{i\ne\rho\\j_i\ge T}}f(j_i),
\end{aligned}
\]
where the last inequality follows from Claim~\ref{cl:direct-zero}. This proves the claim.
\end{claimproof}

To proceed, we sum over \(\boldsymbol j\) in \eqref{eq:direct-count-decomposition}. Recall that for each \(\boldsymbol j\), the reserved index \(\rho\) is chosen to be the first index at which \(j_i\) attains its maximum. We first fix \(\rho\), of which there are at most \(k+1\) choices. Since \(\sum_{i=1}^{k+1}j_i=m\), the remaining \(k\) coordinates determine \(j_\rho=m-\sum_{i\ne\rho}j_i\). Thus, for fixed \(\rho\), we only need to sum over the nonnegative integers \(j_i\), \(i\ne\rho\), subject to the restrictions that \(j_\rho\ge0\) and \(j_\rho\ge j_i\) for every \(i\ne\rho\), together with the tie-breaking condition defining \(\rho\).

By Claim~\ref{cl:direct-count}, the contribution of each such \(\boldsymbol j\) is bounded by \(2^k\prod_{i\ne\rho}g(j_i)\), where we set \(g(0)=1/m\), \(g(u)=4/m\) for \(1\le u<T\), and \(g(u)=f(u)\) for \(T\le u\le m\). Here we used the fact that \(j_\rho>0\), so the \(q\) zero coordinates counted in Claim~3.15 all occur among the indices \(i\ne\rho\). Since every factor \(g(j_i)\) is nonnegative, dropping the restrictions on \(j_\rho\) and its maximality can only increase the sum. We may therefore let each \(j_i\), \(i\ne\rho\), range independently over \(\{0,\ldots,m\}\). The resulting sum factors as
\[
\sum_{(j_i)_{i\ne\rho}}\prod_{i\ne\rho}g(j_i)
\le
\prod_{i\ne\rho}\left(\sum_{u=0}^m g(u)\right)
=
\left(\sum_{u=0}^m g(u)\right)^k.
\]
Passing back to $f$, the value \(u=0\) contributes \(1/m\), the values \(1\le u<T\) contribute at most \(4T/m\) in total, and the values \(T\le u\le m\) contribute \(\sum_{u=T}^m f(u)\). Summing over the at most \(k+1\) choices of \(\rho\), we obtain
\[
F_k(S;z_1,\ldots,z_k)
\le
(k+1)2^k
\left(
\frac1m+\frac{4T}{m}
+\sum_{u=T}^m f(u)
\right)^k.
\]
Finally, since $\sum_{u=1}^m u^{-1/2}\le2\sqrt m$ and
$T\ge80(k+1)\log m$, using \eqref{eqn:fu} we have
\[
 \sum_{u=T}^m f(u)
 \le\frac m{|H|}
    +\frac{512\sqrt{k+1}}{\sqrt{\alpha m}}
    +\frac{28(k+1)}{m^3}.
\]
Because $T\le81(k+1)\log(m+1)$,
$(k+1)^{1/k}\le k+1$, and $(k+1)^2\le4k^2$, we obtain
\begin{align*}
 F_k(S;z_1,\ldots,z_k)
 &\le(k+1)\left[
 \frac{8T+2}{m}+\frac{2m}{|H|}
 +\frac{1024\sqrt{k+1}}{\sqrt{\alpha m}}
 +\frac{56(k+1)}{m^3}\right]^k\\
 &\le\left[2^{13}k^2\left(
 \frac m{|H|}+\frac1{\sqrt{\alpha m}}
             +\frac{\log(m+1)}m\right)\right]^k. \qedhere
\end{align*}
\end{proof}

\section{Local repair}
\label{sec:repair-full}

We use the local-repair mechanism of
Pham--Sauermann~\cite[Section~5]{PS26}. Repeated partial
sums correspond to zero-sum intervals; swapping nearby entries
eliminates the corresponding repetitions. The following reformulation takes as
inputs the integrated prefix and nested-chain estimates developed
in the preceding section. Uniformity after deleting a bounded
number of elements allows us to reveal the entries near a proposed
swap while controlling the remaining intervals.

\begin{proposition}\label{prop:repair-full}
For every integer $D\ge6$ there is $c_D>0$ with the following property.
Let $S\subseteq G\setminus\{0\}$ satisfy $|S|\ge50D+2$. Suppose
$q,F\ge0$ satisfy the following bounds for every $A\subseteq S$ with
$|S\setminus A|\le20D+1$: the prefix sums $P_\ell$ of a uniform
ordering of $A$ (Definition~\ref{def:prefixes}) satisfy
\begin{align}
 \sum_{\ell=1}^{|A|-1}\Prb(P_\ell=z)&\le q
       &&\text{for every }z\in G,\label{rt:onepoint}\\
 \sum_{1\le b_1<\cdots<b_D<|A|}
 \Prb(P_{b_i}=z_i\text{ for all }i)&\le F
       &&\text{for every }z_1,\ldots,z_D\in G.\label{rt:chain}
\end{align}
Then $S$ has a valid ordering whenever $q+|S|F<c_D$.
\end{proposition}

As we will see in Section~\ref{sec:main-proof}, the results of Section~\ref{sec:nested} imply that $q$ and $F$ are very small in most cases. Consequently, Proposition~\ref{prop:repair-full} applies to almost all $S$, with the remaining exceptional cases treated separately in Section \ref{sec:external-absorption}.

For Proposition~\ref{prop:repair-full} we choose swaps at the right endpoints of zero-sum intervals.
The setting below records four obstructions; excluding them will
permit a greedy choice of disjoint swaps.

\begin{setting}\label{setting} Let $n=|S|$, and choose a uniform bijection $\sigma:[n]\to S$.
For any ordering $\tau:[n]\to S$, write
$Z_\tau(a,b)=\sum_{i=a}^b\tau(i)$.
An interval $[a,b]$ is a \emph{zero interval of $\tau$}
if $Z_\tau(a,b)=0$.
Let
\[
 B=\{b\in[n]: Z_\sigma(a,b)=0
                 \text{ for some }2\le a<b\}
\]
be the set of \emph{bad endpoints of $\sigma$}.
Throughout the section, $B$ refers to the original ordering
$\sigma$. 

Throughout this section, permutations act on the position set $[n]$.
A permutation $\pi$ \emph{fixes} a set $I\subseteq[n]$ if
$\pi(i)=i$ for every $i\in I$, and is \emph{supported on}
$W\subseteq[n]$ if it fixes $[n]\setminus W$.

For distinct $x,y\in[n]$, write $(x \, y)$ for the
transposition exchanging $x$ and $y$ and fixing every
other position.
A permutation is \emph{admissible} if it is a product
of disjoint transpositions $(x \, y)$ with $x<y\le x+5D$. For a permutation $\pi$ of $[n]$ and $b\in B$, a position
$b<y\le \min\{b+5D,n\}$ is \emph{blocked for $\pi$} if
$\sigma\circ\pi\circ (b \,y)$ has a zero interval $[s,t]$,
with $2\le s<t\le n$, satisfying
$s\in[b+1,y]$ or $t\in[b,y-1]$.
Such an interval is a \emph{witness}.
\footnote{%
In the repair procedure, $\pi$ records earlier swaps.
Swapping $b,y$ in $\sigma\circ\pi$ changes precisely the
sums of the first $\ell$ entries for $b\le\ell<y$.
The endpoint conditions detect repetitions involving these
changed sums. Thus an unblocked swap repairs $b$ without
creating new repetitions among partial sums.}

We avoid the following events. Position intervals are intersected
with $[n]$ when necessary.
\begin{enumerate}[label=\(B_\arabic*\),start=0]
\item : There exist $b\in B$ with $b\le n-30D$ and distinct
      subsets $J,J'\subseteq[b,b+20D]$ such that
      $\sum_{i\in J}\sigma(i)=\sum_{i\in J'}\sigma(i)$. The empty subset is allowed, with sum $0$.
\item : $B\cap[n-30D,n]\ne\varnothing$.
\item : Some interval $[z-10D,z+10D]$ contains more than $D$ elements of $B$.
\item : For some $b\in B$ and admissible $\pi$ fixing all positions
      below $b$ (i.e. $\pi(i)=i$ for every $i<b$), at least $2D$ choices $y\in[b+1,b+5D]\cap[n]$
      are blocked for $\pi$.
\end{enumerate}
\end{setting}

Avoiding $B_1$ leaves room to the right of every bad endpoint;
avoiding $B_2$ and $B_3$ leaves an unused, unblocked swap partner.
The local uniqueness supplied by avoiding $B_0$ is used to bound
the latter two events. 

\begin{figure}[htbp]
\centering
\begin{tikzpicture}[x=1cm,y=1cm,font=\small]
\node[anchor=west,font=\small\bfseries] at (.5,3.95)
  {(a) Intervals affected by a swap};
\draw[abGrey,line width=.65pt] (.5,2.85)--(12.8,2.85);
\draw[abGrey!45,densely dashed] (4.2,.8)--(4.2,3.05);
\draw[abGrey!45,densely dashed] (7.5,.8)--(7.5,3.05);
\node[circle,fill=abPurple,inner sep=0pt,minimum size=3.5pt] at (4.2,2.85) {};
\node[circle,fill=abPurple,inner sep=0pt,minimum size=3.5pt] at (7.5,2.85) {};
\node[below=3pt] at (4.2,2.85) {$b$};
\node[below=3pt] at (7.5,2.85) {$y$};
\draw[abPurple,line width=.85pt,<->,>=Stealth]
  (4.2,2.98)..controls (4.8,3.7) and (6.9,3.7)..(7.5,2.98);
\node[fill=white,inner sep=2pt] at (5.85,3.54) {$\pi_{b,y}$};
\node[anchor=west,font=\footnotesize,text=abGrey] at (8.2,3.05)
  {$b<y\le b+5D$};

\draw[abBlue,line width=1.3pt] (.95,2.02)--(5.8,2.02);
\foreach \x in {.95,5.8}{\draw[abBlue,line width=.9pt] (\x,1.92)--(\x,2.12);}
\node[below=2pt,font=\footnotesize] at (.95,2.02) {$s$};
\node[below=2pt,font=\footnotesize] at (5.8,2.02) {$t$};
\node[anchor=west,text=abBlue] at (8.2,2.02) {$t\in[b,y-1]$};

\draw[abOrange,line width=1.3pt] (5.4,1.18)--(12.15,1.18);
\foreach \x in {5.4,12.15}{\draw[abOrange,line width=.9pt] (\x,1.08)--(\x,1.28);}
\node[below=2pt,font=\footnotesize] at (5.4,1.18) {$s$};
\node[below=2pt,font=\footnotesize] at (12.15,1.18) {$t$};
\node[anchor=west,text=abOrange] at (.95,1.18) {$s\in[b+1,y]$};

\node[anchor=west,font=\small\bfseries] at (.5,.22)
  {(b) Disjoint swaps, selected from right to left};
\draw[abGrey,line width=.65pt] (.5,-.9)--(12.8,-.9);
\foreach \b/\y/\i in {2/3.2/3,5/6.5/2,9/10.7/1}{
  \draw[abPurple,line width=.85pt]
    (\b,-.9)..controls (\b,-.15) and (\y,-.15)..(\y,-.9);
  \node[circle,fill=abPurple,inner sep=0pt,minimum size=3.5pt] at (\b,-.9) {};
  \node[circle,draw=abPurple,fill=white,line width=.7pt,inner sep=0pt,minimum size=3.5pt] at (\y,-.9) {};
  \node[below=3pt,font=\footnotesize] at (\b,-.9) {$b_{\i}$};
  \node[below=3pt,font=\footnotesize] at (\y,-.9) {$y_{\i}$};
  \node[font=\footnotesize,fill=white,inner sep=1.5pt] at ({(\b+\y)/2},-.28) {\i};
}
\draw[abGrey,line width=.65pt,-{Stealth[length=1.6mm]}] (10.7,-1.66)--(2,-1.66);
\node[fill=white,inner sep=2pt,font=\footnotesize,text=abGrey] at (6.35,-1.66)
  {repair order};
\end{tikzpicture}
\caption{Local repair by disjoint swaps.
(a) Intervals affected by swapping positions $b$ and $y$.
(b) Disjoint swaps selected from right to left.}
\label{fig:local-repair}
\end{figure}
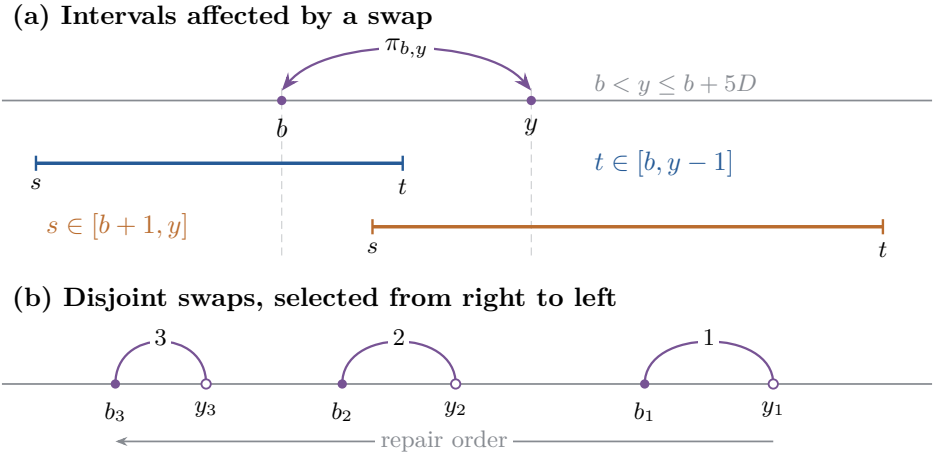

To prove Proposition \ref{prop:repair-full} we will rely on three lemmas below.

First, the one-point estimate controls both local subset-sum coincidences
near bad endpoints and bad endpoints near the end of the ordering.

\begin{lemma}\label{lem:end-effects}
Under the assumptions of Proposition~\ref{prop:repair-full},
the events $B_0$ and $B_1$ 
satisfy $\Prb(B_0)\le2^{40D+3}q$ and
$\Prb(B_1)\le(30D+1)q$.
\end{lemma}

Once these events are excluded, a cluster of bad endpoints forces
several nested prefix-sum conditions. This is where the chain
estimate enters.

\begin{lemma}\label{lem:clusters-bad-endpoints}
Under the assumptions of Proposition~\ref{prop:repair-full},
the events $B_0,B_1,B_2$ 
satisfy $\Prb(B_2\setminus(B_0\cup B_1))\le n(20D)^DF$.
\end{lemma}

Several blocked partners also force nested prefix-sum conditions.
The next estimate additionally accounts for the possible effects
of earlier repair swaps.

\begin{lemma}\label{lem:blocked-swap-partners}
Under the assumptions of Proposition~\ref{prop:repair-full},
the events $B_0,B_1,B_3$ 
satisfy $\Prb(B_3\setminus(B_0\cup B_1))\le
       2n(5D)^{2D}2^{26D^2}F$.
\end{lemma}

\begin{proof}[Proof of Proposition~\ref{prop:repair-full}]
Assuming the three lemmas, their probability bounds imply
\[
 \Prb(B_0\cup B_1\cup B_2\cup B_3)
 \le \bigl(30D+1+2^{40D+3}\bigr)q
      +n\bigl((20D)^D+2(5D)^{2D}2^{26D^2}\bigr)F.
\]
Choose $c_D>0$ sufficiently small so that the right side is then less than $1/2$ when $q+nF<c_D$. Fix $\sigma$ avoiding all four events, and write
$B=\{b_1>\cdots>b_s\}$.

We prove by induction that, for $1\le j\le s+1$,
there is a permutation
$\pi_j=(b_1\,y_1)\circ\cdots\circ(b_{j-1}\,y_{j-1})$,
where the transpositions are disjoint,
$b_\ell<y_\ell\le b_\ell+5D$, and $y_\ell\notin B$
for every $1\le\ell<j$, such that every zero interval
$[a,t]$ of $\sigma\circ\pi_j$ with $2\le a<t\le n$
satisfies $t\in\{b_j,b_{j+1},\ldots,b_s\}$. After all $s$ swaps, no zero interval remains, so $\sigma\circ\pi_{s+1}$ is valid.

For $j=1$, take $\pi_1=\mathrm{id}$.
Now suppose the assertion holds for some $j\le s$,
and fix such a $\pi_j$.
We see that $\pi_j$ is admissible
and fixes $[b_j]$.
Put $Y_j=[b_j+1,b_j+5D]$.
The absence of $B_1$ gives $Y_j\subseteq[n]$.
By the absence of $B_3$, fewer than $2D$ positions in $Y_j$
are blocked for $\pi_j$, and by the absence of $B_2$,
$|Y_j\cap B|\le D$.
Moreover, if an earlier partner $y_\ell$ lies in $Y_j$,
then $b_j<b_\ell<y_\ell\le b_j+5D$.
The corresponding $b_\ell$ are distinct members of $Y_j\cap B$,
so at most $D$ earlier partners lie in $Y_j$.
Thus at least $5D-2D-D-D=D$ choices in
$Y_j\setminus(B\cup\{y_1,\ldots,y_{j-1}\})$ are unblocked.
Choose one, $y_j$, and put
$\pi_{j+1}=\pi_j\circ(b_j\,y_j)$.
The new transposition is disjoint from all earlier ones.

Let $[a,t]$ be a zero interval of $\sigma\circ\pi_{j+1}$ with
$2\le a<t\le n$.
Since $y_j$ is unblocked, $a\notin[b_j+1,y_j]$ and
$t\notin[b_j,y_j-1]$.
Thus $[a,t]$ contains both $b_j,y_j$ or neither, so
$Z_{\sigma\circ\pi_j}(a,t)=Z_{\sigma\circ\pi_{j+1}}(a,t)=0$.
The induction hypothesis gives $t\in\{b_j,\ldots,b_s\}$,
while the second exclusion gives $t\ne b_j$.
Hence $t\in\{b_{j+1},\ldots,b_s\}$, completing the induction.
\end{proof}
We now prove the lemmas.
\begin{proof}[Proof of Lemma~\ref{lem:end-effects}]
Fix $b\in[n]$, and let $\sigma'$
denote the sequence
$\sigma(b),\ldots,\sigma(1),
\sigma(n),\ldots,\sigma(b+1)$. This is a uniform ordering of $S$, since it is a fixed
rearrangement of $\sigma$.
For $2\le a<b$, we have
$Z_\sigma(a,b)=\sum_{i=1}^{b-a+1}\sigma'(i)$,
where $1\le b-a+1<n$.
Thus \eqref{rt:onepoint} gives
$$\Prb(b\in B)\le\sum_{\ell=1}^{n-1}\Prb_{\sigma'}(\sum_{i=1}^{\ell}\sigma'(i)=0)\le q.$$
By the union bound,
$$\Prb(B_1)\le \sum_{b=n-30D}^n\Prb(b\in B)\le (30D+1)q.$$

To bound $\Prb(B_0)$, fix $b\le n-30D$ and distinct subsets
$J,J'$ of $W=[b,b+20D]$.
Choose $v\in J\mathbin\triangle J'$ and reveal $\sigma$ on
$W\setminus\{v\}$.
Since $v$ belongs to exactly one of $J,J'$, the equality
$\sum_{i\in J}\sigma(i)=\sum_{i\in J'}\sigma(i)$
allows at most one value of $\sigma(v)$. Conditional on the revealed values, $\sigma(v)$ is uniform
among the $n-20D$ remaining elements of $S$.
Averaging over the revealed values therefore gives
\[
 \Prb\Bigl(\sum_{i\in J}\sigma(i)
          =\sum_{i\in J'}\sigma(i)\Bigr)
 \le \frac{1}{n-20D}\le\frac{2}{n}.
\]

Now reveal $\sigma|_W$.
Let $\sigma''$ denote the sequence
$\sigma(b-1),\ldots,\sigma(1),
\sigma(b+20D+1),\ldots,\sigma(n)$.
This is a uniform ordering of $A=S\setminus\sigma(W)$.
For $2\le a<b$, we have
$Z_\sigma(a,b)=\sigma(b)+\sum_{i=1}^{b-a}\sigma''(i)$, where $1\le b-a<|A|$.
For every possible $\sigma|_W$, it again follows from
\eqref{rt:onepoint} that
\[
 \Prb(b\in B\mid\sigma|_W)
 \le \sum_{\ell=1}^{|A|-1}
      \Prb_{\sigma''}\Bigl(\sum_{i=1}^{\ell}\sigma''(i)=-\sigma(b)
                \,\Bigm|\,\sigma|_W\Bigr)
 \le q.
\]
The event $\sum_{i\in J}\sigma(i)=\sum_{i\in J'}\sigma(i)$
is determined by $\sigma|_W$.
Therefore,
\[
 \Prb\Bigl(b\in B,\,
     \sum_{i\in J}\sigma(i)=\sum_{i\in J'}\sigma(i)\Bigr)
 \le q\,\Prb\Bigl(
     \sum_{i\in J}\sigma(i)=\sum_{i\in J'}\sigma(i)\Bigr)
 \le \frac{2q}{n}.
\]
There are at most $n2^{2(20D+1)}=n2^{40D+2}$
choices of $(b,J,J')$. A union bound over these choices yields
$$\Prb(B_0)\le n2^{40D+2}\cdot(2q/n)=2^{40D+3}q.$$
\end{proof}

\begin{proof}[Proof of Lemma~\ref{lem:clusters-bad-endpoints}]
Fix an ordering $\sigma$ for which $B_2$ occurs but neither
$B_0$ nor $B_1$ occurs.
Since $B_2$ occurs, choose $z$ such that
$|B\cap[z-10D,z+10D]|\ge D+1$.
Let $b_0$ be the smallest bad endpoint in this interval,
and choose $D$ others, $b_1,\ldots,b_D$.
Then $b_0<b_i\le b_0+20D$ for $1\le i\le D$,
and $b_0<n-30D$ since $B_1$ does not occur.
For each $1\le i\le D$, choose $2\le a_i<b_i$ with
$Z_\sigma(a_i,b_i)=0$, as guaranteed by $b_i\in B$.

The left endpoints $a_1,\ldots,a_D$ are distinct and
all lie below $b_0$. Indeed, otherwise two zero intervals
would yield a nonempty zero-sum subset of
$[b_0,b_0+20D]$, contrary to the absence of $B_0$.
Relabel so that $a_1<\cdots<a_D$.

Fix $2\le b_0\le n-30D$ and distinct positions
$b_1,\ldots,b_D\in(b_0,b_0+20D]$.
Put $W_0=[b_0,b_0+20D]$ and reveal $\sigma(i)$
for every $i\in W_0$.
Set $A=S\setminus\sigma(W_0)$, and let
$\sigma'(1),\ldots,\sigma'(|A|)$ denote the sequence
$\sigma(b_0-1),\ldots,\sigma(1),
\sigma(b_0+20D+1),\ldots,\sigma(n)$.
Conditional on $\sigma|_{W_0}$, this is a uniform ordering of $A$.
Put $P_\ell=\sigma'(1)+\cdots+\sigma'(\ell)$ and
$z_i=-\sum_{j=b_0}^{b_i}\sigma(j)$.
For $2\le a_i<b_0$, the equation $Z_\sigma(a_i,b_i)=0$
is equivalent to $P_{b_0-a_i}=z_i$, with the targets
$z_i$ fixed by $\sigma|_{W_0}$ as $b_i \in W_0$.

For $2\le a_1<\cdots<a_D<b_0$, set
$\ell_i=b_0-a_{D+1-i}$.
Then $1\le\ell_1<\cdots<\ell_D\le b_0-2<|A|$.
For every possible $\sigma|_{W_0}$, since
$|S\setminus A|=20D+1$, \eqref{rt:chain} gives
\[
\begin{aligned}
&\sum_{2\le a_1<\cdots<a_D<b_0}
   \Prb\bigl(Z_\sigma(a_i,b_i)=0\text{ for all }i
             \mid\sigma|_{W_0}\bigr)\\
&\quad \quad \le
   \sum_{1\le\ell_1<\cdots<\ell_D\le b_0-2}
   \Prb_{\sigma'}\bigl(P_{\ell_i}=z_{D+1-i}\text{ for all }i
             \mid\sigma|_{W_0}\bigr)
   \le F.
\end{aligned}
\]
Averaging and taking a union bound over the left endpoints,
we obtain, for each fixed $(b_0,b_1,\ldots,b_D)$, a bound
of $F$ on the probability that such zero intervals exist.
A further union bound over the at most $n(20D)^D$
ordered tuples $(b_0,b_1,\ldots,b_D)$ gives
\[
 \Prb(B_2\setminus(B_0\cup B_1))
 \le n(20D)^D F. \qedhere
\]
\end{proof}

\begin{proof}[Proof of Lemma~\ref{lem:blocked-swap-partners}]
We first start with a claim showing that even after applying fixed local permutations and allowing an arbitrary admissible permutation, it is unlikely to have many zero-sum intervals whose near endpoints lie in a short window, while their other endpoints lie outside the window, on the same side, at distinct positions.

\begin{claim}\label{claim:remote-witnesses}
Fix $2\le b<b'\le n-1$ with $b'-b=5D$,
positions $u_1,\ldots,u_D\in [b,b']$, and permutations
$\pi_1,\ldots,\pi_D$ supported on $[b,b']$. Then
\begin{equation}\label{rt:far-right}
 \Prb\left(\begin{array}{c}
       \exists\ x_1<\cdots<x_D,\quad x_i>b',\\
       \exists\text{ admissible }\pi:\quad
       Z_{\sigma\circ\pi\circ\pi_i}(u_i,x_i)=0\enskip \text{for } 1\le i\le D
       \end{array}\right)
 \le2^{26D^2}F.
\end{equation}
The same bound holds for intervals $[x_i,u_i]$ with
$x_1>\cdots>x_D$ and $x_i<b$.
\end{claim}

We will use this claim to control the effect of earlier swaps in our local-repair procedure. In particular, with high probability, only a small number of potential swap partners are blocked, ensuring that we can continue the repair process.

\begin{claimproof}
First fix the far-endpoint vector $\bm x=(x_1,\ldots,x_D)$
and put $I_i=\pi_i([u_i,x_i])$. The equations $Z_{\sigma\circ\pi\circ\pi_i}(u_i,x_i)=0$ depend
on an admissible $\pi$ only through the ordered image tuple
$(\pi(I_1),\ldots,\pi(I_D))$. We bound the number of such tuples
by recording, at each position $v\in [n]$, its $D$ membership indicators.

The membership vectors
$\bigl(\mathbf1_{\{v\in\pi(I_1)\}},\ldots,\mathbf1_{\{v\in\pi(I_D)\}}\bigr)$,
$v\in[n]$, determine the ordered image tuple.
Admissibility of $\pi$ gives $|\pi^{-1}(v)-v|\le5D$.
Since $I_i\subseteq[b,n]$, the vectors are zero for
$v<b-5D$.
The window $[b-5D,b'+5D]\cap[n]$ contains at most
$15D+1$ positions.
Each position has at most $2^D$ possible membership vectors,
so there are at most $(2^D)^{15D+1}=2^{D(15D+1)}$
possible assignments of membership vectors to all positions
in this window.

For $v>b'+5D$, we have $\pi^{-1}(v)>b'$.
Since $\pi_i$ fixes positions beyond $b'$,
$I_i\cap[b'+1,n]=[b'+1,x_i]$, so
$\mathbf1_{\{v\in\pi(I_i)\}}
=\mathbf1_{\{\pi^{-1}(v)\le x_i\}}$.
As $\pi$ varies, this coordinate can change only when
$x_i\in[v-5D,v+5D-1]$, equivalently
$v\in[x_i-5D+1,x_i+5D]$.
For each $i$, this allows at most $10D$ positions.
Thus, among the membership vectors at positions $v>b'+5D$,
at most $10D^2$ binary coordinates can vary as $\pi$ varies.
Therefore, the number of ordered image tuples $(\pi(I_1),\ldots,\pi(I_D))$
is at most
\[
 2^{D(15D+1)}\,2^{10D^2}\le2^{26D^2}.
\]
To summarize, for fixed $\bm x$, the event
$Z_{\sigma\circ\pi\circ\pi_i}(u_i,x_i)=0$ for all $i$
depends on $\pi$ only through
$(\pi(I_1),\ldots,\pi(I_D))$.
There are at most $2^{26D^2}$ distinct events arise as $\pi$
ranges over the admissible permutations.

For every fixed $\pi$, the ordering $\sigma\circ\pi$
has the same distribution as $\sigma$, so each event
has probability
$\Prb(Z_{\sigma\circ\pi_i}(u_i,x_i)=0\text{ for all }i)$.
A union bound over these events and over $\bm x$ gives
\[
\Prb\bigl(\exists\bm x,\text{ admissible }\pi:
       Z_{\sigma\circ\pi\circ\pi_i}(u_i,x_i)=0
       \text{ for all }i\bigr)\le 2^{26D^2}
   \sum_{b'<x_1<\cdots<x_D\le n}
       \Prb\bigl(Z_{\sigma\circ\pi_i}(u_i,x_i)=0
                    \text{ for all }i\bigr).
\]
Now reveal $\sigma(j)$ for every $j\in[b,b']$ and put
$A=S\setminus\sigma([b,b'])$.
Let $\sigma'(1),\ldots,\sigma'(|A|)$ denote the sequence
$\sigma(b'+1),\ldots,\sigma(n),\sigma(1),\ldots,\sigma(b-1)$.
Conditional on $\sigma|_{[b,b']}$, this is a uniform ordering
of $A$.
Put $P_\ell=\sigma'(1)+\cdots+\sigma'(\ell)$ and
$z_i=-\sum_{j=u_i}^{b'}\sigma(\pi_i(j))$.
Since $\pi_i$ fixes positions beyond $b'$, the equation
$Z_{\sigma\circ\pi_i}(u_i,x_i)=0$ is equivalent to
$P_{x_i-b'}=z_i$.
The targets $z_i$ are fixed by the revealed values
and do not depend on $\bm x$.

The prefix lengths satisfy
$1\le x_1-b'<\cdots<x_D-b'\le n-b'<|A|$.
For every possible $\sigma|_{[b,b']}$, since $|S\setminus A|=5D+1$, \eqref{rt:chain} gives
\[
\sum_{b'<x_1<\cdots<x_D\le n}
   \Prb\bigl(P_{x_i-b'}=z_i\text{ for all }i
             \mid\sigma|_{[b,b']}\bigr)=
   \sum_{1\le\ell_1<\cdots<\ell_D\le n-b'}
   \Prb\bigl(P_{\ell_i}=z_i\text{ for all }i
             \mid\sigma|_{[b,b']}\bigr)
   \le F.
\]
By averaging over $\sigma|_{[b,b']}$ and substituting into
the preceding union bound, we complete the proof of \eqref{rt:far-right}. For the left version, we just need to reflect positions by $v\mapsto n+1-v$.
This preserves admissibility and turns the decreasing left endpoints
into increasing right endpoints.
\end{claimproof}

\begin{figure}[htbp]
\centering
\begin{tikzpicture}[x=1cm,y=1cm,font=\small]

\begin{scope}[yshift=5.5cm]

\node[anchor=west,font=\small\bfseries] at (.4,3.1)
  {(a) Blocked swaps produce right or left witnesses};

\node[anchor=east,font=\footnotesize,text=abBlue] at (1.5,2.35) {Right witness};

\fill[abOrange!14] (3.0,2.20) rectangle (8.15,2.48);
\draw[abGrey,line width=.65pt] (1.8,2.34)--(12.0,2.34);

\foreach \x/\lab in {
  3.0/b,4.1/s,5.8/y,8.15/{b'},11.1/t
}{
  \draw[abGrey,line width=.65pt] (\x,2.28)--(\x,2.40);
  \node[above=3pt,font=\footnotesize] at (\x,2.34) {$\lab$};
}

\node[circle,fill=abPurple,inner sep=0pt,minimum size=3.6pt] at (3.0,2.34) {};
\node[circle,fill=abPurple,inner sep=0pt,minimum size=3.6pt] at (5.8,2.34) {};

\draw[abBlue,line width=1.25pt] (4.1,2.02)--(11.1,2.02);
\foreach \x in {4.1,11.1}{
  \draw[abBlue,line width=.9pt] (\x,1.95)--(\x,2.09);
}

\node[font=\scriptsize,text=abGrey] at (9.6,2.65) {$W=[b,b']$};

\node[anchor=east,font=\footnotesize,text=abOrange] at (1.5,1.00) {Left witness};

\fill[abOrange!14] (3.0,.85) rectangle (8.15,1.13);
\draw[abGrey,line width=.65pt] (1.8,.99)--(12.0,.99);

\foreach \x/\lab in {
  2.2/s,3.0/b,4.7/t,5.8/y,8.15/{b'}
}{
  \draw[abGrey,line width=.65pt] (\x,.93)--(\x,1.05);
  \node[above=3pt,font=\footnotesize] at (\x,.99) {$\lab$};
}

\node[circle,fill=abPurple,inner sep=0pt,minimum size=3.6pt] at (3.0,.99) {};
\node[circle,fill=abPurple,inner sep=0pt,minimum size=3.6pt] at (5.8,.99) {};

\draw[abOrange,line width=1.25pt] (2.2,.67)--(4.7,.67);
\foreach \x in {2.2,4.7}{
  \draw[abOrange,line width=.9pt] (\x,.60)--(\x,.74);
}

\end{scope}

\begin{scope}[yshift=1.8cm]

\node[anchor=west,font=\small\bfseries] at (.4,3.2)
  {(b) Distinct far endpoints on the right};

\fill[abOrange!14] (1.9,2.35) rectangle (8.2,2.62);
\draw[abGrey,line width=.65pt] (1.5,2.48)--(12.0,2.48);

\foreach \x/\lab in {
  1.9/b,3.0/{s_i},4.3/{y_i},5.0/{s_j},6.2/{y_j},8.2/{b'},11.2/t
}{
  \draw[abGrey,line width=.65pt] (\x,2.42)--(\x,2.54);
  \node[above=3pt,font=\scriptsize] at (\x,2.48) {$\lab$};
}

\node[anchor=east,font=\footnotesize,text=abBlue] at (1.7,1.95) {$[s_i,b']\setminus\{y_i\}$};
\draw[abBlue,line width=2pt]
  (3.0,1.95)--(4.17,1.95)
  (4.43,1.95)--(8.2,1.95);

\node[anchor=east,font=\footnotesize,text=abOrange] at (1.7,1.35) {$[s_j,b']\setminus\{y_j\}$};
\draw[abOrange,line width=2pt]
  (5.0,1.35)--(6.07,1.35)
  (6.33,1.35)--(8.2,1.35);

\draw[abGrey!70,line width=2pt]
  (8.35,1.95)--(11.2,1.95)
  (8.35,1.35)--(11.2,1.35);

\node[font=\scriptsize,text=abGrey] at (9.8,1.08) {common far tail};

\node[circle,fill=abBlue,inner sep=0pt,minimum size=5.2pt] at (6.2,1.95) {};
\node[circle,draw=abOrange,fill=white,line width=.9pt,inner sep=0pt,minimum size=6.0pt] at (6.2,1.35) {};

\draw[abGrey!55,densely dashed] (6.2,2.30)--(6.2,1.20);

\node[font=\footnotesize] at (6.45,.72)
  {$y_j\in [s_i,b']\setminus\{y_i\},\qquad
    y_j\notin [s_j,b']\setminus\{y_j\}$};

\end{scope}

\end{tikzpicture}

\caption{ (a) If a candidate swap partner $y\in[b+1,b']$ is blocked, then after swapping $b$ and $y$ there is a zero-sum witness interval extending either to the right of $W=[b,b']$ or to the left of $W$.
(b) If two right witnesses had the same far endpoint $t$, cancelling the common far tail would produce equal sums on two distinct subsets of $[b,b']$.}
\label{fig:lemma44}
\end{figure}
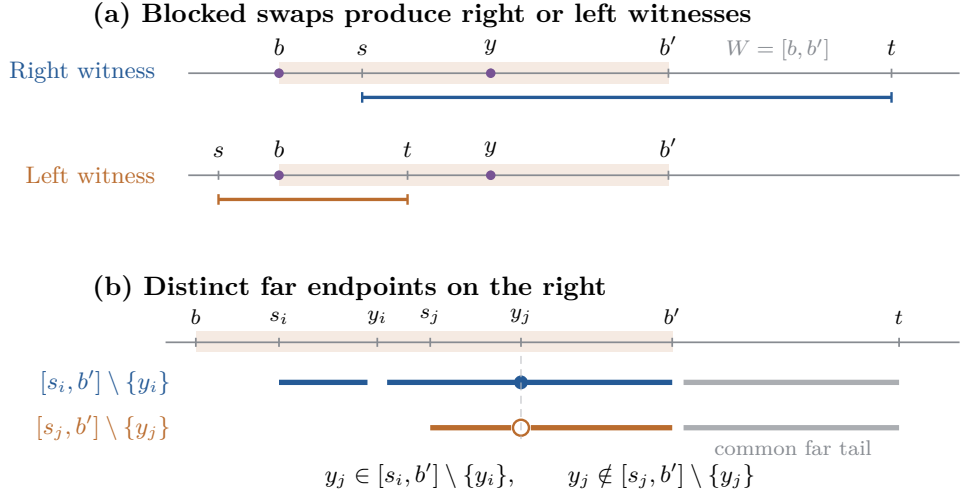

Back to Lemma~\ref{lem:blocked-swap-partners}. Fix an ordering $\sigma$ for which $B_3$ occurs but neither
$B_0$ nor $B_1$ occurs.
By the definition of $B_3$, choose $b\in B$ and an admissible
$\pi$ fixing every position below $b$ such that at least $2D$
positions $y\in[b+1,b+5D]\cap[n]$ are blocked for $\pi$.
Put $b'=b+5D$ and $\tau=\sigma\circ\pi$.
Since $B_1$ does not occur, $b<n-30D$ and $[b,b']\subseteq[n]$.
Moreover, $\pi([b,b'])\subseteq[b,b+10D]\subseteq[b,b+20D]$,
since $\pi$ fixes positions below $b$ and moves each position
by at most $5D$.
The absence of $B_0$ therefore implies that distinct subsets
of $[b,b']$ have distinct $\tau$-sums. In particular, every
nonempty subset has nonzero $\tau$-sum.

A witness $[s,t]$ for a position $y\in[b+1,b']$
blocked for $\pi$ cannot lie wholly in $[b,b']$:
its image under $(b\,y)$ would be a nonempty subset
of $[b,b']$ with $\tau$-sum zero.
Thus every witness is either a \emph{right witness},
with $b<s\le y$ and $t>b'$, or a \emph{left witness},
with $s<b$ and $b\le t<y$.
Choose one witness for each position $y\in[b+1,b']$
blocked for $\pi$.
Since there are at least $2D$ such positions, there are
distinct $y_1,\ldots,y_D$ whose chosen witnesses $[s_i,t_i]$
are all right witnesses or all left witnesses.

For right witnesses, suppose $t_i=t_j=t$ and $y_i<y_j$.
For $h\in\{i,j\}$, the witness equation is
\[
 0=Z_{\tau\circ(b\,y_h)}(s_h,t)
  =\tau(b)+\sum_{v\in[s_h,b']\setminus\{y_h\}}\tau(v)
          +\sum_{v=b'+1}^{t}\tau(v).
\]
Cancelling the common terms gives equal $\tau$-sums on
$[s_i,b']\setminus\{y_i\}$ and $[s_j,b']\setminus\{y_j\}$.
The first subset contains $y_j$, since $s_i\le y_i<y_j\le b'$,
whereas the second does not.
This contradicts the distinctness of subset sums in $[b,b']$,
so the endpoints $t_i$ are distinct.

For left witnesses, suppose $s_i=s_j=s$ and $y_i<y_j$.
For $h\in\{i,j\}$, the witness equation is
\[
 0=Z_{\tau\circ(b\,y_h)}(s,t_h)
  =\sum_{v=s}^{b-1}\tau(v)+\tau(y_h)
                          +\sum_{v=b+1}^{t_h}\tau(v).
\]
Cancelling the common sum $\sum_{v=s}^{b-1}\tau(v)$
gives equal $\tau$-sums on
$[b+1,t_i]\cup\{y_i\}$ and $[b+1,t_j]\cup\{y_j\}$.
The second subset contains $y_j$, whereas the first does not,
since $t_i<y_i<y_j$.
This gives the same contradiction, so the endpoints $s_i$
are also distinct.

Relabel the witnesses and their corresponding positions $y_i$
so that $t_1<\cdots<t_D$ on the right, or
$s_1>\cdots>s_D$ on the left.
Apply Claim~\ref{claim:remote-witnesses} with
$\pi_i=(b\,y_i)$, taking $(u_i,x_i)=(s_i,t_i)$ on the right
and $(u_i,x_i)=(t_i,s_i)$ on the left.
For each fixed $b$ and side, each local endpoint $u_i$
and position $y_i$ has at most $5D$ possible values.
A union bound over the two sides and the at most
$n(5D)^{2D}$ choices of $b$ and these ordered tuples gives
\[
 \Prb(B_3\setminus(B_0\cup B_1))
 \le 2n(5D)^{2D}2^{26D^2}F. \qedhere
\]
\end{proof}

\section{Absorption}
\label{sec:external-absorption}

This section treats the dense-subgroup case (Item (3) of the
affine-core case) described in Subsection~\ref{sub:overview}.
Here almost all elements of $S$ lie in a subgroup in which they
have relatively high density. The anti-concentration estimates
from the preceding sections are not strong enough for our
local-repair criterion. Instead, we adapt the path-and-absorption
method of Bedert et al.~\cite{BBKMM}. The basic idea is to interpret an ordering as a rainbow path
in a Cayley graph: the vertices represent partial sums, and
each edge is coloured by the group element added at that step.
We seek a simple rainbow path using all but at most one
colour of $S$. Our strategy has three main components: (i) we construct long rainbow paths using almost all
colours from a suitable dense subgroup; (ii) we build an exceptional path to accommodate the
elements outside the subgroup, while keeping the different paths disjoint; (iii) we reserve some colours for a final absorption procedure, which incorporates the colours left unused
by the preceding paths. The resulting construction gives the desired valid ordering.
We explain the individual ingredients and their roles below.

\subsection{Definitions and path lemmas}

We represent an ordering by a walk whose edge colours are the ordered
elements. The following terminology describes the paths and random
sets used in the construction.

\begin{definition}\label{def:cayley}
A \emph{path} in $G$ is an ordered tuple $(v_0,\ldots,v_m)$
of distinct elements of $G$, where $m\ge0$.
Its edge $v_{i-1}\to v_i$ has colour $v_i-v_{i-1}$
for $1\le i\le m$.
The path is \emph{rainbow} if its edge colours are pairwise distinct. A \emph{colour word} $(c_1,\ldots,c_\ell)\in G^\ell$
with \emph{anchor} $u\in G$ specifies the walk
$(u,u+c_1,\ldots,u+\sum_{i=1}^{\ell}c_i)$.

For $A\subseteq G$, the Cayley digraph $\operatorname{Cay}_G(A)$
has vertex set $G$ and edges $x\to x+a$ of colour $a$
for $x\in G$ and $a\in A$.
Its paths are precisely the paths in $G$ whose colours lie in $A$.
A \emph{palette} is a specified set of available edge colours.

For disjoint vertex sets $X,Y$ in a digraph $\Gamma$,
let $e_\Gamma(X,Y)$ count the directed edges from $X$ to $Y$.
The digraph has \emph{no $\zeta$-sparse cut} if
$e_\Gamma(X,Y)\ge\zeta|X||Y|$ for every partition
$V(\Gamma)=X\sqcup Y$ into nonempty parts.

A \emph{$p$-random subset} of a finite set includes each
element independently with probability $p$.
Different random subsets are not assumed independent
unless this is stated.
\end{definition}

\begin{remark} 
A rainbow path in $\operatorname{Cay}_G(S)$ using all but
at most one colour of $S$ gives a valid ordering of $S$.
Indeed, let $v_0,\ldots,v_m$ be its vertices in order.
If every colour of $S$ is used, then
$(v_1-v_0,\ldots,v_m-v_{m-1})$ is a valid ordering of $S$.
If one colour $c\in S$ is unused, place it first,
giving the ordering $(c,v_1-v_0,\ldots,v_m-v_{m-1})$.
Its nonempty partial sums are
$c+v_i-v_0$, $0\le i\le m$, which are distinct.
\end{remark}

\begin{definition}\label{def:pairs}
For $g\in G$, a \emph{$g$-pair} is an ordered pair $(a,b)$
of distinct elements of $G$ with $a+b=g$.
A family $\mathcal P$ of $g$-pairs is \emph{disjoint}
if no element occurs in two different pairs. Write
$\operatorname{col}(\mathcal P)
=\bigcup_{(a,b)\in\mathcal P}\{a,b\}$,
also denoted by $\bigcup\mathcal P$.

Let $\mathcal P$ be a disjoint family of $g$-pairs,
and choose a subfamily $\mathcal P_{\rm flex}\subseteq\mathcal P$,
whose members we call \emph{flexible pairs}.
For an integer $0\le\ell\le|\mathcal P_{\rm flex}|$,
a \emph{$(\mathcal P_{\rm flex},\ell)$-absorbing incidence family}
is a sequence
$\mathcal P_1,\ldots,\mathcal P_{|\mathcal P|-\ell}
\subseteq\mathcal P$, each of size at most $40$, such that,
for every $\ell$-element subfamily
$\mathcal P'\subseteq\mathcal P_{\rm flex}$,
one can choose pairwise distinct pairs $p_i\in\mathcal P_i\setminus\mathcal P'$
for $1\le i\le|\mathcal P|-\ell$.
\end{definition}

In short, the flexible part of the construction uses disjoint colour pairs with
a common sum. An incidence family specifies the pairs available at
each position and permits a fixed number of flexible pairs to be
omitted, whichever pairs are chosen.


We will realize these choices by parallel two-edge paths with the same
endpoints. Joining several such structures produces a ``waveform", from
which a rainbow path can later be selected.

\begin{definition}\label{def:waveform}
Let $g\in G\setminus\{0\}$, $v\in G$, and
$E\subseteq G\setminus\{0\}$.
For a nonempty disjoint family $\mathcal Q$ of $g$-pairs
in $E$, the \emph{theta graph} $\Theta(v,\mathcal Q)$
is the union of the two-edge paths
$v\xrightarrow{\,a\,}v+a\xrightarrow{\,b\,}v+a+b=v+g$,
$(a,b)\in\mathcal Q$.

For $m\ge1$, let $\mathcal P_1,\ldots,\mathcal P_m$
be nonempty families of disjoint $g$-pairs with entries
in $E$; different families may overlap.
A \emph{waveform} starting at $u\in G$ consists of
pairwise vertex-disjoint theta graphs
$\Theta(v_i,\mathcal P_i)$, $1\le i\le m$, avoiding $u$,
together with the \emph{connector edges}
$u\to v_1$ and $v_i+g\to v_{i+1}$ for $1\le i<m$.
Their \emph{connector colours} $v_1-u$ and $v_{i+1}-v_i-g$ must be
pairwise distinct elements of
$E\setminus\bigcup_{i=1}^m\operatorname{col}(\mathcal P_i)$.
The waveform ends at $v_m+g$.
\end{definition}

\begin{remark}[Collapsing a waveform]
Consider a waveform formed from a
$(\mathcal P_{\rm flex},\ell)$-absorbing incidence family
$\mathcal P_1,\ldots,\mathcal P_m$.
For any $\ell$-element subfamily
$\mathcal P'\subseteq\mathcal P_{\rm flex}$, choose pairwise
distinct $p_i\in\mathcal P_i\setminus\mathcal P'$.
Retaining only the two-edge path corresponding to $p_i$
in each theta graph $\Theta(v_i,\mathcal P_i)$,
together with all connector edges, \emph{collapses}
the waveform to a rainbow path.
This path uses exactly
$\operatorname{col}(\mathcal P\setminus\mathcal P')$
together with the connector colours.
Figure~\ref{fig:waveform-collapse} illustrates this selection.
\end{remark}

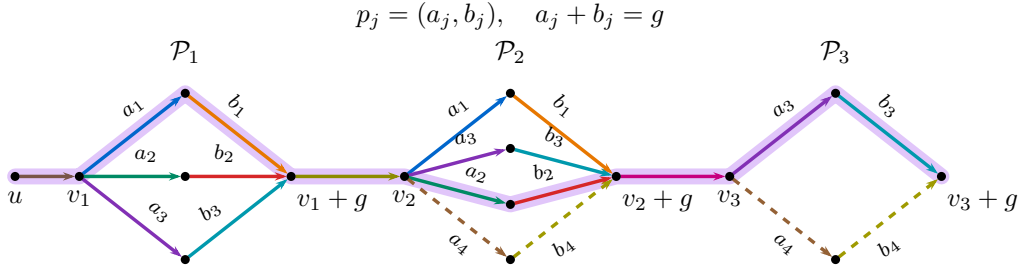
\begin{figure}[htbp]
\centering
\begingroup%
\definecolor{wfAOne}{HTML}{0066CC}%
\definecolor{wfBOne}{HTML}{E77700}%
\definecolor{wfATwo}{HTML}{008A64}%
\definecolor{wfBTwo}{HTML}{D62728}%
\definecolor{wfAThree}{HTML}{8230B5}%
\definecolor{wfBThree}{HTML}{0099AD}%
\definecolor{wfCZero}{HTML}{795548}%
\definecolor{wfCOne}{HTML}{8C8B00}%
\definecolor{wfCTwo}{HTML}{C6007E}%
\definecolor{abGrey}{HTML}{65717A}%
\tikzset{figedge/.append style={line width=1.3pt}}%
\definecolor{wfAFour}{HTML}{946135}
\definecolor{wfBFour}{HTML}{9F9A00}
\definecolor{wfHalo}{HTML}{E1C6FA}
\begin{tikzpicture}[x=1cm,y=1cm,font=\small]
\node at (5.7,2.12) {$p_j=(a_j,b_j),\quad a_j+b_j=g$};
\coordinate (U) at (-.85,0);
\foreach \i/\x in {1/0,2/4.3,3/8.6}{
 \coordinate (L\i) at (\x,0);
 \coordinate (R\i) at (\x+2.8,0);
 \node at (\x+1.4,1.65) {$\mathcal P_{\i}$};
}
\coordinate (A11) at (1.4,1.10);
\coordinate (A12) at (1.4,0);
\coordinate (A13) at (1.4,-1.10);
\coordinate (A21) at (5.7,1.10);
\coordinate (A23) at (5.7,.37);
\coordinate (A22) at (5.7,-.37);
\coordinate (A24) at (5.7,-1.10);
\coordinate (A33) at (10.0,1.10);
\coordinate (A34) at (10.0,-1.10);
\draw[draw=wfHalo,line width=6pt,line join=round,line cap=round]
 (U)--(L1)--(A11)--(R1)--(L2)--(A22)--(R2)--(L3)--(A33)--(R3);
\foreach \l/\m/\r/\j/\ca/\cb in {
 L1/A11/R1/1/wfAOne/wfBOne,
 L1/A12/R1/2/wfATwo/wfBTwo,
 L1/A13/R1/3/wfAThree/wfBThree,
 L2/A21/R2/1/wfAOne/wfBOne,
 L2/A23/R2/3/wfAThree/wfBThree,
 L2/A22/R2/2/wfATwo/wfBTwo,
 L3/A33/R3/3/wfAThree/wfBThree}{
 \draw[figedge,draw=\ca] (\l)--node[pos=.63,sloped,above=1pt,font=\scriptsize] {$a_{\j}$} (\m);
 \draw[figedge,draw=\cb] (\m)--node[pos=.37,sloped,above=1pt,font=\scriptsize] {$b_{\j}$} (\r);
}
\foreach \l/\m/\r in {L2/A24/R2,L3/A34/R3}{
 \draw[figedge,draw=wfAFour,dashed] (\l)--node[pos=.63,sloped,below=1pt,font=\scriptsize] {$a_4$} (\m);
 \draw[figedge,draw=wfBFour,dashed] (\m)--node[pos=.37,sloped,below=1pt,font=\scriptsize] {$b_4$} (\r);
}
\foreach \l/\r/\col in {U/L1/wfCZero,R1/L2/wfCOne,R2/L3/wfCTwo}{
 \draw[figedge,draw=\col] (\l)--(\r);
}
\node[figvertex,label=below:{$u$}] at (U) {};
\foreach \i in {1,2,3}{
 \node[figvertex] at (L\i) {};
 \node[anchor=north,inner sep=0pt,yshift=-5pt] at (L\i) {$v_{\i}$};
 \node[figvertex] at (R\i) {};
 \node[anchor=north west,inner sep=0pt,xshift=2pt,yshift=-5pt] at (R\i) {$v_{\i}+g$};
}
\foreach \v in {A11,A12,A13,A21,A23,A22,A24,A33,A34}{\node[figvertex] at (\v) {};}
\end{tikzpicture}
\endgroup
\caption{A waveform and its collapse, adapted from
\cite{BBKMM}. Each pair in $\mathcal P_i$ gives a two-edge
path from $v_i$ to $v_i+g$. Omitting $p_4$, the highlighted
path selects $p_1,p_2,p_3$ and retains all connector edges.}
\label{fig:waveform-collapse}
\end{figure}

We now explain in more details how these ingredients are used in the proof. We first
establish the construction when all colours belong to a dense subgroup,
and then adapt it to the case where a small number of colours lie outside
that subgroup.
\begin{itemize}
\item (Regularization and robust main paths.)
Lemma~\ref{ext:regularity} allows us to pass to a subgroup $H$ containing
almost all the core colours, where the associated Cayley digraph has no
sparse cut. Lemma~\ref{ext:mainpath} then supplies long rainbow paths
between prescribed vertices, uniformly even after certain reserved
colours have been forbidden.
\vskip .1in
\item  (Flexible pairs and waveforms.) We reserve disjoint $g$-pairs for a common nonzero sum $g$.
The matching template in Lemma~\ref{ext:matching-template} and
Corollary~\ref{ext:flex} produces an absorbing incidence family:
after any $\ell$ flexible pairs are set aside, the remaining pairs
can be assigned bijectively to the theta graphs.
Lemma~\ref{ext:wave} embeds the corresponding waveform using
fresh connector colours and vertices from a reserved random set.
\vskip .1in

\item  (Absorption of unused colours.) The main paths leave a small set $L$ of colours unused.
Proposition~\ref{prop:corrected-pair-absorber} constructs a rainbow
path using these colours together with exactly $\ell$ selected flexible
pairs, omitting at most one colour. Unlike the setting of
\cite{BBKMM}, distinct colours $a$ and $-a$ can occur;
when they are the only two unused colours, a two-stage extension
avoids the repeated vertex that would arise from the usual step.
Lemma~\ref{lem:restricted-pair-absorber} supplies the version needed
when the possible starting vertices lie in only a few $H$-cosets.
\vskip .1in

\item (Dense subgroup and exceptional colours.)
Proposition~\ref{prop:repaired-dense-source} handles the case
in which all colours lie in a subgroup $H$ and occupy at
least three quarters of it.
When exceptional colours remain, Lemma~\ref{lem:fixed-order-exceptional}
incorporates them in a prescribed order, preceding each
by a padding colour from $H$.
The resulting exceptional path joins the relevant cosets.
If it returns to its starting coset, we remove its last
edge and leave that colour for the absorber, allowing
the two main paths to lie in distinct cosets.

\vskip .1in

\item (Final assembly.) In Proposition~\ref{prop:localized-dense-core} we reserve disjoint
vertex sets for the exceptional path, the waveform, the main paths,
and the absorber. After constructing these pieces, we identify the
unused colours, select the flexible pairs needed by the absorber,
and collapse the waveform using all remaining pairs.
The paths can then be concatenated without repeating either vertices
or colours. They use all but at most one colour of $S$, which yields
a valid ordering by the observation after Definition~\ref{def:cayley}.
\end{itemize}

Putting together, our construction is probabilistic in the choice of auxiliary colour and vertex sets, but the final assembly is deterministic. The main probabilistic input is the robust rainbow-path result, Lemma~\ref{ext:mainpath}, taken from~\cite{BBKMM}. For the waveforms and absorbers, Lemma~\ref{ext:wave}, Proposition~\ref{prop:corrected-pair-absorber}, and Lemma~\ref{lem:restricted-pair-absorber} use independent sampling, Chernoff bounds, and union bounds to ensure that sufficiently many local extensions are available simultaneously. The same approach is used for the exceptional paths in Lemma~\ref{lem:fixed-order-exceptional}. Once these properties hold, we fix the sampled sets and assemble the desired rainbow path deterministically.

\subsection{Regularization and robust rainbow paths} The next lemma finds a subgroup containing almost all of $S$
such that the Cayley graph generated by the retained colours
satisfies a quantitative sparse-cut bound.
It is the abelian specialization of
\cite[Corollary~4.1]{BBKMM}.
In the localization argument, we apply it to $S\cap K$ inside
$K$, so the relevant density is $|S\cap K|/|K|$.

\begin{lemma}
[{\cite[Corollary~4.1]{BBKMM}}]\label{ext:regularity}
Let $\varnothing\ne S\subseteq G$ and $0<\epsilon<1/2$.
There is a subgroup $H\le G$ such that
$|S\cap H|\ge(1-\epsilon)|S|$ and
$\operatorname{Cay}_H(S\cap H)$ has no
$\frac{\epsilon}{1000}\left(\frac{|S|}{|G|}\right)^3$-sparse cut.
\end{lemma}

The main path must be chosen after the reserved colours are known.
The next input is therefore uniform over every forbidden colour
set contained in the sampled palette $D$, with both endpoints
prescribed. These are features of \cite[Lemma~5.7]{BBKMM}; we
specialize its forbidden vertex set to the empty set and its density
parameter to $|S|/|G|$.

\begin{lemma}[
{\cite[Lemma~5.7]{BBKMM}}]\label{ext:mainpath}
Let $S\subseteq G\setminus\{0\}$ satisfy
$8|G|^{-1/9500}\le\zeta,\mu,q,|S|/|G|\le1$.
Assume $\operatorname{Cay}_G(S)$ has no $\zeta$-sparse cut
(Definition~\ref{def:cayley}).
Sample a $q$-random vertex set $M \subseteq G$ and, independently, a
$q'$-random subset $D\subseteq S$, where
$q\ge(1+\mu)|S|/|G|$ and $0\le q'\le1-\mu q/4$.
With probability at least $1-7/|G|$, simultaneously for every
$S_F\subseteq D$ and all distinct $u,v\in G$, there is a rainbow
path from $u$ to $v$ with colours in $S\setminus S_F$, omitting
at most $\mu q|G|$ colours of that set, and with all internal vertices
in $M$.
\end{lemma}

We will also use the following standard Chernoff bounds.

\begin{lemma}[Chernoff bounds;
  {\cite[Theorems~A.1.12 and~A.1.13]{AS08}}]
\label{lem:Chernoff}
Let $X$ be a sum of independent Bernoulli variables and $\mu=\Ex X$.
For $0<\theta<1$,
\begin{itemize}
\item $\Prb(X\le(1-\theta)\mu)\le e^{-\theta^2\mu/2}$;
\item $\Prb(X\ge(1+\theta)\mu)\le e^{-\theta^2\mu/3}$.
\end{itemize}
\end{lemma}

\subsection{Matching templates and waveforms}

We need a waveform that can be collapsed after any $\ell$ flexible pairs have been reserved for the absorber. Montgomery's matching template, in the form of \cite[Proposition~9.3]{BBKMM}, provides the required pair choices.

\begin{lemma}
\label{ext:matching-template}
For positive integers $\ell\le t$, there is a bipartite graph with
parts $X\sqcup Y$ and $Z$, where $|X|=t+\ell$, $|Y|=2t$ and
$|Z|=3t$, of maximum degree at most $40$, such that deleting any
$\ell$ vertices of $X$ leaves a perfect matching between the two parts.
\end{lemma}

Interpreting the template vertices as colour pairs gives the incidence
family required in Definition~\ref{def:pairs}.

\begin{corollary}[
{\cite[Corollary~9.5]{BBKMM}}]\label{ext:flex}
Let $1\le\ell\le t$ be integers, and let $\mathcal P$ be a
disjoint family of $3t+\ell$ pairs having the same sum.
There is a subfamily
$\mathcal P_{\rm flex}\subseteq\mathcal P$ of size $t+\ell$
and a $(\mathcal P_{\rm flex},\ell)$-absorbing incidence family
$\mathcal P_1,\ldots,\mathcal P_{3t}$
(Definition~\ref{def:pairs}) satisfying
$\bigcup_{i=1}^{3t}\mathcal P_i=\mathcal P$.
\end{corollary}

\begin{proof}
Identify the pairs of $\mathcal P$ with $X\sqcup Y$ in the matching
template, and let $\mathcal P_{\rm flex}$ correspond to $X$.
For each of the $3t$ vertices of $Z$, let $\mathcal P_i$ be its
neighbourhood. The degree bound gives $|\mathcal P_i|\le40$.
After deleting any $\ell$ pairs of $\mathcal P_{\rm flex}$,
the perfect matching chooses distinct representatives from every
$\mathcal P_i$ and uses every remaining pair. This is precisely the
required incidence family. Moreover $\bigcup_i\mathcal P_i=\mathcal P$:
every pair can be left undeleted in one of these matchings.
\end{proof}

The next lemma embeds the prescribed theta graphs in a
random vertex set and joins them with fresh connector
colours, simultaneously for every starting vertex.

\begin{lemma}[
{\cite[Lemma~9.8]{BBKMM}}]\label{ext:wave}
Let $|G|\ge2$, let $g\in G\setminus\{0\}$, and let
$E\subseteq G\setminus\{0\}$, and let $m\ge1$ be an integer. Let
$\mathcal P_1,\ldots,\mathcal P_m$ be families of disjoint
$g$-pairs in $E$, with $1\le|\mathcal P_i|\le40$; pair families
at different indices may overlap. Suppose $0<p\le1$ and
$|E|p^{42}/2^{18}\ge\max\{m,\log|G|\}$.
For a $p$-random vertex set $R\subseteq G$, with probability at least $1-1/|G|$,
every $u\in G$ is the start of a waveform corresponding to this
sequence (Definition~\ref{def:waveform}), with all its other vertices in $R$.
\end{lemma}

For completeness, we give the proof in
Appendix~\ref{app:waveforms}, following
\cite[Lemma~9.8]{BBKMM}.

\subsection{Pair absorbers}
The next proposition combines at most $\ell$ unused colours
with the colours of $\ell$ selected equal-sum pairs into a
rainbow path, omitting at most one colour in total.
The path starts at any prescribed vertex and has all other
vertices in a random set. In the assembly, the pairs are
selected from the flexible family, so the waveform can be
collapsed while avoiding them.

The proof adapts \cite[Lemma~9.9]{BBKMM}, using a two-stage
extension when the remaining unused colours are $a$ and $-a$
with $a\ne-a$.

\begin{proposition}
\label{prop:corrected-pair-absorber}
Let $|G|\ge2$ and $g\in G$.
Let $0<p\le1$, and let $t\ge40$ and $\ell\ge0$ be integers
satisfying $tp^3\ge\max\{2^{16}\log|G|,512\ell\}$.
Let $\mathcal P$ be a disjoint family of at least
$t+\ell$ $g$-pairs with entries in $G\setminus\{0\}$.
For a $p$-random set $V_T\subseteq G$, with probability
at least $1-1/|G|$, simultaneously for every $v\in G$
and every $L\subseteq G\setminus
(\{0\}\cup\operatorname{col}(\mathcal P))$ with $|L|\le\ell$,
there exist a subfamily $\mathcal P'\subseteq\mathcal P$
with $|\mathcal P'|=\ell$ and a rainbow path starting at $v$,
with all other vertices in $V_T$, using only colours in
$L\cup\operatorname{col}(\mathcal P')$ and omitting
at most one colour of this set.
\end{proposition}

\begin{proof}
Write $\mathcal P=\{(a_j,b_j):1\le j\le|\mathcal P|\}$.
We first establish the extensions needed in the construction.

\begin{claim}
\label{cl:pair-extensions}
With probability at least $1-1/|G|$, the following holds
simultaneously for every $u\in G$.
For each $a\in G\setminus\{0\}$, there are more than $4\ell$
indices $j$ such that the colour words $(a_j,a)$, anchored
at $u$, give rainbow paths whose vertex sets meet only
at $u$ and whose other vertices all lie in $V_T$.
The same holds for the words $(a_j,a,b)$ for every
distinct $a,b\in G\setminus\{0\}$ with $a+b\ne0$.
\end{claim}

\begin{claimproof}
Fix an anchor $u\in G$ and distinct $a,b\in G\setminus\{0\}$
with $a+b\ne0$.
The colour word $(a_i,a,b)$, anchored at $u$, gives a
rainbow path unless $a_i\in\{-a,-a-b,a,b\}$,
excluding at most four pair indices.
Call each resulting path a \emph{candidate}.
Its vertices other than $u$ are
$u+a_i$, $u+a_i+a$, and $u+a_i+a+b$.
Candidates indexed by $i,j$ share a vertex other than $u$
only if
$a_j-a_i\in\{0,a,a+b\}-\{0,a,a+b\}$.
Since this difference set has size at most nine, each
candidate meets at most nine candidates outside $u$,
counting itself.
For the two-edge word $(a_i,a)$, anchored at $u$, the
excluded values are $a_i\in\{-a,a\}$, and each candidate
meets at most four candidates outside $u$.
Thus, in either case, at least $t-4$ candidates remain,
and each meets at most nine candidates outside $u$.
Greedy selection therefore gives at least $(t-4)/9\ge t/40$
candidates whose vertex sets meet only at $u$, since $t\ge40$.

Before sampling $V_T$, fix one such collection for every
$u\in G$ and $a\in G\setminus\{0\}$ in the two-edge case,
and for every $u\in G$ and ordered pair of distinct
$a,b\in G\setminus\{0\}$ with $a+b\ne0$ in the three-edge case.
A candidate \emph{survives} if all its vertices other
than its anchor lie in $V_T$.
Within each collection, the survival events are independent,
each with probability at least $p^3$.
The number of survivors therefore has mean at least $tp^3/40$.
Lemma~\ref{lem:Chernoff} gives at least
$tp^3/80>4\ell$ survivors, except with probability at most
$e^{-tp^3/320}$.
Let $\mathcal E$ be the event that every fixed collection
has more than $4\ell$ survivors.
There are at most $2|G|^3$ collections, so a union bound gives
\[
 \Prb(\mathcal E^c)
 \le 2|G|^3e^{-tp^3/320}
 \le \frac1{|G|}. \qedhere
\]
\end{claimproof}

Fix an outcome satisfying Claim~\ref{cl:pair-extensions},
and take $v\in G$ and $L$ as in the statement.
If $\ell=0$, then $L=\varnothing$; take
$\mathcal P'=\varnothing$ and the single-vertex path at $v$.
Henceforth assume $\ell>0$.
For the construction below, relabel the pairs in their
order of selection.
After selecting $i<\ell$ pairs $(a_1,b_1),\ldots,(a_i,b_i)$
from $\mathcal P$, the path has at most $3i$ edges.
Among the more than $4\ell$ paths supplied by
Claim~\ref{cl:pair-extensions} for the current endpoint
and required colours, at most $i$ use a selected pair
and at most $3i$ meet an earlier vertex of the constructed path.
Thus one remains that uses an unselected pair and meets
the constructed path only at its current endpoint.
At each extension below, choose such a path and denote
its pair by $(a_{i+1},b_{i+1})$.

Start with the single-vertex path at $v$ and put $L_0=L$.
After selecting $(a_1,b_1),\ldots,(a_i,b_i)$, we maintain
that the path uses only colours from
$L\cup\{a_1,b_1,\ldots,a_i,b_i\}$, each at most once,
and that $L_i$ consists of the elements of
$L\cup\{a_1,b_1,\ldots,a_i,b_i\}$
that have not yet been used as edge colours of the path.
Since $\mathcal P$ is disjoint and $L$ avoids all its entries,
$L_i$ contains no entry of any unselected pair.
If there exist distinct $a,b\in L_i$ with $a+b\ne0$,
choose such $a,b$, which are unused colours from
$L\cup\{a_1,b_1,\ldots,a_i,b_i\}$.
Choose the next pair $(a_{i+1},b_{i+1})$ as above,
add the three edges with colours $(a_{i+1},a,b)$
at the current endpoint, and set
$L_{i+1}=(L_i\setminus\{a,b\})\cup\{b_{i+1}\}$.
The maintained conditions are preserved and
$|L_{i+1}|=|L_i|-1$.

If $|L_i|\ge2$ but no such $a,b$ exist, then
$L_i=\{a,-a\}$ with $a\ne-a$. Choose an unselected pair $(a_{i+1},b_{i+1})$ as above,
add the two edges with colours $(a_{i+1},a)$ at the
current endpoint, and set $L_{i+1}=\{b_{i+1},-a\}$.
Since the unselected pair $(a_{i+1},b_{i+1})$ has no entry
in $L\cup\{a_1,b_1,\ldots,a_i,b_i\}$, we have
$b_{i+1}\notin L_i=\{a,-a\}$.
Thus the new unused colours $b_{i+1}$ and $-a$ are
distinct and have nonzero sum.
The above three-edge step then reduces the unused set
to a singleton.
This two-stage repair occurs at most once, so for nonempty
$L$, at most $|L|\le\ell$ pairs suffice to reach a singleton.

If $L=\varnothing$, fix any $a\in\operatorname{col}(\mathcal P)$
and choose a surviving two-edge candidate anchored at $v$.
Denote its pair by $(a_1,b_1)$, so its colour word is $(a_1,a)$.
Select this pair, add only the first edge of the candidate,
and put $L_1=\{b_1\}$.

The unused set is now a singleton.
If $L_i=\{a\}$ and $i<\ell$, choose an unselected pair
$(a_{i+1},b_{i+1})$ as above, add the two edges with colours
$(a_{i+1},a)$ at the current endpoint, and set
$L_{i+1}=\{b_{i+1}\}$.
Repeat this step until exactly $\ell$ pairs have been selected.
Every colour added is new by the maintained conditions,
and the final unused colour is the only omission.

Let $\mathcal P'=\{(a_i,b_i):1\le i\le\ell\}$ be the
subfamily of selected pairs.
Then $|\mathcal P'|=\ell$, and the constructed rainbow path
starts at $v$, has all other vertices in $V_T$, and uses
only colours in $L\cup\operatorname{col}(\mathcal P')$,
omitting at most one colour of this set.
\end{proof}

The following version restricts the starting vertices to
at most two $H$-cosets and the allowed colours to a set
$A\subseteq H$ together with at most one element outside $H$.
Its sampling condition and failure bound depend on
$|H|$ and $|A|$, rather than $|G|$.

\begin{lemma}
\label{lem:restricted-pair-absorber}
Let $H\le G$, $A\subseteq H\setminus\{0\}$, and $g\in H$.
Let $b\in(G\setminus H)\cup\{0\}$, and let $W$ be a union
of at most two $H$-cosets.
Let $0<p\le1$, and let $t\ge40$ and $\ell\ge0$ be integers
satisfying
$tp^3\ge\max\{2^{17}(1+\log(|H||A|)),512\ell\}$.
Let $\mathcal P$ be a disjoint family of at least
$t+\ell$ $g$-pairs with entries in $A$.
For a $p$-random set $V_T\subseteq G$, with probability
at least $1-(|H||A|)^{-2}$, simultaneously for every $v\in W$
and every $L\subseteq(A\cup\{b\})\setminus
(\{0\}\cup\operatorname{col}(\mathcal P))$ with $|L|\le\ell$,
there exist a subfamily $\mathcal P'\subseteq\mathcal P$
with $|\mathcal P'|=\ell$ and a rainbow path starting at $v$,
with all other vertices in $V_T$, using only colours in
$L\cup\operatorname{col}(\mathcal P')$ and omitting
at most one colour of this set.
\end{lemma}

\begin{proof}
Write $\mathcal P=\{(a_j,b_j):1\le j\le|\mathcal P|\}$
and put $U=W\cup(W+b)$, so $|U|\le4|H|$.
We repeat the packing argument from the proof of
Proposition~\ref{prop:corrected-pair-absorber},
restricting the anchors to $U$ and the other colours
to $(A\cup\{b\})\setminus\{0\}$.
Specifically, for each $u\in U$ and
$c\in(A\cup\{b\})\setminus\{0\}$, consider the words
$(a_j,c)$ anchored at $u$, as $j$ varies.
Likewise, for each $u\in U$ and ordered pair of distinct
$c,d\in(A\cup\{b\})\setminus\{0\}$ with $c+d\ne0$,
consider the words $(a_j,c,d)$.
For each such choice, the packing argument gives at least
$t/40$ rainbow paths whose vertex sets meet only at $u$.
Fix one such collection for every choice before sampling $V_T$.

A path \emph{survives} if all its vertices other than its
starting vertex lie in $V_T$.
Within each fixed collection, the survival events are
independent, each with probability at least $p^3$.
The number of surviving paths therefore has mean at least
$tp^3/40$.
Since $tp^3/80>4\ell$, Lemma~\ref{lem:Chernoff}
bounds the probability of at most $4\ell$ surviving paths
by $e^{-tp^3/320}$.
Let $\mathcal E$ be the event that every fixed collection
has more than $4\ell$ surviving paths.
Since there are at most $8|H|(|A|+1)^2$ collections,
\[
 \Prb(\mathcal E^c)
 \le 8|H|(|A|+1)^2e^{-tp^3/320}
 \le (|H||A|)^{-2}.
\]
The last inequality follows from the hypothesis and
$|H|\ge|A|\ge2(t+\ell)\ge80$.

Fix an outcome in $\mathcal E$, and take $v\in W$ and $L$
as in the statement.
If $\ell=0$, then $L=\varnothing$; take
$\mathcal P'=\varnothing$ and the single-vertex path at $v$.
Henceforth assume $\ell>0$.
We use the construction from the proof of
Proposition~\ref{prop:corrected-pair-absorber},
relabelling the pairs in their order of selection.

Every unused colour in that construction lies in
$(A\cup\{b\})\setminus\{0\}$, so the required two-edge
and three-edge words are among those considered above.
This also covers the auxiliary colour chosen when
$L=\varnothing$, since it belongs to
$\operatorname{col}(\mathcal P)\subseteq A$.
All colours other than $b$ lie in $H$, and $b$ is used
at most once.
Thus the current endpoint lies in $v+H\subseteq W$
before $b$ is used and in $v+b+H\subseteq W+b$ afterwards.
In particular, it always belongs to $U$.

After selecting $k<\ell$ pairs
$(a_1,b_1),\ldots,(a_k,b_k)$ from $\mathcal P$,
the constructed path has at most $3k$ edges.
In the fixed collection for its current endpoint and
the required colours, at most $k$ paths use a selected pair.
Since the paths in this collection meet only at their
starting vertex, at most $3k$ meet another vertex of
the constructed path.
As more than $4\ell$ paths survive, one remains that uses
an unselected pair and meets the constructed path only
at its current endpoint.
At each extension, choose such a path and denote its pair
by $(a_{k+1},b_{k+1})$.

The same construction therefore selects exactly $\ell$
pairs and leaves at most one unused colour.
Let $\mathcal P'=\{(a_j,b_j):1\le j\le\ell\}$.
The resulting rainbow path starts at $v$, has all other
vertices in $V_T$, and uses only colours in
$L\cup\operatorname{col}(\mathcal P')$, omitting at most
one colour of this set.
\end{proof}

\subsection{Dense subsets of a subgroup}

For dense colour sets, the following degree count verifies the sparse-cut condition required by Lemma~\ref{ext:mainpath}.
We will also use this observation in the localization argument.

\begin{observation}
\label{obs:dense-cayley-cut}
For $D\subseteq H$ with $|D|>|H|/2$,
$\operatorname{Cay}_H(D)$ has no
$(2|D|/|H|-1)$-sparse cut.
\end{observation}

\begin{proof}
Put $\Gamma=\operatorname{Cay}_H(D)$.
Consider a partition $H=X\sqcup Y$ into nonempty parts,
labelling them so that $|Y|\le |H|/2$.
Each vertex in $Y$ has at least $|D|-|Y|$
out-neighbours in $X$, so
\[
 e_\Gamma(Y,X)
 \ge |Y|(|D|-|Y|)
 \ge (2|D|/|H|-1)|X||Y|.
\]
Since every vertex has equal indegree and outdegree,
$e_\Gamma(X,Y)=e_\Gamma(Y,X)$, proving the assertion.
\end{proof}

The next proposition handles the case in which all colours lie in a subgroup and occupy at least three quarters of it. It is a consequence of \cite[Theorem~9.10]{BBKMM}.

\begin{proposition}[Valid orderings for dense subsets of a subgroup]\label{prop:repaired-dense-source}
There is an absolute constant $C$ such that, for every
group $H$ with $|H|\ge C$, every $S\subseteq H\setminus\{0\}$
with $|S|\ge3|H|/4$ has a valid ordering.
\end{proposition}

We include a proof in Appendix~\ref{app:dense}, following the assembly of \cite[Section~9]{BBKMM} and using Proposition~\ref{prop:corrected-pair-absorber} to handle distinct inverse colours.

\subsection{Localization to subgroup cosets}

We now return to the dense-subgroup case described in
Subsection~\ref{sub:overview}. Almost all elements of $S$
lie in a subgroup $K$, where $S\cap K$ has relatively high
density. Applying Lemma~\ref{ext:regularity} inside $K$
gives a subgroup $H\le K$ in which we construct the main
paths and absorbers. We must also incorporate the
exceptional colours in $S\setminus H$.

The next lemma constructs a path through these exceptional
colours in a prescribed order. Following
\cite[Lemma~6.1]{BBKMM}, we precede each exceptional colour
by a \emph{padding colour} from $H$, giving a colour word
\[
(e_1,x_1,\ldots,e_r,x_r),
\qquad e_i\in H,\quad x_i\notin H.
\]
Fixing the order of the $x_i$ confines the possible starting
vertices at each step to a single $H$-coset. This allows
the probability estimates to depend on the subgroup and
colour-set sizes rather than on $|G|$. We then combine
this path with the preceding ingredients to prove
Proposition~\ref{prop:localized-dense-core}.

\begin{lemma}
\label{lem:fixed-order-exceptional}
Let $H\le G$ and $E\subseteq H\setminus\{0\}$.
Fix distinct $x_1,\ldots,x_r\in G\setminus H$, where $r\ge1$,
and let $0<p\le1$ satisfy
$|E|p^2\ge\max\{10r,32\log(|H|r)\}$.
For a $p$-random vertex set $V_A\subseteq G$, with probability
at least $1-(|H|r)^{-3}$, simultaneously for every $v\in H$,
there exist $e_1,\ldots,e_r\in E$ such that the colour word
$(e_1,x_1,\ldots,e_r,x_r)$, anchored at $v$, gives a rainbow
path with all vertices other than $v$ in $V_A$.
\end{lemma}

\begin{proof}
Fix $1\le i\le r$ and $u\in\sum_{j=1}^{i-1}x_j+H$.
For each $e\in E$, the colour word $(e,x_i)$, anchored at $u$,
gives a rainbow path; call it the candidate indexed by $e$.
Its vertices other than $u$ are $u+e$ and $u+e+x_i$.
These lie in the distinct cosets $u+H$ and $u+x_i+H$,
and the two-element sets are pairwise disjoint as $e$ varies.
A candidate \emph{survives} if both vertices lie in $V_A$.
Thus the number $X_{u,i}$ of surviving candidates has
distribution $\operatorname{Bin}(|E|,p^2)$.

Since $|E|p^2\ge10r$, Lemma~\ref{lem:Chernoff} gives
$\Prb(X_{u,i}<5r)\le e^{-|E|p^2/8}$.
Let $\mathcal E$ be the event that $X_{u,i}\ge5r$
for every such choice of $(u,i)$.
A union bound over the $|H|r$ choices gives
\[
 \Prb(\mathcal E^c)
 \le |H|r\,e^{-|E|p^2/8}
 \le (|H|r)^{-3}.
\]

Fix an outcome in $\mathcal E$ and any $v\in H$.
Start with the single-vertex path at $v$ and put $u_1=v$.
Before step $i$, maintain that the current path is rainbow,
has colour word $(e_1,x_1,\ldots,e_{i-1},x_{i-1})$,
has all vertices other than $v$ in $V_A$, and ends at
$u_i\in\sum_{j=1}^{i-1}x_j+H$.
These conditions hold initially for $i=1$.

On $\mathcal E$, there are at least $5r$
surviving candidates anchored at $u_i$ for step $i$.
Exclude the surviving candidates indexed by
$e_1,\ldots,e_{i-1}$, at most $i-1$ in total.
Also exclude any surviving candidate that meets the current path
outside $u_i$. Distinct candidates meet only at $u_i$, so each
of the $2(i-1)$ other vertices of the current path excludes at
most one candidate. 
In total, at most $3(i-1)<5r$ candidates are excluded,
so a surviving candidate remains.

Choose a remaining candidate indexed by $e_i$ and extend
the current path by the two edges
$u_i\to u_i+e_i\to u_i+e_i+x_i$.
Its two new vertices, $u_i+e_i$ and $u_i+e_i+x_i$, lie in $V_A$ and avoid the existing path.
The colour $e_i$ is new by choice, while $x_i$ differs from
all earlier $x_j$ and every padding colour.
The extended path is therefore rainbow.
Its endpoint is
$u_{i+1}:=u_i+e_i+x_i
 \in\sum_{k=1}^{i}x_k+H$,
which establishes the invariant for the next step.

After $r$ steps, the resulting path has the required colour
word and all vertices other than $v$ in $V_A$.
The construction works for every $v\in H$ on $\mathcal E$.
\end{proof}

The next proposition handles the remaining structural case
of Theorem~\ref{thm:Graham}, in which almost all of $S$ lies
in a subgroup $K$. It gives a valid ordering of the full set
$S$, including the exceptional elements outside $K$.

The required size of $S\cap K$ is polynomial in the reciprocal
of its density in $K$. This allows the core density to decrease
as a sufficiently small power of $|S|$, as may occur in the
structural reduction.

\begin{proposition}[Valid orderings from a dense subgroup core]\label{prop:localized-dense-core}
There is an absolute constant $C$ such that the following holds.
Let $K\le G$ and $S\subseteq G\setminus\{0\}$ with
$S\cap K\ne\varnothing$, and put $\delta=|S\cap K|/|K|$.
If $|S\setminus K|\le2^{-30}|S\cap K|$ and
$|S\cap K|\ge C\delta^{-28500}$, then $S$ has a valid ordering.
\end{proposition}

\begin{proof}
We first regularize $S\cap K$ and prepare the random colour
and vertex sets. We then obtain main paths inside subgroup
cosets, an exceptional path, a waveform and a pair absorber.
Three claims establish these inputs and their assembly into
a rainbow path using all but at most one colour of $S$.

Apply Lemma~\ref{ext:regularity} to $S\cap K$ in $K$
with $\epsilon=2^{-35}$, and let $H\le K$ be the resulting
subgroup. Put $A_0=S\cap H$, $a_0=|A_0|$,
$J=S\setminus H$, $r=|J|$, and $h=|H|$.
Then $a_0\ge(1-2^{-35})|S\cap K|$,
$r\le2^{-29}a_0$, and $a_0/h\ge\delta/2$.
Moreover, $\operatorname{Cay}_H(A_0)$ has no
$\zeta$-sparse cut, where $\zeta=2^{-35}\delta^3/1000$.
Put $p=1/32$, $q=29/32$, $c=2^{-240}$,
$t=\lfloor c\delta a_0\rfloor$,
$\ell=\lfloor ct\rfloor$, and $\mu=c^3\delta^2$.
For sufficiently large absolute $C$, these parameters satisfy
\begin{equation}\label{eq:core-capacity}
 \begin{gathered}
 8h^{-1/9500}\le\min\{\zeta,\mu,\delta/8,q\},\quad t\ge40,\\
 tp^3\ge\max\{2^{17}(1+\log(ha_0)),512\ell\},\quad
 2\mu qh+1\le\ell.
 \end{gathered}
\end{equation}

If $J=\varnothing$ and $a_0>3h/4$,
Proposition~\ref{prop:repaired-dense-source} applies.
Henceforth exclude this case.
If $a_0\le3h/4$, put $A_1=A_0$ and $A_2=\varnothing$;
call this the \emph{unsplit case}.
Otherwise, assign each element of $A_0$ to $A_1$ only,
$A_2$ only, or both, with probabilities $1/4,1/4,1/2$,
independently for different elements;
call this the \emph{split case}.
In both cases, $A_1\cup A_2=A_0$.
In the split case, call the palettes \emph{good} if
\[
 5a_0/8\le |A_i|\le5h/6
 \quad\text{and}\quad
 |A_i|\ge17h/32
 \qquad(i=1,2).
\]
In the unsplit case, regard the palette $A_1$ as good.

Assign each element of $A_1$
to $E$, to $S'$, or to neither, with probabilities
$1/4,1/4,1/2$, independently for different elements.
The disjoint sets $E,S'\subseteq A_1$ supply the colours
for the exceptional path and the waveform, respectively.
In the unsplit case, enlarge $E\cup S'$ to a $5/6$-random
subset of $A_1$ by including each remaining element
independently with probability $2/3$.
If $r>0$, fix an ordering $x_1,\ldots,x_r$ of $J$
and put $b=x_r$; otherwise put $b=0$.
Independently of these colour choices, form a partition
$G=V_A\sqcup V_R\sqcup V_M\sqcup V_T$ by assigning each
vertex to the four parts with probabilities $p,p,q,p$,
independently for different vertices.
Put
\[
 \mathcal C=\{H,\Sigma(J)+H,\Sigma(J)-b+H\},\quad
 D_1=(A_1\cap A_2)\cup E\cup S',\quad
 D_2=(E\cup S')\cap A_2.
\]
The family $\mathcal C$ contains at most three cosets.
The next claim provides paths after any choice of forbidden
colours from $D_1$ or $D_2$.

\begin{claim}[Robust main paths in the relevant cosets]\label{cl:localized-palettes}
With probability at least $1-48/h$, the following hold simultaneously.
\begin{itemize}
\item The palettes are good.
\item $|E|,|S'|\ge a_0/8$.
\item For each nonempty palette $A_i$, every coset $\Gamma\in\mathcal C$,
all distinct $u,v\in\Gamma$, and every $F\subseteq D_i$,
there is a rainbow path from $u$ to $v$, with all internal
vertices in $V_M\cap\Gamma$, using only colours in
$A_i\setminus F$ and omitting at most $\mu qh$ colours
of this set.
\end{itemize}
\end{claim}

\begin{claimproof}
In the split case, each $A_i$ is a $3/4$-random subset of $A_0$.
Lemma~\ref{lem:Chernoff} bounds the probability that the
palettes are not good by $6e^{-a_0/1152}$.
For the condition $|A_i|\ge17h/32$, note that the mean
is at least $18h/32$.
By Observation~\ref{obs:dense-cayley-cut},
$\operatorname{Cay}_H(A_i)$ has no $1/16$-sparse cut
whenever $|A_i|\ge17h/32$, and hence no $\zeta$-sparse cut.
In the unsplit case, the latter property follows directly
from the choice of $H$.

In the split case, conditional on $A_2$, each of its elements
belongs to $A_1$ with probability $2/3$ and then to $E\cup S'$
with probability $1/2$.
Thus $D_2$ is a $1/3$-random subset of $A_2$.
Conditional on $A_1$, the set $A_1\cap A_2$ is $2/3$-random
and independent of the $1/2$-random set $E\cup S'$,
so $D_1$ is a $5/6$-random subset of $A_1$.
In each calculation, the membership decisions are independent
for different elements.
In the unsplit case, the enlarged superset of $D_1$
is $5/6$-random in $A_1$.

For each application of Lemma~\ref{ext:mainpath}, condition
only on the palette $A_i$ to which it is applied.
For every possible $A_i$ satisfying the required size bounds,
its Cayley digraph has no $\zeta$-sparse cut, and its random
set of deletable colours has the distribution just described.
Translate each $\Gamma\in\mathcal C$ to $H$;
the corresponding restriction of $V_M$ is $q$-random
and independent of all colour choices.
Apply the lemma with parameters $\zeta,\mu,q$ and
$q'=5/6$ for $A_1$, or $q'=1/3$ for $A_2$.
The lower bounds follow from \eqref{eq:core-capacity}
and $|A_i|/h\ge\delta/8$.
The remaining conditions follow from
$(1+\mu)5/6\le q$ and $5/6\le1-\mu q/4$.
In the unsplit case, use the enlarged superset of $D_1$;
the conclusion then holds for every $F\subseteq D_1$.

For each palette and coset, the conditional failure
probability is at most $7/h$.
Averaging and taking a union bound over at most two palettes
and three cosets shows that the probability that the palettes
are good but some required path does not exist is at most $42/h$.
Conditional on any palettes of the required sizes,
$|E|$ and $|S'|$ each have mean at least $5a_0/32$.
Lemma~\ref{lem:Chernoff} therefore bounds the probability
that either is below $a_0/8$ by $2e^{-a_0/320}$.
Thus the probability that any assertion of the claim fails
is at most
\[
 6e^{-a_0/1152}+2e^{-a_0/320}+\frac{42}{h}
 \le\frac{48}{h},
\]
for $h\le2a_0^2$.
\end{claimproof}

Whenever the palettes are good and $|E|,|S'|\ge a_0/8$,
make the following choices using only the colour sets
$A_1,A_2,E,S'$, without using the vertex partition.
Among the $\binom{|S'|}{2}$ unordered pairs from $S'$,
at most $|S'|/2$ have sum zero.
By averaging, some $g\in H\setminus\{0\}$ is the sum of
at least
$|S'|(|S'|-2)/2h \ge \delta a_0/1000>3t+\ell$ such pairs.
These pairs are disjoint, since each element determines
its partner; orient them arbitrarily to obtain $g$-pairs.
Choose a subfamily $\mathcal P$ of exactly $3t+\ell$ such pairs.
Since $1\le\ell\le t$, Corollary~\ref{ext:flex} supplies
$\mathcal P_{\rm flex}\subseteq\mathcal P$ of size $t+\ell$
and a $(\mathcal P_{\rm flex},\ell)$-absorbing incidence family
$\mathcal P_1,\ldots,\mathcal P_{3t}$ with $\bigcup_{i=1}^{3t}\mathcal P_i=\mathcal P$.
Put $W=(\Sigma(J)+H)\cup(\Sigma(J)-b+H)$;
when $r=0$, this is $H$.

\begin{claim}[Exceptional paths, waveforms, and absorbers]\label{cl:localized-reserved}
Conditional on the colour assignments, if the palettes are
good and $|E|,|S'|\ge a_0/8$, then the following hold
simultaneously with probability at least $1-4/h$
over the vertex partition.
\begin{enumerate}[label=(\roman*)]
\item If $r>0$, then for every $v\in H$ there exist
      $e_1,\ldots,e_r\in E$ such that the colour word
      $(e_1,x_1,\ldots,e_r,x_r)$, anchored at $v$, gives
      a rainbow path with all other vertices in $V_A$.
\item Every $u\in W$ starts a waveform formed from
      $\mathcal P_1,\ldots,\mathcal P_{3t}$, with connector
      colours in $S'$ and all other vertices in $V_R$.
\item For every $y\in W$ and every
      $L\subseteq(A_0\cup\{b\})\setminus
      (\{0\}\cup\operatorname{col}(\mathcal P_{\rm flex}))$
      with $|L|\le\ell$, there exist a subfamily
      $\mathcal P'\subseteq\mathcal P_{\rm flex}$ with
      $|\mathcal P'|=\ell$ and a rainbow path starting at $y$,
      with all other vertices in $V_T$, using only colours in
      $L\cup\operatorname{col}(\mathcal P')$ and omitting
      at most one colour of this set.
\item $V_A\cap H\ne\varnothing$ and $|V_M\cap\Gamma|\ge2$
      for every $\Gamma\in\mathcal C$.
\end{enumerate}
\end{claim}

\begin{claimproof}
Suppose $r>0$.
Lemma~\ref{lem:fixed-order-exceptional} applies to
$x_1,\ldots,x_r$, since
$|E|p^2\ge a_0/8192\ge10r$ and
$a_0/8192\ge32\log(hr)$.
The logarithmic bound follows from \eqref{eq:core-capacity},
using $t\le a_0$ and $r\le a_0$.
Thus the probability that (i) fails is at most $(hr)^{-3}$;
there is no failure event in (i) when $r=0$.

Apply Lemma~\ref{ext:wave} in $H$ with colour set $S'$
and families $\mathcal P_1,\ldots,\mathcal P_{3t}$,
translating each coset of $W$ to $H$.
All waveform colours belong to $S'\subseteq H$,
so the waveform remains in its starting coset.
The required size condition follows from
\[
 \frac{|S'|p^{42}}{2^{18}}
 \ge2^{-231}a_0\ge\max\{3t,\log h\},
\]
using $t\le2^{-240}\delta a_0$ and \eqref{eq:core-capacity}.
Since $W$ contains at most two cosets, the probability
that (ii) fails is at most $2/h$.

Apply Lemma~\ref{lem:restricted-pair-absorber} with $A=A_0$,
$b$, $W$, and $\mathcal P_{\rm flex}$.
Its parameter conditions are included in \eqref{eq:core-capacity},
so the probability that (iii) fails is at most $(ha_0)^{-2}$.
The probability that (iv) fails is at most
$(1-p)^h+3(1+h)(1-q)^{h-1}$.

These estimates condition only on the fixed colour choices
and use the marginal distribution of the relevant vertex part.
For sufficiently large $C$, the bounds in (i) and (iii)
are each at most $1/2h$, as is each of the two terms
in the bound for (iv).
Together with the bound $2/h$ for (ii), a union bound shows
that the probability of any failure is at most $4/h$.
\end{claimproof}

The conditional failure bound in
Claim~\ref{cl:localized-reserved} is uniform over all colour
choices with good palettes and $|E|,|S'|\ge a_0/8$.
Averaging this bound over the colour choices and combining
it with Claim~\ref{cl:localized-palettes}, a union bound
shows that all required conclusions hold with probability
at least $1-52/h>0$.
Fix an outcome for which they all hold.

\begin{figure}[htbp]
\centering
\begingroup%
\definecolor{origWave1}{HTML}{850B00}%
\definecolor{origWave2}{HTML}{00BCFF}%
\definecolor{origWave3}{HTML}{FF90BC}%
\definecolor{origWave4}{HTML}{2C370B}%
\definecolor{origWave5}{HTML}{F400F4}%
\definecolor{origWave6}{HTML}{9B6443}%
\definecolor{origWave7}{HTML}{21C800}%
\definecolor{origWave8}{HTML}{212C59}%
\definecolor{origWave9}{HTML}{A6596F}%
\definecolor{origPairA}{HTML}{0B7A00}%
\definecolor{origPairB}{HTML}{B1B17A}%
\tikzset{
 origWaveEdge/.style={figedge,line width=1.05pt},
 origWaveGhost/.style={figedge,line width=.72pt,dash pattern=on 2.2pt off 1.6pt},
 origWaveUnused/.style={circle,draw=abGrey!60,fill=white,inner sep=0pt,minimum size=3.4pt}
}%
\begin{tikzpicture}[x=1cm,y=1cm,font=\small]
\foreach \lo/\hi/\col in {0.000000/3.133333/abOrange,3.133333/6.266667/abPurple,6.266667/10.366667/abBlue,10.366667/13.500000/abRed}{
 \fill[\col!5] (\lo,.5) rectangle (\hi,1.7);
 \fill[\col!5] (\lo,-1.7) rectangle (\hi,-.5);
}
\draw[abGrey!55,rounded corners=3pt] (0.000000,0.500000) rectangle (13.500000,1.700000);
\draw[abGrey!55,rounded corners=3pt] (0.000000,-1.700000) rectangle (13.500000,-0.500000);
\foreach \x in {3.133333,6.266667,10.366667}{
 \draw[abGrey!45] (\x,.5)--(\x,1.7);
 \draw[abGrey!45] (\x,-1.7)--(\x,-.5);
}
\foreach \x/\part in {1.566667/A,4.700000/R,8.316667/M,11.933333/T}{\node at (\x,2.05) {$V_{\part}$};}
\node[anchor=east] at (-0.223810,1.100000) {$H$};
\node[anchor=east] at (-0.223810,-1.100000) {$-b+H$};
\coordinate (S) at (9.816667,1.100000);
\coordinate (V) at (7.016667,1.100000);
\coordinate (A) at (0.503571,1.100000);
\coordinate (U) at (1.790476,-1.100000);
\coordinate (W) at (6.053333,-1.100000);
\coordinate (Y) at (9.766667,-1.100000);
\coordinate (O) at (2.685714,1.100000);
\draw[figpath,draw=abBlue,decorate,decoration={snake,amplitude=.5mm,segment length=3mm,post length=2mm}]
 (S)--(V);
\node[abBlue] at (8.366667,1.480000) {$P_2$};
\draw[figedge,draw=abOrange] (V)..controls (6.266667,1.500000) and (1.286905,1.500000)..(A);
\draw[figpath,draw=abOrange,decorate,decoration={snake,amplitude=.5mm,segment length=3mm,post length=2mm}]
 (A)..controls (0.503571,0.100000) and (0.503571,-0.650000)..(U);
\node[fill=white,inner sep=1pt] at (0.615476,0.000000) {$\vdots$};
\node[text=abOrange,font=\footnotesize] at (1.589048,1.560000) {exceptional path};
\draw[figedge,draw=abGrey!70,dashed] (U)..controls (3.133333,-0.600000) and (3.133333,0.550000)..(O);
\node[abGrey,font=\footnotesize,anchor=east] at (2.931905,0.220000) {$b$ deleted};
\draw[abGrey,line width=.7pt] (2.786429,-0.100000)--(2.976667,0.070000) (2.786429,0.070000)--(2.976667,-0.100000);
\node[circle,draw=abGrey,fill=white,inner sep=0pt,minimum size=3.4pt] at (O) {};
\definecolor{wfHalo}{HTML}{E1C6FA}
\coordinate (CW1L) at (3.403333,-1.1);
\coordinate (CW1R) at (4.353333,-1.1);
\coordinate (CW1T) at (3.878333,-.64);
\coordinate (CW1C) at (3.878333,-1.1);
\coordinate (CW1B) at (3.878333,-1.56);
\coordinate (CW2L) at (5.103333,-1.1);
\coordinate (CW2R) at (6.053333,-1.1);
\coordinate (CW2T) at (5.578333,-.64);
\coordinate (CW2Ca) at (5.578333,-.946667);
\coordinate (CW2Cb) at (5.578333,-1.253333);
\coordinate (CW2B) at (5.578333,-1.56);
\draw[draw=wfHalo,line width=4.8pt,line join=round,line cap=round]
 (U)--(CW1L)--(CW1T)--(CW1R)--(CW2L)--(CW2B)--(CW2R);
\foreach \l/\m/\r/\ca/\cb in {
 CW1L/CW1C/CW1R/origWave5/origWave6,
 CW1L/CW1B/CW1R/origPairA/origPairB,
 CW2L/CW2T/CW2R/origWave2/origWave3,
 CW2L/CW2Ca/CW2R/origWave8/origWave9,
 CW2L/CW2Cb/CW2R/origPairA/origPairB}{
 \draw[origWaveGhost,draw=\ca] (\l)--(\m);
 \draw[origWaveGhost,draw=\cb] (\m)--(\r);
}
\draw[origWaveEdge,draw=origWave1] (U)--(CW1L);
\draw[origWaveEdge,draw=origWave2] (CW1L)--(CW1T);
\draw[origWaveEdge,draw=origWave3] (CW1T)--(CW1R);
\draw[origWaveEdge,draw=origWave4] (CW1R)--(CW2L);
\draw[origWaveEdge,draw=origWave5] (CW2L)--(CW2B);
\draw[origWaveEdge,draw=origWave6] (CW2B)--(CW2R);
\foreach \v in {CW1L,CW1R,CW1T,CW2L,CW2B}{\node[figvertex] at (\v) {};}
\foreach \v in {CW1C,CW1B,CW2T,CW2Ca,CW2Cb}{\node[origWaveUnused] at (\v) {};}
\node[font=\scriptsize,fill=abPurple!5,inner sep=.4pt] at (4.728333,-1.1) {$\cdots$};
\node[abPurple] at (4.700000,-.28) {$P_R$};
\node[font=\footnotesize] at (4.700000,-2.10) {$S_F\setminus\operatorname{col}(\mathcal P^\prime)$};
\coordinate (M) at (6.666667,-1.100000);
\draw[figedge,draw=abBlue] (W)--(M);
\draw[figpath,draw=abBlue,decorate,decoration={snake,amplitude=.5mm,segment length=3mm,post length=2mm}]
 (M)--(Y);
\node[abBlue] at (8.316667,-0.730000) {$P_1$};
\coordinate (T0) at (10.758333,-1.100000);
\coordinate (T1) at (11.541667,-1.100000);
\coordinate (T2) at (12.499074,1.100000);
\coordinate (T3) at (13.108333,1.100000);
\draw[figedge,draw=origPairA] (Y)--node[above=1pt,font=\scriptsize] {$f_1$}(T0);
\draw[figpath,draw=abRed,decorate,decoration={snake,amplitude=.5mm,segment length=3mm,post length=2mm}]
 (T0)--(T1);
\draw[figedge,draw=abRed] (T1)--node[right=2pt] {$b$} (T2);
\draw[figpath,draw=origPairB,decorate,decoration={snake,amplitude=.5mm,segment length=3mm,post length=2mm}]
 (T2)--node[above=1pt,font=\scriptsize] {$f_2$}(T3);
\node[abRed,font=\footnotesize] at (11.672222,-1.500000) {absorber};
\foreach \p in {S,A,M,T0,T1,T2,T3}{\node[figvertex] at (\p) {};}
\node[figvertex,label=below:{$v$}] at (V) {};
\node[figvertex,label=below:{$u$}] at (U) {};
\node[figvertex,label=below:{$w$}] at (W) {};
\node[figvertex,label=below:{$y$}] at (Y) {};
\node[font=\footnotesize] at (11.933333,-2.10) {$L\cup\operatorname{col}(\mathcal P^\prime)$};
\node[font=\scriptsize,text=abGrey] at (11.933333,-2.40) {at most one omitted};
\end{tikzpicture}
\endgroup
\caption{Dense-core assembly with $b$ used by the absorber, adapted from
\cite{BBKMM}. The highlighted path is $P_R$. The green and olive
colours are left for the absorber; colours on the other path
segments are schematic.}
\label{fig:dense-core-assembly}
\end{figure}
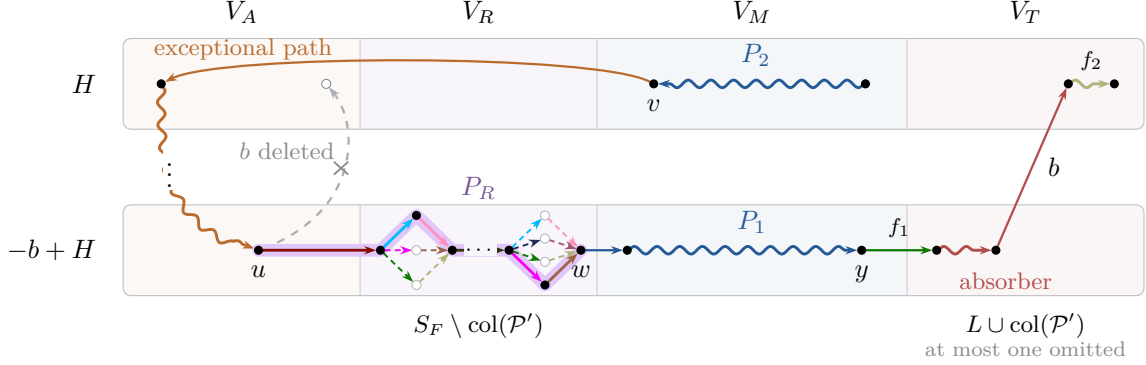

\begin{claim}[Assembly into a simple rainbow path]\label{cl:localized-assembly}
There is a rainbow path using all but at most one colour of $S$.
\end{claim}

\begin{claimproof}
Figure~\ref{fig:dense-core-assembly} illustrates the split case
with $\Sigma(J)\in H$.
In the split case, choose two distinct vertices in $V_M\cap H$.
Claim~\ref{cl:localized-palettes}, applied to $A_2$ with
forbidden set $D_2$, gives a rainbow path $P_2$ between them,
with all vertices in $V_M\cap H$, using only colours in
$A_2\setminus(E\cup S')$ and omitting at most $\mu qh$
colours of this set.
Let $v$ be its final vertex.
In the unsplit case, omit $P_2$ and choose $v\in V_A\cap H$.
These choices are possible by Claim~\ref{cl:localized-reserved}(iv).

If $r>0$, Claim~\ref{cl:localized-reserved}(i) gives a rainbow
path starting at $v$ with colour word
$(e_1,x_1,\ldots,e_r,x_r)$, where $e_i\in E$,
and all other vertices in $V_A$.
Its endpoint lies in $\Sigma(J)+H$.
If $\Sigma(J)\in H$, delete its last edge, of colour $b=x_r$;
the new endpoint then lies in $\Sigma(J)-b+H$,
which is disjoint from $H$.
Otherwise keep the whole path.
Write $P_A$ for the resulting path and $u$ for its endpoint.
When $r=0$, take $P_A$ to be the single-vertex path at $v$
and put $u=v$.
Then $\Gamma:=u+H$ belongs to $\mathcal C$ and is contained in $W$.
In the split case, $J\ne\varnothing$, and our choice ensures
that $\Gamma\cap H=\varnothing$.

By Claim~\ref{cl:localized-reserved}(ii), choose a waveform
starting at $u$, formed from $\mathcal P_1,\ldots,\mathcal P_{3t}$,
with all other vertices in $V_R$.
All its colours lie in $S'\subseteq H$, so it lies in $\Gamma$.
Let $w\in V_R\cap\Gamma$ be its endpoint and let $S_F\subseteq S'$
be the set of all its edge colours.
This set contains $\operatorname{col}(\mathcal P)$ and
all connector colours.

Choose $y\in V_M\cap\Gamma$, using
Claim~\ref{cl:localized-reserved}(iv).
Write $\operatorname{col}(Q)$ for the edge-colour set of a path $Q$,
and put $\operatorname{col}(P_2)=\varnothing$ when $P_2$ is absent.
Let
\[
 F=S_F\cup\bigl[A_1\cap
       (\operatorname{col}(P_2)\cup\operatorname{col}(P_A))\bigr].
\]
Since $S_F\subseteq S'$, $A_1\cap\operatorname{col}(P_2)
\subseteq A_1\cap A_2$, and
$A_1\cap\operatorname{col}(P_A)\subseteq E$, we have $F\subseteq D_1$.
Claim~\ref{cl:localized-palettes} gives a rainbow path $P_1$
from $w$ to $y$, with all internal vertices in $V_M\cap\Gamma$,
using only colours in $A_1\setminus F$ and omitting at most
$\mu qh$ colours of this set.

Put
\[
 L=S\setminus\bigl(S_F\cup\operatorname{col}(P_1)
       \cup\operatorname{col}(P_2)\cup\operatorname{col}(P_A)\bigr).
\]
Since $A_0=A_1\cup A_2$ and $D_2\subseteq A_1$,
\[
 L\cap A_0\subseteq
 \bigl[A_1\setminus(F\cup\operatorname{col}(P_1))\bigr]
 \cup\bigl[A_2\setminus(D_2\cup\operatorname{col}(P_2))\bigr].
\]
Each set on the right has size at most $\mu qh$,
and $L\setminus A_0\subseteq\{b\}$.
Thus $|L|\le2\mu qh+1\le\ell$.
Moreover, $\operatorname{col}(\mathcal P_{\rm flex})\subseteq S_F$,
so $L\subseteq(A_0\cup\{b\})\setminus
(\{0\}\cup\operatorname{col}(\mathcal P_{\rm flex}))$.

Since $y\in W$, Claim~\ref{cl:localized-reserved}(iii)
supplies a subfamily $\mathcal P'\subseteq\mathcal P_{\rm flex}$
with $|\mathcal P'|=\ell$ and a rainbow path $P_T$ starting at $y$,
with all other vertices in $V_T$, using only colours in
$L\cup\operatorname{col}(\mathcal P')$ and omitting at most
one colour of this set.
By the absorbing property, choose pairwise distinct
$p_i\in\mathcal P_i\setminus\mathcal P'$ for $1\le i\le3t$.
Retaining the corresponding two-edge path in each theta graph
and all connector edges collapses the waveform to a rainbow
path $P_R$ from $u$ to $w$.
Its colours are exactly $S_F\setminus\operatorname{col}(\mathcal P')$.
Consequently, $P_2,P_A,P_R,P_1,P_T$ have disjoint colour sets
and together use all but at most one colour of $S$,
omitting $P_2$ in the unsplit case.

It remains to check that joining these paths in this order
introduces no repeated vertex.
All vertices of $P_A$ other than its starting vertex lie in $V_A$,
all vertices of $P_R$ other than $u$ lie in $V_R$,
and all vertices of $P_T$ other than $y$ lie in $V_T$.
All vertices of $P_1$ other than $w$ lie in $V_M\cap\Gamma$.
In the split case, $P_2$ lies in $V_M\cap H$, and
$\Gamma\cap H=\varnothing$ separates it from $P_1$.
In the unsplit case, $P_2$ is absent and $P_A$ starts in $V_A$.
Thus consecutive paths meet only at their common endpoint,
and nonconsecutive paths are disjoint.
Their concatenation is the required rainbow path.
\end{claimproof}

The observation following Definition~\ref{def:cayley}
now gives a valid ordering of $S$.
\end{proof}

\section{Proof of Theorem~\ref{thm:Graham}}\label{sec:main-proof}

We apply the dichotomy of Lemma~\ref{lem:filtration}.
The reservoir case and the first two affine-core subcases
provide the uniform estimates required for local repair.
The remaining subcase is handled by
Proposition~\ref{prop:localized-dense-core}.

\begin{proof}[Proof of Theorem~\ref{thm:Graham}]
Fix $\varepsilon=10^{-5}$, $D=400000$, and
$c=2^{-50}/(D+1)$.
Write $n=|S|$ and assume throughout that $n$ is sufficiently
large. Put $r=\lceil n^\varepsilon\rceil$ and
$\tau=c/(r+1)\ge c/(3n^\varepsilon)$.

To apply Proposition~\ref{prop:repair-full}, we require
bounds $F_1(A;z)\le q$ and $F_D(A;\boldsymbol z)\le F$
that hold for every $A\subseteq S$ with
$|S\setminus A|\le20D+1$ and every
$z\in G$, $\boldsymbol z\in G^D$.
Below, $A$ always denotes such a retained set.
By Lemma~\ref{lem:filtration}, at least one of the following
two cases holds.

\smallskip\noindent\emph{Case 1.}
$S$ contains an $r$-element set $P$ with distinct subset sums
at each fixed size.

For every retained set $A$, we have $|A|\ge n/2$ and
$|P\cap A|\ge r-(20D+1)\ge n^\varepsilon/2$.
The set $P\cap A$ also has distinct subset sums at each
fixed size.
Lemmas~\ref{lem:reservoir-prefixes}
and~\ref{lem:reservoir-chains}, applied to $A$ and $P\cap A$,
give the required bounds with
\[
 q=\frac{26\log(2n)}{n^\varepsilon},\quad
 F=(D+1)
 \left(\frac{52(D+1)\log(2n)}{n^\varepsilon}\right)^D
 +D n^D e^{-n^\varepsilon/[16(D+1)]}.
\]
Since $\varepsilon D=4$, we have $q+nF=o(1)$.
Proposition~\ref{prop:repair-full} therefore gives
a valid ordering of $S$.

\smallskip\noindent\emph{Case 2.}
There are $H\le G$ and $a\in G$ such that
$T:=S\cap(a+H)$ satisfies $|T|\ge(1-c)n$ and
$T-a$ is $\tau$-balanced in $H$.

Put $d=\ord(a+H)$ in $G/H$.
We distinguish three subcases.

\smallskip\noindent\emph{Case 2(a): $d\ge n^\varepsilon$.}
We have $|S\setminus T|\le cn\le n/[4(D+1)]$.
Moreover, every retained set $A$ satisfies
$|A\cap T|\ge n/2\ge|A\setminus T|$.
Apply Proposition~\ref{aq:localtime} to
$A=(A\cap T)\sqcup(A\setminus T)$, and
Proposition~\ref{aq:chains} to
$S=T\sqcup(S\setminus T)$ with $k=D$.
The required bounds hold with
\[
 q=6n^{-\varepsilon}+512\sqrt{2/n},\quad
 F=(D+1)\left(6n^{-\varepsilon}
                  +1024\sqrt{(D+1)/n}\right)^D.
\]
Again $q+nF=o(1)$, so Proposition~\ref{prop:repair-full}
gives a valid ordering of $S$.

\smallskip\noindent\emph{Case 2(b):
$d<n^\varepsilon$ and $|H|>n^{1+\varepsilon}$.}
For each retained set $A$, put $B=A\cap T$.
Then $|B|\ge n/2$ and $|A\setminus B|\le cn\le|B|$.
For sufficiently large $n$, we have
$20D+1\le\tau|T|/2$.
For every coset $Q$ of a proper subgroup of $H$,
the balancedness of $T-a$ therefore gives
\[
 |(B-a)\setminus Q|
 \ge \tau|T|-(20D+1)
 \ge \frac{\tau}{2}|B|.
\]
Thus $B-a$ is $\tau/2$-balanced in $H$.
For sufficiently large $n$, we also have
$(\tau/2)|B|\ge2^{13}D^2\log(|B|+1)$.
Hence Proposition~\ref{prop:almost-coset-direct-chain}
applies to $A=B\sqcup(A\setminus B)$ with
$\alpha=\tau/2$ and $k=1,D$.

Using $|B|\le n$, $|H|>n^{1+\varepsilon}$,
$|B|\ge n/2$, and $\tau\ge c/(3n^\varepsilon)$,
the quantity inside the brackets in that proposition
is at most
\[
 2^{13}D^2\left(
       n^{-\varepsilon}
       +\sqrt{12/c}\,n^{-(1-\varepsilon)/2}
       +\frac{2\log(n+1)}n
       \right)
 \le M n^{-\varepsilon},
\]
where $M=2^{13}D^2(5+\sqrt{12/c})$.
Thus the required bounds hold with
$q=M n^{-\varepsilon}$ and $F=(M n^{-\varepsilon})^D$,
uniformly over all retained sets and targets.
Since $q+nF=o(1)$, Proposition~\ref{prop:repair-full}
again gives a valid ordering of $S$.

\smallskip\noindent\emph{Case 2(c):
$d<n^\varepsilon$ and $|H|\le n^{1+\varepsilon}$.}
Put $K=H+\langle a\rangle$.
Then $T\subseteq K$,
$|K|=d|H|\le n^{1+2\varepsilon}$, and
$|S\cap K|\ge(1-c)n$.
Consequently,
$\delta:=|S\cap K|/|K|\ge(1-c)n^{-2\varepsilon}$.
Also,
$|S\setminus K|/|S\cap K|\le c/(1-c)<2^{-30}$, while
\[
 |S\cap K|\delta^{28500}
 \ge (1-c)^{28501}n^{1-57000\varepsilon}
 = (1-c)^{28501}n^{0.43}.
\]
Thus hypotheses of
Proposition~\ref{prop:localized-dense-core} hold for
sufficiently large $n$, giving a valid ordering of $S$.
\end{proof}

\appendix

\section{Roos' permanent inequality}
\label{app:Roos}

We derive the random-permutation estimate used in the proof of Lemma~\ref{lem:multivariate-partition} from Roos' permanent inequality~\cite{Roos}.

\begin{lemma}\label{thm:Roos}
Let $(b_{ir})_{i,r=1}^n$ be an array in $G$,
where $n\ge2$, and let $\sigma$ be uniform over permutations of $[n]$.
For every additive character $\xi:G\to\mathbb R/\mathbb Z$,
\[
 \Bigl|\mathbb E_\sigma e^{2\pi i\xi(\sum_i b_{i\sigma(i)})}\Bigr|
 \le\exp\Bigl(-\frac1{n^3}\sum_{i,j,r,s=1}^n
       \|\xi(b_{ir}-b_{jr}-b_{is}+b_{js})\|^2\Bigr).
\]
\end{lemma}

\begin{proof}
Choose real lifts $x_{ir}$ of $\xi(b_{ir})$, and put
$Z=(e^{2\pi i x_{ir}})_{i,r=1}^n$.
The expectation in the statement equals $(\operatorname{per}Z)/n!$.
By \cite[Theorem~1.1]{Roos},
\begin{align*}
 \frac{|\operatorname{per}Z|}{n!}
 &\le
 \Bigl(
 \frac1{n^2(n-1)^2}
\sum_{i\ne j, \, r\ne s}
 \cos^2\bigl(\pi(x_{ir}-x_{jr}-x_{is}+x_{js})\bigr)
 \Bigr)^{(n-1)/4}\\
 &\le
 \exp\Bigl(
 -\frac1{n^2(n-1)}
 \sum_{i\ne j,\,r\ne s}
 \|x_{ir}-x_{jr}-x_{is}+x_{js}\|^2
 \Bigr).
\end{align*}
The second inequality uses $\cos^2(\pi x)\le1-4\|x\|^2$
and $1-t\le e^{-t}$.
Since the $x_{ir}$ are real lifts of $\xi(b_{ir})$,
the norms in the last sum equal those in the statement.
The terms with $i=j$ or $r=s$ vanish, and
$n^2(n-1)\le n^3$ gives the desired bound.
\end{proof}

\section{Auxiliary absorption proofs}
\label{app:absorption}

We provide the proofs of Lemma~\ref{ext:wave} and
Proposition~\ref{prop:repaired-dense-source}, using the notation
and definitions of Section~\ref{sec:external-absorption}.

\subsection{Embedding waveforms}
\label{app:waveforms}

For completeness, we give the argument of
\cite[Lemma~9.8]{BBKMM} in our notation.

\begin{proof}[Proof of Lemma~\ref{ext:wave}]
Put $F=\bigcup_i\operatorname{col}(\mathcal P_i)$ and
$B_i=\{0,g\}\cup\{a:(a,b)\in\mathcal P_i\}$.
The hypotheses give $|F|\le80m\le|E|/2^{11}$ and $|B_i|\le42$.
The listed elements of $B_i$ are distinct, since $g\ne0$,
the entries of every pair are nonzero, and the pairs are disjoint.

Fix $v\in G$ and $1\le i\le m$.
For each $e\in E\setminus F$, call the connector $v\to v+e$
together with $\Theta(v+e,\mathcal P_i)$ the
\emph{candidate indexed by $e$}.
Its theta graph $\Theta(v+e,\mathcal P_i)$ has vertex set
\[
 v+e+B_i
 =\{v+e,v+e+g\}\cup
   \{v+e+a:(a,b)\in\mathcal P_i\}.
\]
For $e,f\in E\setminus F$, the sets $v+e+B_i$ and
$v+f+B_i$ intersect only if $f-e\in B_i-B_i$.
Thus, for each $e$, at most $|B_i-B_i|\le |B_i|^2 \le 1764$
choices of $f$ give an intersection, including $f=e$.
We can greedily select at least
$|E\setminus F|/1764
 \ge (1-2^{-11})|E|/1764
 \ge |E|/2048$
candidates whose theta graphs are pairwise vertex-disjoint.
For every $(v,i)$, fix such a collection before sampling $R$.

A candidate \emph{survives} if all vertices of its theta graph
belong to $R$.
Within each fixed collection, these events are independent,
each with probability $p^{|B_i|}\ge p^{42}$.
The number $X$ of surviving candidates therefore has mean
$\mu\ge |E|p^{42}/2048\ge128m$.
Since $50m<\mu/2$, Lemma~\ref{lem:Chernoff} gives
\[
 \Prb(X<50m)
 \le e^{-\mu/8}
 \le e^{-|E|p^{42}/2^{14}}
 \le |G|^{-16}.
\]
The size hypothesis also gives $m<|G|$.
A union bound over the $|G|m$ choices of $(v,i)$ shows that,
with probability at least $1-|G|^{-14}$,
every fixed collection has at least $50m$ surviving candidates.

Fix such an outcome and any starting vertex $u\in G$.
Construct the waveform inductively.
After adding the first $k<m$ theta graphs, let $v$ be the
current endpoint, taking $v=u$ when $k=0$.
The construction so far has at most $1+42k$ vertices
and $k$ connector colours.
Consider the surviving candidates in the fixed collection
for $(v,k+1)$.
Their theta graphs are pairwise vertex-disjoint, so
each previously used vertex, including $v$, excludes
at most one candidate.
Each earlier connector colour excludes at most one more,
since different candidates have different connector colours $e$.
Thus at most $1+43k<50m$ candidates are excluded.

Choose a remaining candidate indexed by $e$.
Add the connector $v\to v+e$, whose colour is $e$,
and the theta graph $\Theta(v+e,\mathcal P_{k+1})$
to the current waveform.
This theta graph lies in $R$ and avoids all previously
used vertices, while $e\in E\setminus F$ differs from
every earlier connector colour.
Replace the current endpoint $v$ by $v+e+g$ for the next step.
Repeating this step constructs the required waveform,
with every vertex other than $u$ in $R$.
\end{proof}

\subsection{Sets of density at least \texorpdfstring{$3/4$}{3/4}}
\label{app:dense}

We prove Proposition~\ref{prop:repaired-dense-source}
by following the absorption construction of
\cite[Section~9]{BBKMM}, using
Proposition~\ref{prop:corrected-pair-absorber} to handle
distinct inverse colours. We first recall the theorem
of M\"uyesser--Pokrovskiy~\cite{MP25} that covers the
very dense case.

\begin{theorem}[Valid orderings for very dense sets; {\cite{MP25}; see \cite[Theorem~7.1]{BBKMM}}]
\label{ext:verydense}
For every finite abelian group $G$ of sufficiently large order,
every $S\subseteq G\setminus\{0\}$ satisfying
$|S|\ge|G|-|G|^{1-2^{-20}}$ has a valid ordering.
\end{theorem}

\begin{proof}[Proof of Proposition~\ref{prop:repaired-dense-source}]
Put $h=|H|$ and $\gamma=2^{-20}$, and assume
that $h$ is sufficiently large.
Theorem~\ref{ext:verydense} handles $|S|\ge h-h^{1-\gamma}$.
In the remaining case, put
$p=\tfrac12h^{-2\gamma}$, $q=1-2p$,
$\mu=2^{-80}h^{-90\gamma}$,
$t=\lfloor2^{-66}h^{1-84\gamma}\rfloor$, and
$\ell=\lfloor tp^3/512\rfloor$.
The definitions give $t\ge40$,
$tp^3\ge2^{-70}h^{1-90\gamma}$, and
$\ell\ge2^{-79}h^{1-90\gamma}$.
Consequently, $1\le\ell\le t$,
$tp^3\ge\max\{2^{16}\log h,512\ell\}$, and $\mu qh\le\ell$.

Choose a $1/8$-random subset $E\subseteq S$.
Independently of $E$, form a partition
$H=V_R\sqcup V_M\sqcup V_T$ by assigning each vertex
to the three parts with probabilities $p,q,p$,
independently for different vertices.
By Observation~\ref{obs:dense-cayley-cut},
$\operatorname{Cay}_H(S)$ has no $1/4$-sparse cut.
Apply Lemma~\ref{ext:mainpath} in $H$ with
$M=V_M$, $D=E$, $q'=1/8$, and $\zeta=1/4$.
Its lower parameter bounds hold since
$90\gamma<1/9500$ and $|S|/h\ge3/4$.
Moreover, $\mu\le h^{-\gamma}$ gives
\[
 (1+\mu)|S|/h
 \le(1+\mu)(1-h^{-\gamma})
 \le1-h^{-2\gamma}
 =q,
\]
and $1/8\le1-\mu q/4$.
Thus the lemma's conclusion holds simultaneously for
every $S_F\subseteq E$ and every pair of distinct endpoints
in $H$, except on an event of probability at most $7/h$.

By Lemma~\ref{lem:Chernoff},
$\Prb(|E|<h/16)\le e^{-h/192}$.
For each $E$ with $|E|\ge h/16$, make the following
choices using only $E$.
Among the $\binom{|E|}{2}$ unordered pairs from $E$,
at most $|E|/2$ have sum zero.
By averaging, some $g\in H\setminus\{0\}$ is the sum of
at least $|E|(|E|-2)/(2h)\ge h/3072$ such pairs.
These pairs are disjoint, since each element determines
its partner; orient them arbitrarily to obtain $g$-pairs.
As $3t+\ell\le4t\le2^{-64}h$, choose a subfamily
$\mathcal P$ of exactly $3t+\ell$ such pairs.
Corollary~\ref{ext:flex} supplies
$\mathcal P_{\rm flex}\subseteq\mathcal P$ of size $t+\ell$
and a $(\mathcal P_{\rm flex},\ell)$-absorbing incidence family
$\mathcal P_1,\ldots,\mathcal P_{3t}$ with
$\bigcup_{i=1}^{3t}\mathcal P_i=\mathcal P$.

Conditional on these choices, $V_R$ and $V_T$ are each
$p$-random subsets of $H$.
Lemma~\ref{ext:wave} applies to $E$ and
$\mathcal P_1,\ldots,\mathcal P_{3t}$, since
\[
 \frac{|E|p^{42}}{2^{18}}
 \ge2^{-64}h^{1-84\gamma}
 \ge\max\{3t,\log h\}.
\]
Its conclusion fails with probability at most $1/h$.
Proposition~\ref{prop:corrected-pair-absorber}, applied
in $H$ to $\mathcal P_{\rm flex}$ and $V_T$, also fails
with probability at most $1/h$.
Finally, $\Prb(|V_M|<2)\le(1+h)(1-q)^{h-1}$.
The vertex parts are dependent, but each has the required
marginal distribution, so a union bound shows that the
probability of any failure is at most
\[
 e^{-h/192}+\frac9h+(1+h)(1-q)^{h-1}<1.
\]
Fix an outcome for which all these conclusions hold. We now assemble the paths
(see Figure~\ref{fig:rainbow-assembly}).

\begin{figure}[htbp]
\centering
\begingroup
\definecolor{Ma}{HTML}{A32727}
\definecolor{Mb}{HTML}{DC5A2D}
\definecolor{Mc}{HTML}{EF9932}
\definecolor{Md}{HTML}{D6A900}
\definecolor{Me}{HTML}{C73969}
\definecolor{Mf}{HTML}{8E2346}
\definecolor{Cfirst}{HTML}{226D80}
\definecolor{Pa}{HTML}{2063BF}
\definecolor{Pb}{HTML}{48A4D7}
\definecolor{Csecond}{HTML}{258B76}
\definecolor{Pc}{HTML}{6951B5}
\definecolor{Pd}{HTML}{9B73C9}
\definecolor{Cthird}{HTML}{344B89}
\definecolor{Clast}{HTML}{337EAA}
\definecolor{Pe}{HTML}{1E9B9F}
\definecolor{Pf}{HTML}{3D745B}
\definecolor{Ta}{HTML}{B85C1D}
\definecolor{Tb}{HTML}{4B2E85}
\definecolor{Tc}{HTML}{74A998}
\definecolor{Td}{HTML}{D96B93}
\definecolor{Te}{HTML}{4385DA}
\definecolor{Tf}{HTML}{624A9C}
\definecolor{Tg}{HTML}{9A6500}
\definecolor{unused}{HTML}{A6ABB2}
\definecolor{muted}{HTML}{6A7078}
\definecolor{Mwash}{HTML}{FCF5F0}
\definecolor{Rwash}{HTML}{F1F6FB}
\definecolor{Twash}{HTML}{F8F6FA}
\tikzset{
  chosen/.style={line width=1.55pt,-{Stealth[length=1.6mm,width=1.1mm]},shorten <=1.2pt,shorten >=1.2pt},
  unusededge/.style={draw=unused,line width=.8pt,dash pattern=on 3pt off 2.4pt,-{Stealth[length=1.65mm,width=1.2mm]},shorten <=1.5pt,shorten >=1.5pt},
  dot/.style={circle,fill=black,inner sep=0pt,minimum size=3.5pt},
  ghostdot/.style={circle,fill=unused,inner sep=0pt,minimum size=3.2pt},
  joint/.style={circle,fill=black,inner sep=0pt,minimum size=4.7pt},
}

\definecolor{Mg}{HTML}{CD4229}
\definecolor{Mh}{HTML}{E77824}
\definecolor{Mi}{HTML}{C79C3A}
\definecolor{Mj}{HTML}{BE7A20}
\definecolor{Mk}{HTML}{D54158}
\definecolor{Ml}{HTML}{9F357B}

\resizebox{.94\linewidth}{!}{%
\begin{tikzpicture}[x=.90cm,y=1.15cm,font=\small]
\path[use as bounding box] (-0.0571,-1.3600) rectangle (15.4830,2.6800);
\fill[Mwash,rounded corners=3pt] (0.0000,-0.8200) rectangle (6.3226,0.8600);
\fill[Rwash,rounded corners=3pt] (6.3626,-0.8200) rectangle (11.0640,0.8600);
\fill[Twash,rounded corners=3pt] (11.1040,-0.8200) rectangle (15.3780,0.8600);
\node[font=\normalsize] at (3.1613,2.3500) {Main path $P_M$};
\node[font=\normalsize] at (8.7133,2.3500) {Waveform and $P_R$};
\node[font=\normalsize] at (13.2357,2.3500) {Absorber path $P_T$};
\draw[Ma,line width=2.6pt,line cap=round] (1.7613,1.9100)--(2.0613,1.9100);
\draw[Mb,line width=2.6pt,line cap=round] (2.2613,1.9100)--(2.5613,1.9100);
\draw[Mc,line width=2.6pt,line cap=round] (2.7613,1.9100)--(3.0613,1.9100);
\draw[Md,line width=2.6pt,line cap=round] (3.2613,1.9100)--(3.5613,1.9100);
\draw[Me,line width=2.6pt,line cap=round] (3.7613,1.9100)--(4.0613,1.9100);
\draw[Mf,line width=2.6pt,line cap=round] (4.2613,1.9100)--(4.5613,1.9100);
\draw[Cfirst,line width=2.6pt,line cap=round] (7.8933,1.9100)--(8.1533,1.9100);
\draw[Pa,line width=2.6pt,line cap=round] (8.3533,1.9100)--(8.6133,1.9100);
\draw[Pc,line width=2.6pt,line cap=round] (8.8133,1.9100)--(9.0733,1.9100);
\draw[Pe,line width=2.6pt,line cap=round] (9.2733,1.9100)--(9.5333,1.9100);
\node[font=\normalsize] at (3.1613,1.4300) {$(S\setminus S_F)\setminus L$};
\node[font=\normalsize] at (8.7133,1.4300) {$S_F\setminus\operatorname{col}(\mathcal P')$};
\draw[Ta,line width=2.6pt,line cap=round] (11.5656,1.9100)--(11.7956,1.9100);
\draw[Td,line width=2.6pt,line cap=round] (11.9656,1.9100)--(12.1956,1.9100);
\draw[Tg,line width=2.6pt,line cap=round] (12.3656,1.9100)--(12.5956,1.9100);
\draw[Tb,line width=2.6pt,line cap=round] (13.6449,1.9100)--(13.8749,1.9100);
\draw[Te,line width=2.6pt,line cap=round] (14.0449,1.9100)--(14.2749,1.9100);
\draw[Tf,line width=2.6pt,line cap=round] (14.4449,1.9100)--(14.6749,1.9100);
\node[font=\normalsize] at (12.0806,1.4300) {$L$};
\node[font=\normalsize] at (13.1307,1.4300) {$\cup$};
\node[font=\normalsize] at (14.1599,1.4300) {$\operatorname{col}(\mathcal P')$};
\node[font=\footnotesize,text=muted] at (13.2357,1.0000) {at most one colour omitted};
\coordinate (w) at (0.1255,0.0000);
\coordinate (m1) at (0.5820,0.0450);
\coordinate (m2) at (1.0385,-0.0350);
\coordinate (m3) at (1.4951,0.0450);
\coordinate (m4) at (1.9516,-0.0350);
\coordinate (m5) at (2.4081,0.0450);
\coordinate (ml) at (2.8646,0.0000);
\coordinate (mr) at (3.5208,0.0000);
\coordinate (m6) at (3.9773,0.0450);
\coordinate (m7) at (4.4338,-0.0350);
\coordinate (m8) at (4.8903,0.0450);
\coordinate (m9) at (5.3468,-0.0350);
\coordinate (m10) at (5.8033,0.0450);
\coordinate (v) at (6.2598,0.0000);
\draw[chosen,Ma] (w)--(m1);
\draw[chosen,Mg] (m1)--(m2);
\draw[chosen,Mb] (m2)--(m3);
\draw[chosen,Mh] (m3)--(m4);
\draw[chosen,Mc] (m4)--(m5);
\draw[chosen,Md] (m5)--(ml);
\draw[chosen,Mi] (mr)--(m6);
\draw[chosen,Mj] (m6)--(m7);
\draw[chosen,Mk] (m7)--(m8);
\draw[chosen,Me] (m8)--(m9);
\draw[chosen,Mf] (m9)--(m10);
\draw[chosen,Ml] (m10)--(v);
\node[font=\large] at (3.1927,0.0150) {$\cdots$};
\definecolor{wfHalo}{HTML}{E1C6FA}
\coordinate (L1) at (6.6395,0);
\coordinate (R1) at (8.1000,0);
\coordinate (T1) at (7.3698,.60);
\coordinate (C1) at (7.3698,0);
\coordinate (B1) at (7.3698,-.60);
\coordinate (gapl) at (8.48,0);
\coordinate (gapr) at (9.08,0);
\coordinate (L3) at (9.5000,0);
\coordinate (T3) at (10.2309,.60);
\coordinate (C3a) at (10.2309,.20);
\coordinate (C3b) at (10.2309,-.20);
\coordinate (B3) at (10.2309,-.60);
\coordinate (z) at (10.9618,0);
\draw[draw=wfHalo,line width=5.3pt,line join=round,line cap=round]
 (v)--(L1)--(T1)--(R1)--(gapl);
\draw[draw=wfHalo,line width=5.3pt,line join=round,line cap=round]
 (gapr)--(L3)--(B3)--(z);
\foreach \l/\m/\r/\ca/\cb in {
 L1/C1/R1/Pc/Pd,L1/B1/R1/Tb/Te,
 L3/T3/z/Pa/Pb,L3/C3a/z/Pc/Pd,L3/C3b/z/Tb/Te}{
 \draw[unusededge,draw=\ca] (\l)--(\m);
 \draw[unusededge,draw=\cb] (\m)--(\r);
}
\draw[chosen,Cfirst] (v)--(L1);
\draw[chosen,Pa] (L1)--(T1);
\draw[chosen,Pb] (T1)--(R1);
\draw[chosen,Csecond] (R1)--(gapl);
\draw[chosen,Clast] (gapr)--(L3);
\draw[chosen,Pe] (L3)--(B3);
\draw[chosen,Pf] (B3)--(z);
\node[font=\large] at (8.78,.015) {$\cdots$};
\coordinate (t1) at (11.5451,0.0450);
\coordinate (t2) at (12.1436,-0.0350);
\coordinate (t3) at (12.7422,0.0450);
\coordinate (t4) at (13.3408,-0.0350);
\coordinate (t5) at (13.9393,0.0450);
\coordinate (t6) at (14.5379,-0.0350);
\coordinate (end) at (15.1785,0.0000);
\draw[chosen,Ta] (z)--(t1);
\draw[chosen,Tb] (t1)--(t2);
\draw[chosen,Tc] (t2)--(t3);
\draw[chosen,Td] (t3)--(t4);
\draw[chosen,Te] (t4)--(t5);
\draw[chosen,Tf] (t5)--(t6);
\draw[chosen,Tg] (t6)--(end);
\node[joint] at (w) {};
\node[dot] at (m1) {};
\node[dot] at (m2) {};
\node[dot] at (m3) {};
\node[dot] at (m4) {};
\node[dot] at (m5) {};
\node[dot] at (ml) {};
\node[dot] at (mr) {};
\node[dot] at (m6) {};
\node[dot] at (m7) {};
\node[dot] at (m8) {};
\node[dot] at (m9) {};
\node[dot] at (m10) {};
\node[joint] at (v) {};
\foreach \v in {L1,T1,R1,gapl,gapr,L3,B3}{\node[dot] at (\v) {};}
\foreach \v in {C1,B1,T3,C3a,C3b}{\node[ghostdot] at (\v) {};}
\node[joint] at (z) {};
\node[dot] at (t1) {};
\node[dot] at (t2) {};
\node[dot] at (t3) {};
\node[dot] at (t4) {};
\node[dot] at (t5) {};
\node[dot] at (t6) {};
\node[joint] at (end) {};
\node[below=4pt,font=\normalsize] at (w) {$w$};
\node[below=4pt,font=\normalsize] at (v) {$v$};
\node[below=4pt,font=\normalsize] at (z) {$z$};
\node[above=2.5pt,font=\footnotesize] at (T1) {$p_1$};
\node[below=1.5pt,font=\footnotesize] at (B3) {$p_{3t}$};
\node[font=\large] at (3.1613,-1.1300) {$V_M$};
\node[font=\large] at (8.7133,-1.1300) {$V_R$};
\node[font=\large] at (13.2357,-1.1300) {$V_T$};
\end{tikzpicture}%
}
\endgroup
\caption{The waveform is fixed before $P_M$ is chosen.
Once the unused set $L$ is known, the absorber selects
$\ell$ flexible pairs $\mathcal P'$.
The absorbing property then allows the waveform to be
collapsed to $P_R$ avoiding these pairs, without changing
its endpoints $v,z$.
Any colour of $S$ omitted by
$P_M\cup P_R \cup P_T$ is placed first in the final ordering.}
\label{fig:rainbow-assembly}
\end{figure}
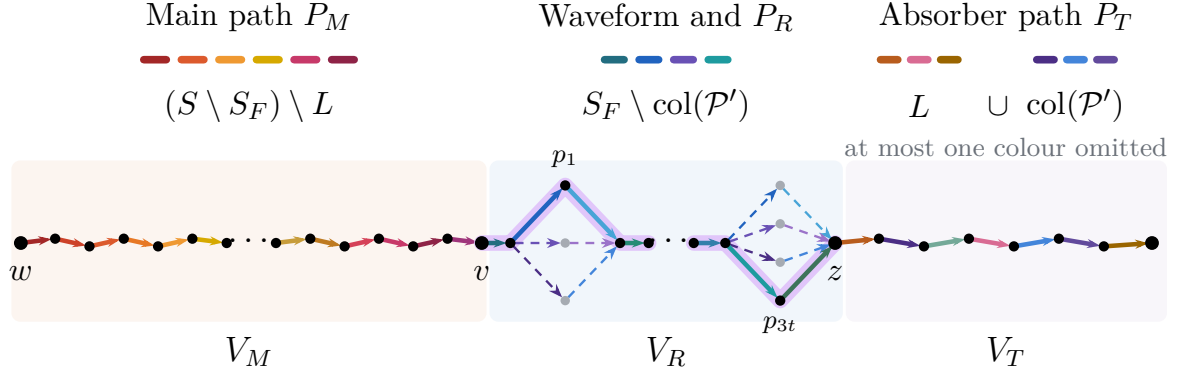

Choose distinct $w,v\in V_M$.
By Lemma~\ref{ext:wave}, choose a waveform starting at $v$,
formed from $\mathcal P_1,\ldots,\mathcal P_{3t}$,
with all other vertices in $V_R$.
Let $z\in V_R$ be its endpoint.
Let $S_F\subseteq E$ be the set of all edge colours
of the waveform.
These consist of $\operatorname{col}(\mathcal P)$ and
the $3t$ connector colours, so $|S_F|=9t+2\ell$.

Lemma~\ref{ext:mainpath} gives a rainbow path $P_M$
from $w$ to $v$ using only colours in $S\setminus S_F$,
with all its vertices in $V_M$.
Let $L$ be the set of colours in $S\setminus S_F$
not used by $P_M$.
Then $|L|\le\mu qh\le\ell$, and
$L\cap\operatorname{col}(\mathcal P)=\varnothing$.
The conclusion of Proposition~\ref{prop:corrected-pair-absorber}
therefore supplies a subfamily
$\mathcal P'\subseteq\mathcal P_{\rm flex}$ with
$|\mathcal P'|=\ell$ and a rainbow path $P_T$ starting at $z$,
with all other vertices in $V_T$, using only colours in
$L\cup\operatorname{col}(\mathcal P')$ and omitting
at most one colour of this set.

By the absorbing property, choose pairwise distinct
pairs $p_i\in\mathcal P_i\setminus\mathcal P'$ for $1\le i\le3t$.
Retaining the corresponding two-edge path in each theta graph
and all connector edges collapses the waveform to a
rainbow path $P_R$ from $v$ to $z$.
The selected pairs are precisely those in
$\mathcal P\setminus\mathcal P'$, so $P_R$ uses exactly
the colours in $S_F\setminus\operatorname{col}(\mathcal P')$.

The colour sets of $P_M,P_R,P_T$ are therefore disjoint
and together contain all but at most one element of $S$.
All vertices of $P_M$ lie in $V_M$, all vertices of $P_R$
other than $v$ lie in $V_R$, and all vertices of $P_T$
other than $z$ lie in $V_T$.
Thus joining these paths in the order $P_M,P_R,P_T$
gives a rainbow path using all but at most one colour of $S$.
The observation following Definition~\ref{def:cayley}
then gives a valid ordering of $S$.
\end{proof}

\end{document}